\documentclass[11pt]{amsart}
\usepackage{tensor}
\usepackage[english]{babel}

\usepackage[
pdftex,hyperfootnotes]{hyperref}
\hypersetup{
 colorlinks=true,
 linkcolor=NavyBlue, 
 urlcolor=RoyalPurple,
 citecolor=OliveGreen,
 pdftitle={Hypergeometric multiple orthogonal polynomials and random walks},
 bookmarks=true,
 pdfpagemode=FullScreen,
}
\usepackage[usenames,dvipsnames,svgnames,table,x11names]{xcolor}
\usepackage{pgfplots}
\usepackage{tikz}
\usepackage{tikz-3dplot}
\usetikzlibrary{automata,quotes, chains,matrix,calc,shadows,shapes.callouts,shapes.geometric,shapes.misc,positioning,patterns,decorations.shapes,
decorations.pathmorphing,decorations.markings,decorations.fractals,decorations.pathreplacing,shadings,fadings,arrows.meta,bending}

\usepackage{nicematrix,booktabs}
\usepackage[utf8]{inputenc}
\usepackage{comment}
\usepackage[textwidth=18cm,textheight=24cm,
height=
25cm,
width=18.5cm]{geometry}

\usepackage{rotating,multirow,pdflscape}
\usepackage{amssymb,latexsym,amsmath,amsthm,bm}
\usepackage{mathrsfs}
\usepackage{mathtools,arydshln,mathdots}
\usepackage{bigints}
\makeatletter
\renewcommand*\env@matrix[1][*\c@MaxMatrixCols c]{%
\hskip -\NiceArraycolsep
\let\@ifnextchar\new@ifnextchar
\NiceArray{#1}}
\makeatother

\mathtoolsset{centercolon}
\usepackage[table]{xcolor}
\usepackage[toc,page]{appendix}

\usepackage{tocvsec2}

\usepackage{Baskervaldx}
\usepackage{newtxmath }
\newtheorem{coro}{Corollary}

\newtheorem{teo}{Theorem}

\newtheorem{pro}{Proposition}
\newtheorem{lemma}{Lemma}

\newtheorem{rem}{Remark}

\newcommand{\ii}{\operatorname{i}}

\renewcommand{\d}{\operatorname{d}}
\newcommand{\Exp}[1]{\operatorname{e}^{#1}}

\renewcommand{\Re}{\operatorname{Re}}
\newcommand{\diag}{\operatorname{diag}}

\newcommand{\C}{\mathbb{C}}

\newcommand{\N}{\mathbb{N}}
\newcommand{\R}{\mathbb{R}}
\newcommand{\Z}{\mathbb{Z}}

\newcommand{\p}{\displaystyle}

\newcommand{\1}{\text{\fontseries{bx}\selectfont \textup 1}}

\makeatletter
\def\@settitle{\begin{center}%
\baselineskip14\p@\relax
\bfseries
\uppercasenonmath\@title
\@title
\ifx\@subtitle\@empty\else
\\begin{align}1ex]\uppercasenonmath\@subtitle
\footnotesize\mdseries\@subtitle
\fi
\end{center}%
}
\def\subtitle#1{\gdef\@subtitle{#1}}
\def\@subtitle{}
\makeatother

\usepackage{longtable}
\usepackage{enumerate}
\allowdisplaybreaks

\newcommand{\rchi}{ X}

\title[Hypergeometric Multiple Orthogonal Polynomials and Random Walks]{Hypergeometric Multiple Orthogonal Polynomials \\ and Random Walks}

\author[A Branquinho]{Amílcar Branquinho$^1$}
\address{$^1$Departamento de Matemática,
Universidade de Coimbra, 3001-454 Coimbra, Portugal}
\email{$^1$ajplb@mat.uc.pt}

\author[JE Fernández-Díaz]{Juan E. Fernández-Díaz$^2$}
\email{$^2$juanen01@ucm.es}

\author[A Foulquié]{Ana Foulquié-Moreno$^3$}
\address{$^3$Departamento de Matemática, Universidade de Aveiro, 3810-193 Aveiro, Portugal}
\email{$^3$foulquie@ua.pt}

\author[M Mañas]{Manuel Mañas$^4$}
\address{$^4$Departamento de Física Teórica, Universidad Complutense de Madrid, Plaza Ciencias 1, 28040-Madrid, Spain \&
Instituto de Ciencias Matematicas (ICMAT), Campus de Cantoblanco UAM, 28049-Madrid, Spain}
\email{$^4$manuel.manas@ucm.es}

\thanks{$^1$Acknowledges Centro de Matemática da Universidade de Coimbra (CMUC) -- UID/MAT/00324/2019, funded by the Portuguese Government through FCT/MEC and co-funded by the European Regional Development Fund through the Partnership Agreement PT2020}

\thanks{$^3$Acknowledges CIDMA Center for Research and Development in Mathematics and Applications (University of Aveiro) and the Portuguese Foundation for Science and Technology (FCT) within project UIDB/MAT/UID/04106/2020 and UIDP/MAT/04106/2020}

\thanks{$^4$Thanks financial support from the Spanish ``Agencia Estatal de Investigación'' research project [PGC2018-096504-B-C33], \emph{Ortogonalidad y Aproximación: Teoría y Aplicaciones en Física Matemática}.}



\begin{document}

\maketitle

\begin{abstract}
The recently found hypergeometric multiple orthogonal polynomials on the step-line by Lima and Loureiro are shown to be
random walk polynomials. It is proven that the corresponding Jacobi matrix and its transpose, which are nonnegative matrices and describe higher recurrence relations, can be normalized to two stochastic matrices, dual to each other. Using the Christoffel--Darboux formula on the step-line and the Poincaré theory for non-homogeneous recurrence relations it is proven that both stochastic matrices are related by transposition in the large~$n$~limit. 
These random walks are beyond birth and death, as they describe a chain in where transitions to the two previous states are allowed, or in the dual to the two next states.
The corresponding Karlin--McGregor representation formula is given for these new Markov chains. The regions of hypergeometric parameters where the Markov chains are recurrent or transient are given. 
Stochastic factorizations, in terms of pure birth and of pure death factors, for the corresponding Markov matrices of types I and~II, are provided.
Twelve uniform Jacobi matrices and the corresponding random walks, related to a Jacobi matrix of Toeplitz type, and theirs stochastic or semi-stochastic matrices (with sinks and sources), that describe Markov chains beyond birth and death, are found and studied. One of these uniform stochastic cases, which is a recurrent random walk, is the only hypergeometric multiple random walk having a uniform stochastic factorization. The corresponding weights, Jacobi and Markov transition matrices and sequences of type~II multiple orthogonal polynomials are provided. Chain of Christoffel transformations connecting the stochastic uniform tuples between them, and the semi-stochastic uniform tuples, between them, are presented. As byproduct, summations formulas at unity and three and four terms contiguous relations for the generalized hypergeometric function $\tensor[_3]{F}{_2}$ are found. Moreover, using the uniform recurrence relation and generating functions explicit expressions are found for the type I multiple orthogonal polynomials. 
Finally, the recent Karp--Prilepkina summation formulas, extension of the Karlsson--Minton formulas, are applied to find the corresponding summations 
for the generalized moments of the type II multiple orthogonal polynomials related to the remainder in the corresponding type II Hermite--Padé approximation problem.
\end{abstract}

\thispagestyle{empty}


\tableofcontents

\section{Introduction}

In this paper we continue our investigations on random walk orthogonal polynomials of multiple type. In~our previous paper \cite{bfmaf} we found that the ideas of Karlin and McGregor \cite{KmcG}, see also \cite{KmcG1957-1,KmcG1957-2}, can be extended from standard orthogonality to the multiple orthogonality scenario, whenever the Jacobi matrix is nonnegative and bounded. Now, instead of birth and death Markov chains, in which the non zero transition probabilities could only happen for near neighbors, we have that transitions to the $N$-th previous states are permitted. The dual Markov chain has the $N$-th next states reachable in one transition only. This allowed, as in \cite{KmcG} to give a representation formula for the iterated transition probabilities in terms of integrals of the multiple orthogonal polynomials of type II and the corresponding linear forms of type I.
Also provided possible stationary states and gave a characterization of the recurrent or transient character of the random walk in terms of the divergence or convergence of a certain integral.

In \cite{bfmaf} we presented the Jacobi--Piñeiro multiple orthogonal polynomials as a case study, and several properties were given. In particular, we showed the region where the Jacobi--Piñeiro random walks were recurrent or transient in terms of the Jacobi--Piñeiro parameters.
The large $n$ limit of the corresponding Jacobi matrix leads to a particular interesting case with a Toeplitz matrix describing the transitions. At the time we finished the writing of \cite{bfmaf} a new paper appeared \cite{lima_loureiro}, in where new multiple orthogonal polynomials based on Gauss hypergeometric function weights were given. The corresponding Jacobi matrix is bounded and nonnegative. Therefore, we immediately realized that, as for the Jacobi--Piñeiro case, it was a case amenable to support random walks beyond birth and death. 

In this paper we perform such a study and show that the Lima and Loureiro hypergeometric multiple orthogonal polynomials \cite{lima_loureiro} are indeed random walk polynomials.
For that aim we explicitly compute the value at unity of the type I linear forms.
Then, whenever the zeros of the type II multiple orthogonal polynomials are confined in the interior of the support of the weight system, 
the type II multiple orthogonal polynomials at unity are also positive and our method in \cite{bfmaf} is applicable. Notice that this confinement of zeros happens for example for algebraic Chebyshev (AT) systems of weights, in particular for Nikishin systems as the hypergeometric system of weights is in some region of parameters, see~\cite{lima_loureiro}.

Let us stress that the parameters determining the system of weights in \cite[cf. equation~(1.2)]{lima_loureiro}, 
although very useful to unify their results, in some instances are too restrictive for our purposes.
In fact, sometimes is too strong and not so convenient when it excludes from the discussion cases such as the Piñeiro weights, or the set of parameters leading to a uniform stochastic matrix. Both cases, are \emph{perfect systems }of weights.
Therefore, in this paper we will relax such conditions when required.
For a general introduction on multiple orthogonal polynomials and on Markov chains we refer the reader to our paper \cite{bfmaf} and references therein.

Now, we describe the contents and results of this paper. Within this introduction we remind the reader the main aspects of multiple orthogonal polynomials relevant for our objectives. Then we give a brief description of the results 
in~\cite{lima_loureiro}.
We finish the introduction by recalling some of our results in \cite{bfmaf}. 

Section \ref{S:ratio_asymptotics} is devoted to ratio asymptotics.
In Theorem \ref{teo:linear_forms_at_unity} we use the Rodrigues type formula of \cite{lima_loureiro} to get the type I linear forms at unity explicitly.
This is an important finding, allowing explicit expressions for the type I stochastic matrices. Then, we consider the ratio asymptotics at unity of these type~I linear forms and some further properties of these linear forms. Next, we consider the ratio asymptotics at unity of the type~II multiple orthogonal polynomials. For that aim, we use the Christoffel--Darboux formula on the step-line and the Poincaré theorem for homogeneous linear recurrences relations, see Proposition \ref{pro:ratio_asymptotics_typeII}. We end this section with Theorem~\ref{teo:ratio asymptotics_linear forms} in where the ratio asymptotics of the type I linear forms is analyzed, 
using the Osgood theorem, in compact subsets \cite{Beardon_Minda, Osgood}.

Then, in \S\ref{S:randowm_walks} we give, cf. 
Theorem \ref{teo:JPII_stochastic},
the type I and II stochastic matrices for the hypergeometric multiple orthogonal polynomials. We also show that both matrices are connected for large $n$, as they respective limits are transposed to each other. Clearly, the Karlin--McGregor representation formulas, see Theorems \ref{teo:KMcG} and \ref{teo:KMcG2}, apply here. We also discuss in Theorem \ref{Theorem:recurrent-transient} whether these hypergeometric Markov chains are recurrent or transient and also find left eigenvectors of the stochastic matrices, close to steady states, in Proposition \ref{pro:steady}. 
Then, in Theorem~\ref{teo:stochastic _factorization} we give the stochastic factorization of the hypergeometric stochastic matrices in terms of three very simple pure birth or death matrices (with only one superdiagonal or subdiagonal, respectively). This is an important result as we have an $LU$ factorization of an stochastic matrix in terms of stochastic matrices, leading to the corresponding interpretation in terms of compositions of experiments, and possibly to urn models \cite{grunbaum_de la iglesia,grunbaum_de la iglesia2}. 

In \S\ref{S:uniform} we seek for uniform or almost uniform hypergeometric Jacobi matrices, i.e the Jacobi matrix is a banded Toeplitz matrix but for the first column. In particular, we search for Jacobi matrices that are variations (only in the first column) of the large $n$-limit of the Jacobi matrix for any set of hypergeometric parameters, the banded Toeplitz matrix in \eqref{eq:Jacobi_Jacobi_Piñeiro_uniform}, that appeared in \cite{Coussement_Coussment_VanAssche} and \cite{bfmaf}. In Theorem \ref{teo:12}, we find that there are twelve hypergeometric tuples leading to such cases. These hypergeometric \emph{uniform tuples }are organized in six couples. Each couple lies in the same \emph{gauge class} described in \S\ref{S:gauge_freedom}, and lead to the same sequences of multiple orthogonal polynomials of type II, linear forms and Jacobi matrix. In \S\ref{S:stochastic_uniform}, we analyze three couples of uniform tuples which have named as \emph{stochastic uniform tuples}, describing recurrent Markov chains with stochastic transition matrices, having recurrent dual Markov chains with sinks and sources and semi-stochastic transition matrices. In Theorem \ref{teo:uniform_stochastic_factorization} we find that one of these stochastic uniform tuples has for its Markov matrix an stochastic factorization with three uniform pure death and pure birth factors. Then, in \S\ref{S:semi-stochastic uniform tuples} we discuss
the other three couples of uniform tuples, that we call \emph{semi-stochastic uniform tuples} because the corresponding transient Markov chains have sink states.
One of these six semi-stochastic uniform tuples is the one considered in \cite{lima_loureiro}. For the twelve cases we give the corresponding set of weights, all of them perfect. The three couples of semi-stochastic uniform tuples correspond to Nikishin systems. In Theorem \ref{teo:Christoffel chains for uniform tuples} we connect the stochastic uniform tuples using permuting Christoffel transformations \cite{bfm}, this mixed with the \emph{gauge} transformations connect all the tuples in this set. The same is done for the semi-stochastic uniform tuples. We also discuss some connections trough basic Christoffel transformations \cite{bfm}. As byproduct we get summation formulas at unity and three and four terms contiguous relations for the generalized hypergeometric function $\tensor[_3]{F}{_2}$. To conclude the section,  using   uniform recurrence relations, we construct a  generating function  that leads to explicit expressions  for the type I multiple orthogonal polynomials, see Theorem \ref{teo:uniform_type_I}. 

To conclude the paper, in \S\ref{S:Summations} we consider the very recent summation formulas by Karp and Prilepkina \cite{Karp_Prilepkina} to get summation formulas relevant in the theory of simultaneous Hermite--Padé approximants. In particular, we give in Proposition~\ref{pro:summations_Padé} explicit summations for the generalized type II moments linked to the remainders in the interpolation conditions of the type II Hermite--Padé approximants to the Markov functions. 


\subsection{Two component multiple orthogonal polynomials on the step-line}

We now present a brief introduction to multiple orthogonal polynomials form the Gauss--Borel factorization point of view \cite{afm}, see also~\cite{abv,Coussement_Coussment_VanAssche,Ismail}.

\subsubsection{Gauss--Borel factorization of the moment matrix}
Let us consider a couple of weights $(w_1,w_2)$, the semi-infinite vectors of monomials
\begin{align*}
\rchi := \left( \begin{NiceMatrix}
 1 & x & x^2 & \Cdots 
 \end{NiceMatrix} \right)^\top,
 &&
\rchi_1 :=\left( \begin{NiceMatrix}
 1 & 0 & x & 0 & x^2 & 
 \Cdots 
 \end{NiceMatrix} \right)^\top, 
 && 
\rchi_2 := \left( \begin{NiceMatrix}
 0 & 1 & 0 & x & 0 & x^2 &
 \Cdots
 \end{NiceMatrix} \right)^\top,
\end{align*}
and the following vector of undressed linear forms
\begin{align}
 \xi &:= \rchi_1 w_1 + \rchi_2 w_2
 =\left( \begin{NiceMatrix}
 w_1 & w_2 & xw_1 & xw_2 & x^2w_1 & x^2w_2 &
 \Cdots
 \end{NiceMatrix} \right)^\top. 
\end{align}
Given a measure $\mu$ with support on the closed interval $\Delta\subset\R$, our moment matrix is 
 \begin{align}
 \label{compact.g}
 g :=\int_\Delta\rchi(x)(\xi(x))^\top\d \mu(x).
 \end{align}

 The Gauss--Borel factorization of the moment matrix $g$ is the
 problem of finding the solution of
 \begin{align}\label{facto}
 g &=S^{-1} H \tilde S^{-\top}, 
 \end{align}
 with $S$, $\tilde S$ lower unitriangular semi-infinite matrices
 \begin{align*}
 S&=\left(\begin{NiceMatrix}[columns-width = auto]
 1 & 0 & \Cdots & \\
 S_{1,0 } & 1& \Ddots & \\
 S_{2,0} & S_{2,1} & \Ddots & \\
 \Vdots & \Ddots & \Ddots &
 \end{NiceMatrix}\right), & 
 \tilde S&=
 \left(\begin{NiceMatrix}[columns-width = auto]
 1 & 0 &\Cdots &\\
 \tilde S_{1,0 } & 1&\Ddots&\\
 \tilde S_{2,0} & \tilde S_{2,1} & \Ddots &\\
 \Vdots & \Ddots& \Ddots& 
 \end{NiceMatrix}\right), 
 \end{align*}
 and $H$ a semi-infinite diagonal matrix 
$ H=\diag
 \begin{pNiceMatrix}
 H_0 & H_1 & \Ldots 
 \end{pNiceMatrix} {}$,
 with $H_l\neq 0$, $l\in\N_0$. 
 
 Vector of type~II multiple orthogonal polynomials and of type~I linear forms associated are defined respectively by
 \begin{align*}
 B & :=
 \left( \begin{NiceMatrix}
 B^{(0)}\\
 B^{(1)}\\
 \Vdots
 \end{NiceMatrix} \right)
 = S \, \rchi, &
 A_1 & :=
 \left( \begin{NiceMatrix}
 A_1^{(0)} \\[.1cm]
 A_1^{(1)} \\
 \Vdots
 \end{NiceMatrix} \right) 
 = H^{-1} \, \tilde S \, \rchi_1, &&
 A_2 :=
 \left( \begin{NiceMatrix}
 A_2^{(0)}\\[.1cm]
 A_2^{(1)}\\
 \Vdots
 \end{NiceMatrix} \right) 
 = H^{-1} \, \tilde S \, \rchi_2, &&
 Q :=
 \left( \begin{NiceMatrix}
 Q^{(0)}\\[.1cm]
 Q^{(1)}\\
 \Vdots
 \end{NiceMatrix} \right)
 =H^{-1} \, \tilde S \, \xi .
 \end{align*}
 
\subsubsection{Multiple orthogonality and bi-orthogonality}

 Let us assume that the Gauss--Borel factorization exists, that~is, the system $(w_1,w_2,\d\mu)$ is perfect. Then, in terms of $\, S$ and $\tilde S$ we construct the type~II multiple orthogonal~polynomials
 \begin{align} \label{defmops}
 B^{(m)}&:=x^m+\sum_{i=0}^{m-1} S_{m,i}x^{i}, && m \in\N_0,
 \end{align}
as well as the type~I multiple orthogonal polynomials,
 \begin{align}\label{eq:A12} 
 A^{(2 m)}_1&=\frac{1}{H_{2m}}\Bigg(x^m+\sum_{i=0}^{m-1}\tilde S_{2m,2i} x^{i}\Bigg), &
 A^{(2 m +1)}_1&=\frac{1}{H_{2m+1}}\Bigg(\sum_{i=0}^{m}\tilde S_{2m+1,2i} x^{i}\Bigg), & m \in\N_0,
 \end{align}
 and\begin{align}\label{eq:A12'}
 \begin{aligned}
 A^{(0)}_2&=0, & A^{(1)}_2&=1, & \\
 A^{(2m)}_2&=\frac{1}{H_{2m}}\Bigg(\sum_{i=0}^{m-1}\tilde S_{2m,2i+1} x^{i}\Bigg), &A^{(2m+1)}_2&=\frac{1}{H_{2m+1}}\Bigg(x^m+\sum_{i=0}^{m - 1}\tilde S_{2m+1,2i+1} x^{i} \Bigg), & m \in\N.
 \end{aligned}
 \end{align}
 For $m \in\N_0$, the linear forms are
 \begin{align*}
 Q^{(m)}:=w_1A^{(m)}_1+w_2A_2^{(m)}.
 \end{align*}
Given the vector of indices $\vec \nu (2 m)=(m+1,m)$ and $\vec\nu(2m+1)=(m+1,m+1)$, $m \in \N_0$, and the corresponding multiple orthogonal polynomials
 \begin{gather*} 
 \begin{aligned}
 B^{ (2 m)}&=B_{(m,m)}, & B^{(2m+1)}&=B_{(m+1,m)},
 \end{aligned}\\
 \begin{aligned}
 A^{ (2 m)}_1&=A_{(m+1,m),1}&
 A^{(2m+1)}_1&=A_{(m+1,m+1),1}, &
 A^{ (2 m)}_2&=A_{(m+1,m),2}, &
 A^{(2m+1)}_2&=A_{(m+1,m+1),2},
 \end{aligned}
 \end{gather*}
the following type I orthogonality relations
 \begin{align*}
 \int_{\Delta} x^{j} (A_{\vec \nu, 1}(x)w_{1} (x)+A_{\vec \nu, 2}(x)w_{2} (x))\d\mu (x) =0, && 
 \deg A_{\vec \nu, 1}\leq\nu_{1}-1, &&
 \deg A_{\vec \nu, 2}\leq\nu_{2}-1,
 \end{align*}
 for $j\in\{0,\ldots, |\vec \nu|-2\}$,
 and type~II orthogonality relations
 \begin{align*}
 \int_{\Delta} B_{\vec\nu}(x) w_{a} (x) x^{j} \d\mu (x) =0, && \deg B_{\vec\nu} &\leq|\vec \nu|, && j=0,\ldots, \nu_a - 1, && a=1,2,
 \end{align*}
 are fulfilled.

 The following multiple biorthogonality relations 
 \begin{align}\label{biotrhoganility}
 \int_\Delta B^{(m)}(x) Q^{(k)}(x)\d \mu(x)&=\delta_{m,k},& m,k \in \N_0,
 \end{align}
 hold.

\subsubsection{The Jacobi matrix and the fourth order homogeneous linear recurrence relations}

For describing the linear recurrence relations we require the following shift matrix 
$ \Lambda
 =\left( \begin{NiceMatrix}[small]
 0& 1 & 0 & \Cdots & &&\\
 \Vdots &\Ddots & \Ddots&\Ddots&& &\\
 & & &&&&\\
 & & & &&&\\
 &&&&&&\\
 &&&&&&
 \end{NiceMatrix} \right) $,
that satisfies $\Lambda\chi(x)=x\chi(x)$. The Jacobi matrix 
$ { J } := S \Lambda S^{-1}$
 is a banded semi-infinite matrix
\begin{align}\label{eq:Jacobi}
 { J } 
 = \left(\begin{NiceMatrix}
 \beta_0& 1 & 0 & \Cdots & & \\[-4pt]
 \alpha_1&\beta_1& 1 & \Ddots& & \\ 
 \gamma_1& \alpha_2& \beta_2& 1 &&\\ 
 0& \gamma_2&\alpha_3 & \beta_3& 1 & \\ 
 \Vdots&\Ddots&\Ddots&\Ddots&\Ddots&
 \end{NiceMatrix}\right)
\end{align}
The Jacobi matrix, type~II multiple orthogonal polynomials, and the corresponding type~I multiple orthogonal polynomials and linear forms of type~I fulfill the eigenvalue property
\begin{align}\label{eq:eigen_value}
 { J } \, B&= x \, B, & { J } ^\top A_1&= x \, A_1, & { J } ^\top A_2&= x \, A_2,& { J } ^\top Q&= x \, Q. 
\end{align}
This eigenvalue property component-wise gives the following fourth order homogeneous linear recurrence relations:
\begin{align*}
\gamma_{n-1} B^{(n-2)} + \alpha_{n} B^{(n-1)}+\beta_{n} B^{(n)}+B^{(n+1)}&=xB^{(n)},\\
A_1^{(n-1)} + \beta_{n}A_1^{(n)}+\alpha_{n+1}A_1^{(n+1)}+\gamma_{n+1}A_1^{(n+2)} &= x A_1^{(n)},\\
A_2^{(n-1)}+\beta_{n}A_2^{(n)}+\alpha_{n+1}A_2^{(n+1)}+\gamma_{n+1}A_2^{(n+2)} &= x A_2^{(n)},\\
Q^{(n-1)}+\beta_{n}Q^{(n)}+\alpha_{n+1}Q^{(n+1)}+\gamma_{n+1}Q^{(n+2)}&=xQ^{(n)} ,
\end{align*}
for $n \in \N_0$, with the restriction of objects with negative indices be treated as zero.

\subsubsection{Gauge freedom}\label{S:gauge_freedom}


In \cite{bfm} 
we presented
\begin{teo}\label{teo:gauge_freedom_0}
 Let $\Delta \subset \R$ be the compact support of two perfect systems, $(w_1, w_2,\d\mu)$, with $\int_\Delta w_1(x) \d \mu(x) =1$, and $(\hat w_1,\hat w_2,\d\mu)$, with $\int_\Delta \hat{w}_1 (x) \d \mu(x) =1$, that have the same sequence of type~II multiple orthogonal polynomials $\{B^{(n)}(x)\}_{n=0}^\infty$. 
 Then $ \hat w_1 = w_1$ and there exists $\alpha,\beta \in\R$ with $\alpha,\beta \not =0$ such that
 \begin{align}\label{eq:gauge_1_0}
 \gamma w_1+ \alpha w_2+ \beta \hat w_2 =0. 
 \end{align}
 If $A^{(m)}_1$, $A^{(m)}_2$ are the type I multiple orthogonal polynomials, associated with the system
 $(w_1, w_2,\d\mu)$, then the type~I multiple orthogonal polynomials, $\hat{A}^{(m)}_1,\hat{A}^{(m)}_2$
 associated with the system $(w_1,\hat w_2,\d\mu)$ are given by
 \begin{align*} 
 \hat{A}^{(m)}_1=A^{(m)}_1- \frac{\gamma}{\alpha} A_2^{(m)}, && \hat{A}^{(m)}_2= -\frac{\beta}{\alpha} A_2^{(m)} ,
 \end{align*} and both systems has the same type I linear forms , i.e., 
$ \hat{Q}^{(m)} = Q^{(m)}$.
This will be used in \S\ref{S:uniform}.
\end{teo}

\begin{rem}
	This theorem leads to a surprising fact. For multiple orthogonal polynomials the Jacobi matrix $J$, the sequence of type~II multiple orthogonal polynomials $\{B^{(n)}\}_{n=0}^\infty$, and even the sequence of type I linear forms $\{Q^{(n)}\}_{n=0}^\infty$ do not determine uniquely the spectral system $(w_1,w_2,\d\mu)$, as it determines uniquely $w_1$, for the second weight one has the freedom just described. Inspired by similar phenomena in gauge field theories, we call this 
	a \emph{gauge freedom}.
\end{rem}

\begin{rem}\label{rem:scaling}
	If in \eqref{eq:gauge_1_0} we put $\gamma=0$ and $\beta=-1$ we get $\hat w_1=w_1$ and $\hat w_2=\alpha w_2$, that is we are dealing with a rescaling of the second weight. In this case we get $\hat B^{(n)}=B^{(n)}$, $\hat Q^{(n)}=Q^{(n)}$, $\hat A_1^{(n)}=A_1^{(n)}$ and 
	$\hat A_2^{(n)}=\alpha^{-1}\beta A_2^{(n)}$. This is a particular case of a rescaling $\hat w_1=\alpha_1 w_1$ and $\hat w_2=\alpha_2 w_2$, and using the Gauss--Borel factorization it can be shown that $\hat B^{(n)}=B^{(n)}$, $\hat Q^{(n)}=Q^{(n)}$, $\hat A_1^{(n)}=\alpha_1^{-1} A_1^{(n)}$ and 
	$\hat A_2^{(n)}=\alpha_2^{-1} A_2^{(n)}$.
\end{rem}

\begin{rem}\label{teo:gauge_freedom}
A version of this result that we will use in \S\ref{S:uniform} goes as follows. Assume a compact support $\Delta$, a solvable Hausdorff moment problem for $w_2$, and that two perfect systems $(w_1, w_2,\d\mu)$, with $\int_\Delta w_1(x) \d \mu(x) =1$ and $\int_\Delta w_2(x) \d \mu(x) =1$, and $(w_1,\hat w_2,\d\mu)$ with
$\int_\Delta \hat w_2(x) \d \mu(x) =1$
have the same sequence of type II multiple orthogonal polynomials $\{B^{(n)}(x)\}_{n=0}^\infty$. Then, there exists $\alpha,\beta\in\R$ such that
\begin{align}\label{eq:gauge_1}
 w_1&=\alpha w_2+\beta\hat w_2, & \alpha+\beta&=1.
\end{align}
 Notice that \eqref{eq:gauge_1} can be written as
\begin{align}\label{eq:gauge_2}
	\alpha' w_1+\beta'\hat w_2&= w_2, & \alpha'+\beta'&=1.
\end{align}
Indeed, from \eqref{eq:gauge_2} we get $ w_1= \alpha w_2+\beta\hat w_2$ with $\alpha=\frac{1}{\alpha'}$ and $\beta=1-\alpha=-\frac{\beta'}{\alpha'}$.
 To analyze the \emph{gauge freedom} in \eqref{eq:gauge_2}, after the replacement $\alpha'\to \alpha$ and $\beta'\to \beta$, 
 requires to consider solutions of 
 \begin{align}\label{eq:rho}
 \alpha \rho_{1,0} + \beta \hat \rho_{2,0} &= \rho_{2,0} , &
 \alpha \rho_{1,1} + \beta \hat \rho_{2,1} &= \rho_{2,1} ,
 \end{align}
 where $\rho_{a,n}=\int_{\Delta} x^nw_a(x)\d\mu (x)$ with $a\in\{1,2\}$ and $\hat \rho_{2,n}=\int_{\Delta} x^n\hat w_2(x)\d\mu (x)$.
\end{rem}

\begin{rem}
 This \emph{gauge freedom} will be used in \S\ref{S:uniform}. In that section couple of perfect systems $(W_1,W_2,\d x)$ and $(W_1,\hat W_2,\d x)$ leading to the same sequences multiple orthogonal polynomials of type II and Jacobi matrices are presented for pair of tuples of hypergeometric parameters $(a,b,c,d)$ and $(b,a,c,d)$, see \S\ref{S:LL} below.
\end{rem}

\subsection{Christoffel transformations}

For the aim of this paper two Christoffel transformations presented in \cite{bfm} 
will be useful in the developments in \S\ref{S:uniform}. See also \cite{aagmm, matrix,matrix2}.
First, let us consider $\vec w=(w_1,w_2)$ and the transformed vector of weights
$\vec{\underline w}
=(w_2, x \, w_1)$, that is a simple Christoffel transformation of $w_1$ followed by a permutation of the two weights. Then, \cite[Theorem 4]{bfm} 
says
\begin{teo}[Permuting Christoffel formulas]\label{teo:Christoffel}
 For the type I orthogonal polynomials and linear forms we have,
 \begin{gather*}
 Q_{\underline{\vec w}}^{(n)}(x)=Q_{{\vec w}}^{(n)}(x)-\frac{A_{1,\vec w}^{(n)}(0)}{A_{1,\vec w}^{(n+1)}(0)}Q_{{\vec w}}^{(n+1)}(x),\\
 \begin{aligned}
 A_{1,\underline{\vec w}}^{(n)}(x)&=A_{2,{\vec w}}^{(n)}(x)-\frac{A_{1,\vec w}^{(n)}(0)}{A_{1,\vec w}^{(n+1)}(0)}A_{2,{\vec w}}^{(n+1)}(x),&
 A_{2,\underline{\vec w}}^{(n)}(x)&=\frac{1}{x}\Big(A_{1,{\vec w}}^{(n)}(x)-\frac{A_{1,\vec w}^{(n)}(0)}{A_{1,\vec w}^{(n+1)}(0)}A_{1,{\vec w}}^{(n+1)}(x)\Big).
 \end{aligned}
 \end{gather*}
 For the type II orthogonal polynomials we have 
 \begin{align}\label{eq:permuting_Christoffel}
B^{(n)}_{\underline{\vec w}}(x)
=\frac{1}{x}\bigg(B_{\vec w}^{(n+1)}(x)
+\bigg(\frac{A_{1,\vec w}^{(n-1)}(0)}{A^{(n)}_{1,\vec w}(0)} +J_{n,n}\bigg)B_{\vec w}^{(n)}(x)-\frac{A_{1,\vec w}^{(n+1)}(0)}{A^{(n)}_{1,\vec w}(0)}J_{n+1,n-1}B_{\vec w}^{(n-1)}(x) \bigg). 
\end{align}
 For the $H$'s we find the transformation formula
$ H_{\underline{\vec w},n}=-\frac{A_{1,\vec w}^{(n+1)}(0)}{A_{1,\vec w}^{(n)}(0)} H_{{\vec w},n}$.
\end{teo}
We also will be faced with the basic Christoffel transformation
$ \underline{\vec w}:=x\vec w$,
and \cite[Theorem 5]{bfm} 
states that
\begin{teo}[Basic Christoffel formulas] \label{teo:Basic Christoffel formulas}
 We have the following relations
 \begin{align*}
 B^{(n)}_{\underline{\vec w}}(x)&=\frac{1}{x}
 \frac{\begin{vNiceMatrix}
 B_{\vec w}^{(n)}(0)& B_{\vec w}^{(n)}(x)\\[3pt]
 B_{\vec w}^{(n+1)}(0)& B_{\vec w}^{(n+1)}(x) 
 \end{vNiceMatrix}}{B_{\vec w}^{(n)}(0)}, &
 Q^{(n)}_{\underline{\vec w}}(x)&=\frac{1}{x}
 \frac{\begin{vNiceMatrix}
 A^{(n)}_{\vec w,1} (0)& A^{(n)}_{\vec w,2} (0) & Q^{(n)}_{{\vec w}}(x)\\[3pt]
 A^{(n+1)}_{\vec w,1} (0)& A^{(n+1)}_{\vec w,2} (0)& Q^{(n+1)}_{{\vec w}}(x)\\[3pt]
 A^{(n+2)}_{\vec w,1} (0)& A^{(n+2)}_{\vec w,2} (0)& Q^{(n+2)}_{{\vec w}}(x)
 \end{vNiceMatrix}}{\begin{vNiceMatrix}
 A^{(n+1)}_{\vec w,1} (0)& A^{(n+1)}_{\vec w,2} (0)\\
 A^{(n+2)}_{\vec w,1} (0)& A^{(n+2)}_{\vec w,2} (0)
 \end{vNiceMatrix}}.
 \end{align*}
\end{teo}
Hence, we have Christoffel formulas expressing the new multiple orthogonal polynomials of type II and linear forms of type I, in terms of the 
original ones.
\subsection{The hypergeometric multiple orthogonal polynomials}\label{S:LL}

Here we explain the main results of Lima and Loureiro in \cite{lima_loureiro} regarding hypergeometric multiple orthogonal polynomials that are required in this paper.

Given $\{a_k\}_{k=1}^p\subset \C, \{b_k\}_{k=1}^q\subset \C\setminus (-\N_0)$, the generalized hypergeometric is defined by the following power series
\begin{align*}
\tensor[_p]{F}{_{q}} \left[
\begin{NiceMatrix}[small]a_1, &\Ldots && &, a_p \\&b_1, &\Ldots&, b_q&\end{NiceMatrix}
;x	\right]
=
\sum_{k=0}^\infty \frac{(a_1)_k\cdots\cdots(a_p)_k}{(b_1)_k\cdots(b_{q})_k}\frac{x^k}{k!},
\end{align*}
where the Pochhammer symbol is $(a)_n:=a(a+1)\cdots(a+n-1)$, and $(a)_0=1$, the series converge absolutely when $p\leq q$, and for $|x|<1$ when $p=q+1$, that will be the cases in this paper for $q=,1,2,3$.
See \cite{Andrews} for details about these functions and in particular Theorems 2.1.1 and 2.1.2. Se also \cite{Bailey,Rainville0}.

\subsubsection{The system of weights}

The weights are constructed in terms of the Gauss' hypergeometric function
$\tensor[_2]{F}{_1}$
as follows. 
Given $(a,b,c,d)\in\R$, that we call \emph{hypergeometric tuple}, and $\delta:=c+d-a-b$ let us define the~function
\begin{align*}
\mathscr w (x,a,b;c,d) =\frac{\Gamma(c)\Gamma(d)}{\Gamma(a)\Gamma(b)\Gamma(\delta)}x^{a-1}(1-x)^{\delta-1}\;\tensor[_2]{F}{_1}\hspace*{-3pt}\left[{\begin{NiceArray}{c}[small]c-b,d-b \\\delta\end{NiceArray}};1-x\right].
\end{align*}
For the hypergeometric multiple orthogonality we consider $(W_1, W_2,\d x)$, where the weights are given by
\begin{align}\label{eq:pesosLL_w}
 W_1(x)&:=\mathscr w(x,a,b;c,d), &W_2(x)&:=\mathscr w(x,a,b+1;c+1,d), 
\end{align}
and $\d x$ is the Lebesgue measure in $[0,1]$. 

In \cite{lima_loureiro} the hypergeometric parameters are constrained by 
\begin{align}\notag
 a,b,c,d&\geq 0,\\
 \min(c,d)&>\max (a,b),\label{eq:min-max}
\end{align}
so that $\delta>0$. In \cite[Theorem 2.1]{lima_loureiro} it was shown, for these constrained hypergeometric parameters,
 that $\{\tilde w_1,\tilde w_2,\d x\}$ is a Nikishin system on $(0,1)$, therefore an~AT system and consequently a perfect system.
As we will see the min-max condition, $ \min(c,d)>\max (a,b)$, is not always required for our purposes.

For example, whenever
\begin{align}\label{eq:parameters_moments}
 a,b,\delta>0
\end{align}
Equation 11 in \cite[\S2.21.1]{Prudnikov} gives for the moments the following expression
\begin{align}\label{eq:moments}
 \int_0^1 x^{a+n-1}(1-x)^{\delta-1} \;\tensor[_2]{F}{_1}\hspace*{-3pt}\left[{\begin{NiceArray}{c}[small]c-b,d-b \\\delta\end{NiceArray}};1-x\right]\d x=\frac{(a)_n(b)_n}{(c)_n(d)_n}, && n \in \N_0 .
\end{align}
This formula is very useful
in doing computations to perform the Gauss--Borel factorization.
As we see the min-max is too demanding for this result.
Moreover, this formula is used in \cite{lima_loureiro} to prove the multiple orthogonal relations, and also the expression for the $H$'s.
Orthogonality and bi-orthogonality directly give, for all $n \in \N_0$,
 \begin{align}
\label{eq:H_2n}
 H_{2n}&=\int_0^1B^{(2n)}(x)x^n w_1(x)\d\mu(x), &
 H_{2n+1}&=\int_0^1B^{(2n+1)}(x)x^n w_2(x)\d\mu(x).
 \end{align}
Using \eqref{eq:moments}, in \cite{lima_loureiro} the expressions \eqref{eq:H_2n} 
were 
given as generalized hypergeometric functions ${}_4F_3$ at unity. Then, the 
Karlsson--Minton summation method \cite{minton,karlsson,miller,Karp_Prilepkina} was used in \cite[Equation (3.5)]{lima_loureiro} to evaluate the $H$'s and check their positivity. For complex numbers $\{\beta,f_1,\ldots,f_p\}\subset \C$ 
and positive integers $\{m_1,\ldots,m_p\}\subset \N$, 
with $n\geq m_1+\cdots+m_p$, we have for the value at unity of the hypergeometric functions the following Minton summation
\begin{align*}
 {}_{p+2}F_{p+1}\hspace*{-3pt}\left[\!\!\!\begin{array}{c}
 	\begin{NiceMatrix}[small]-n, &\beta,&f_1+m_1&\Ldots &, f_p+m_p \end{NiceMatrix}\\
 	\begin{NiceMatrix}[small]\beta+1,&f_1, &\Ldots&, f_p&\end{NiceMatrix}
 \end{array}\!\!\!
;1\right]=\frac{n!(f_1-\beta)_{m_1}\cdots(f_p-\beta)_{m_p}}{(\beta+1)_n(f_1)_{m_1}\cdots(f_p)_{m_p}}. 
\end{align*}
At this point we stress that although in \cite{lima_loureiro} it was required that $\Re\beta,\Re f_1,\ldots,\Re f_p>0$, that constraint is not really required, see \cite{minton,Karp_Prilepkina}. From Minton summation, formulas (3.7) and (3.8) in \cite{lima_loureiro} follow, and they~read
 \begin{align}\label{eq:H_pochhammer_1}
 \begin{cases}
\displaystyle
 H_{2n} =\frac{(2n)!(a)_{2n}(b)_{2n}(d-a)_{n}(d-b)_n}{(c)_{3n}(d)_{3n}(d+n-1)_{2n}}, \\[.25cm]
\displaystyle
 H_{2n+1}=\frac{(2n+1)!(a)_{2n+1}(b+1)_{2n}(c-a+1)_n(c-b)_{n+1}}{(c+1)_{3n+1}(c+n)_{2n+1}(d)_{3n+1}} ,
 \end{cases}
&& n \in \N_0 .
 \end{align}
In order to have a Gauss--Borel factorization of the moment matrix, that is equivalent to the perfectness of the system of weights, we must have $H_n\neq 0$, $n \in \N_0$ (positivity of the $H_n$ is sufficient but not necessary).
In this context,
to derive these formulas we assume \eqref{eq:parameters_moments}, 
with $d-a,d-b\not\in -\N_0$ to get $H_{2n}\neq 0$,
and to ensure $H_{2n+1}\neq 0$ we must have $c+1-a,c-b\not\in -\N_0$.
Therefore, perfectness is ensured whenever we have
\begin{align}\label{eq:region_parameters_pochhammer_perfect}
 a,b,\delta &>0, &
 d-a,d-b,c+1-a,c-b&\not\in -\N_0.
\end{align}

\subsubsection{Multiple orthogonal polynomials}

In \cite[Equations (3.2) and (3.1)]{lima_loureiro}, explicit expressions for the corresponding monic orthogonal polynomials of type II, were given for all $n \in \N_0$ by
\begin{align} \label{eq:polinomiohipergeometrico}
 B^{(n)}(x)&=\sum_{j=0}^n(-1)^j\binom{n}{j}\frac{(a+n-j)_j(b+n-j)_j}{\big(c+n+\big\lfloor\frac{n}{2}\big\rfloor-j\big)_j\big(d+n+\big\lfloor\frac{n-1}{2}\big\rfloor-j\big)_j}x^{n-j}
 \\ \notag
 &=
 (-1)^n
 \frac{(a)_n(b)_n}{\big(c+\big\lfloor\frac{n}{2}\big\rfloor\big)_n\big(d+\big\lfloor\frac{n-1}{2}
 \big\rfloor\big)_n }\,
 \tensor[_3]{F}{_2}\hspace*{-3pt}\left[{\begin{NiceArray}{c}[small]-n,\; c+\big\lfloor\frac{n}{2}\big\rfloor,\;d+\big\lfloor\frac{n-1}{2}\big\rfloor \\a,\;b\end{NiceArray}};x\right] . 
\end{align}
The second expression is given in terms of the generalized hypergeometric function $\tensor[_3]{F}{_2}$.
Therefore, for $n \in \N$,
\begin{align*}
 \int_0^1B^{(n)}(x) x^k W_1(x)\d x &=0, & k&\in\Big\{0,\ldots,\Big\lfloor\frac{n-1}{2}\Big\rfloor \Big\} ,&
 \int_0^1B^{(n)}(x)x^k W _2(x)\d x &=0, &k&\in\Big\{0,\ldots,\Big\lfloor\frac{n}{2}\Big\rfloor-1 \Big\} , 
 .
\end{align*}
The proof in \cite{lima_loureiro} of the orthogonality of these polynomials only requires \eqref{eq:moments} so, we need the parameters to satisfy \eqref{eq:region_parameters_pochhammer_perfect}.

The type I orthogonal polynomials $A^{(n)}_{1}(x)$ and $A_2^{(n)}(x)$ and the associated linear forms
\begin{align*}
 Q^{(n)}(x)= W_1(x)A^{(n)}_{1}(x)+ W_2(x)A^{(n)}_{2}(x), && n \in \N_0,
\end{align*}
that satisfies the biorthogonality
$ \int_0^1B^{(n)}(x) Q^{(m)}(x)\d x=\delta_{n , m}$,
were found in \cite[equation (2.16)]{lima_loureiro} to be determined by the following Rodrigues-type formula
\begin{align}\label{eq:LL_Rodrigues}
 Q^{(n)}(x)=\frac{(-1)^n}{n!}\frac{\d^n}{\d x^n}\mathscr{w}\Big(x;a+n,b+n;c+\Big\lfloor\frac{n+1}{2}\Big\rfloor+n,d+\Big\lfloor\frac{n}{2}\Big\rfloor+n\Big), && n \in \N_0 .
\end{align}
We have $ \deg A^{(n)}_1=\big\lfloor\frac{n}{2}\big\rfloor,$ and $\deg A^{(n)}_2=\big\lfloor\frac{n-1}{2}\big\rfloor$, with $A^{(0)}_2 = 0$,
and 
\begin{align*}
 \int_0^1x^k Q^{(n)}(x)\d x&=0, & k& \in\{0,\ldots,n-1 \}, & 
 \int_0^1x^{n} Q^{(n)}(x)\d x&=1.
\end{align*}
A convenient alternative pair of weights is
\begin{align} \label{eq:measures}
 w_1(x)&:=\;\tensor[_2]{F}{_1}\hspace*{-3pt}\left[{\begin{NiceArray}{c}[small]c-b,d-b \\\delta\end{NiceArray}};1-x\right], &
 w_2(x)&:=\frac{c}{b}\;\tensor[_2]{F}{_1}\hspace*{-3pt}\left[{\begin{NiceArray}{c}[small]c-b,d-b-1 \\\delta\end{NiceArray}};1-x\right],
\end{align}
where now we have
\begin{align}\label{eq:mu}
 \d\mu(x)&:= \frac{\Gamma(c)\Gamma(d)}{\Gamma(a)\Gamma(b)\Gamma(\delta)}x^{a-1}(1-x)^{\delta-1}\d x.
\end{align}
Defining
\begin{align} \label{eq:polinomio}
 q^{(n)}(x)= w_1(x)A^{(n)}_{1}(x)+ w_2(x)A^{(n)}_{2}(x)
\end{align}
we have 
\begin{align*}
 \int_0^1B^{(n)}(x) q^{(m)}(x)\d \mu(x)=\delta_{n,m} , && n,m \in \N_0 ,
\end{align*}
and 
\begin{align*}
 \int_0^1x^k q^{(n)}(x)\d \mu(x)&=0, & k&\in\{0, \ldots , n-1 \},&
 \int_0^1x^{n} q^{(n)}(x)\d \mu(x)&=1.
\end{align*}

\subsubsection{The hypergeometric Jacobi matrix}

In \cite[Theorem 3.5]{lima_loureiro} it is shown that the Jacobi matrix \eqref{eq:Jacobi} describing the higher recurrence relation of these hypergeometric multiple orthogonal polynomials is nonnegative.
Indeed, they show that for $n \in \N_0$,
\begin{align}\label{eq:jacobi_hyper_coeff}
 \beta_n =\lambda_{3n}+\lambda_{3n+1}+\lambda_{3n+2},
&&
 \alpha_{n+1}=(\lambda_{3n+1} 
 +\lambda_{3n+2})\lambda_{3n+3}+\lambda_{3n+2}\lambda_{3n+4},
&&
 \gamma_{n+1}=\lambda_{3n+2}\lambda_{3n+4}\lambda_{3n+6},
\end{align}
being the $\lambda$'s defined in terms of the sequence
\begin{align*}
 c_n=\begin{cases}
 c+k, & n=2k-1,\\
 d+k , & n=2k,
 \end{cases} &&
k \in \N ,
\end{align*}
as follows 
\begin{align}\label{eq:lambdas}
\begin{cases} 
\displaystyle
 \lambda_{3n} =\frac{n(b+n-1)(c_n-a-1)}{(c_n+n-2)(c_n+n-1)(c_{n-1}+n-1)}, \\[.25cm]
\displaystyle
 \lambda_{3n+1} =\frac{n(a+n)(c_{n-1}-b)}{(c_n+n-1)(c_{n-1}+n-1)(c_{n-1}+n)}, \\[.25cm]
\displaystyle
 \lambda_{3n+2}=\frac{(a+n)(b+n)(c_n-1)}{(c_n+n-1)(c_n+n)(c_{n-1}+n)} ,
\end{cases}
 &&
n \in \N_0 .
\end{align}
An important point regarding this paper is to find when these recursion coefficients are nonnegative. As was stated in \cite{lima_loureiro}, for $a,b,c,d>0$, if the max-min \eqref{eq:min-max} condition is fulfilled the $\lambda$'s are positive and Jacobi matrix is nonnegative.
However, this condition is again too restrictive. For example, the region
\begin{align}\label{eq:positivity_jacobi}
a,b&>0, & d&>\max(a,b), & c&\geq a, &c+1\geq b,
\end{align}
is such that \eqref{eq:region_parameters_pochhammer_perfect} holds, i.e. the system of weights is perfect, and the $\lambda$'s are nonnegative, and therefore the Jacobi matrix is nonnegative.

Moreover, if
$ \kappa=\frac{4}{27}$,
we have
$ \lim\limits_{n\to\infty}\lambda_n=\kappa$
and, consequently, 
\begin{align*}
 \lim_{n\to\infty}\beta_n&=3\kappa,&
 \lim_{n\to\infty}\alpha_n&=3\kappa^2,&
 \lim_{n\to\infty}\gamma_n&=\kappa^3.
\end{align*}
As we have seen, the coefficients of the Jacobi matrix associated with the hypergeometric multiple orthogonal polynomials share their limit with the ones of the Jacobi--Piñeiro case 
studied in \cite{bfmaf}.
In \cite{lima_loureiro} the authors comment that the Piñeiro reduction $\gamma=0$ of the Jacobi--Piñeiro multiple orthogonal polynomials ---with parameters $(\alpha,\beta,\gamma)$ with $\alpha,\beta,\gamma>-1$ and such that $\alpha-\beta\not\in\Z$--- corresponds to a particular limit case of these hypergeometric multiple orthogonal polynomials. In fact, for the particular choice of the hypergeometric parameters 
$c=a$, and $d=b+1$
one gets the Piñeiro weights~\cite{pineiro}
$w_1=b x^{b-1}$ and $w_2=ax^{a-1}$. These Piñeiro weights have a nonnegative Jacobi matrix whenever $|a-b|<1$ (cf. \cite{bfmaf}).
With this choice we have $\delta=1>0$ but we cannot meet the required min-max condition \eqref{eq:min-max}.
However, the relaxed conditions \eqref{eq:positivity_jacobi} are satisfied. Now, in terms of $a,b$ we have
\begin{align}\label{eq:positivity_piñeiro}
 b+1&>\max(a,b), & a&\geq a, &a+1\geq b,
\end{align}
if $a\leq b$ we only need to require $b\leq a+1$ and if $a\geq b$ we only need to check that $b+1>a$. These are precisely the conditions $|b-a|<1$. We stress that the Piñeiro case satisfies \eqref{eq:region_parameters_pochhammer_perfect}, thus the moments are quotients of Pochhammer as in \eqref{eq:moments} and the $H$'s do not cancel (cf. \cite[Corollary 3]{bfmaf}),
 \begin{align}\label{eq:H_pochhammer_pineiro}
 \begin{cases}
 H_{2n} =\dfrac{n! (2n)! (a)_{2n}(b)_{2n}(b+1-a)_{n}}{(a)_{3n}(b+1)_{3n}(b+n)_{2n}}, \\[.25cm]
 H_{2n+1}=\dfrac{n! (2n+1)! (a)_{2n+1}(b+1)_{2n}(a-b)_{n+1}}{(a+1)_{3n+1}(a+n)_{2n+1}(b+1)_{3n+1}},
 \end{cases}
 && n \in \N_0 .
\end{align}


Following \cite{lima_loureiro} we see that for the hypergeometric tuple
$\Big(\frac{4}{3},\frac{5}{3},2,\frac{5}{2}\Big)$
we have that
$ \beta_n=3\kappa$, $\alpha_{n+1}=3\kappa^2$ and $\gamma_n=\kappa^3$.
That is, the Jacobi matrix is uniform along its diagonals, a Toeplitz matrix
\begin{align}\label{eq:Jacobi_Jacobi_Piñeiro_uniform}
{ J } 
 = \left(\begin{NiceMatrix}
 3\kappa& 1 & 0 & \Cdots & & \\[-5pt]
 3\kappa^2&3\kappa& 1 & \Ddots& & \\ 
 \kappa^3& 3\kappa^2& 3\kappa& 1 &&\\ 
 0& \kappa^3&3\kappa^2& 3\kappa& 1 & \\
 \Vdots&\Ddots&\Ddots&\Ddots&\Ddots&\Ddots
 \end{NiceMatrix}\right).
\end{align}
We refer to this matrix as the \emph{asymptotic uniform Jacobi matrix}.

That is, the hypergeometric multiple orthogonal polynomials contain as a particular case the limit case, not as for the Jacobi--Piñeiro case, where no set of parameters reproduce such uniform Jacobi matrix. 
Notice that
$ \delta=2+\frac{5}{2}-\frac{4}{3}-\frac{5}{3}=\frac{3}{2}>0$
and
$ \min(c,d)=2>\frac{5}{3}=\max(a,b)$.

The perfect system $(W_1,W_2,\d x)$ is
\begin{align} \label{eq:pesosLL}
\left\{
\begin{aligned}
W_1(x)&=\frac{81\sqrt 3}{16\pi}\sqrt[3]{x}\left(\sqrt[3]{1+\sqrt{1-x}}-\sqrt[3]{1-\sqrt{1-x}}\right),
 \\
W_2(x)&=\frac{243\sqrt 3}{160\pi}\sqrt[3]{x}\left(\left(\sqrt[3]{1+\sqrt{1-x}}\,\right)^4-\left(\sqrt[3]{1-\sqrt{1-x}}\,\right)^4\right).
\end{aligned}
\right.
\end{align}

\subsection{Random walk multiple orthogonal polynomials}

We now describe the main results we derived in our previous paper \cite{bfmaf} extending the results of Karlin and McGregor to multiple orthogonal polynomials. 


We reproduce \cite[Theorems 1 \& 2]{bfmaf} in a convenient form for the aims of this paper.

\begin{teo}
 \label{pro:sigma_spectral}
 Let us assume that
 \begin{enumerate}
 \item The Jacobi matrix $ { J } $ is nonnegative.
 \item The values at unity of the polynomials and of the linear forms are positive, i.e. $B^{(n)}(1)$, $q^{(n)}(1)>0$.
 \end{enumerate}
 Then, the diagonal matrices
 \begin{align*}
 \sigma_{II}&=
 \diag
 \begin{pNiceMatrix}
 \sigma_{II,0},\sigma_{II,1},\ldots
 \end{pNiceMatrix},
 & \sigma_{II,n}&:=\frac{1}{B^{(n)}(1)},&
 \sigma_I&=
 \diag
 \begin{pNiceMatrix}
 \sigma_{I,0},\sigma_{I,1},\ldots
 \end{pNiceMatrix},
 & \sigma_{I,n}:=\frac{1}{q^{(n)}(1)} ,
\end{align*}
 are such that
 \begin{align*}
 P_{II}&:=\sigma_{II} { J } \, \sigma_{II} ^{-1}, & P_{II,n,m}&=\frac{B^{(m)}(1)}{B^{(n)}(1)} J_{n,m},&
 P_{I}&:=\sigma_I { J } ^\top\sigma_I^{-1}, & P_{I,n,m}&=\frac{q^{(m)}(1)}{q^{(n)}(1)}J_{m,n} ,
\end{align*}
are stochastic matrices.
\end{teo}

We have \cite[Theorems 6 \& 7]{bfmaf} that for the reader convenience we reproduce here

\begin{teo}[KMcG representation formula]\label{teo:KMcG}
 Let us assume the conditions requested in Theorem~\ref{pro:sigma_spectral}.
 Then, for random walks with Markov matrices $P_{II}$ and $P_I$, the transition probabilities, after~$r$ transitions from state~$n$ to state~$m$ are given respectively by
 \begin{align*}
 P_{II,nm}^r &=\frac{ B^{(m)}(1)}{B^{(n)}(1)} \int_\Delta x^rB^{(n)}(x) q^{(m)}(x)\d \mu(x),&
 P_{I,nm}^r&=\frac{q^{(m)}(1) }{ q^{(n)}(1)}
 \int_\Delta x^rB^{(m)}(x) q^{(n)}(x)\d \mu(x) .
 \end{align*}
\end{teo}

\begin{teo}\label{teo:KMcG2}
 Let us assume the conditions requested in Theorem~\ref{pro:sigma_spectral}. Then, for $|s|<1$, the transition probability generating functions reads~as
 \begin{align*}
 P_{II,nm}(s)&=\frac{B^{(m)}(1)}{B^{(n)}(1)}
 \bigintssss_\Delta \frac{B^{(n)}(x) q^{(m)}(x)}{1-sx}\d \mu(x),& P_{I,nm}(s)&=\frac{q^{(m)}(1)}{q^{(n)}(1)}
 \bigintssss_\Delta \frac{B^{(m)}(x) q^{(n)}(x)}{1-sx}\d \mu(x),
 \end{align*}
respectively.
The first passage generating functions are given by
\begin{align*}
 F_{II,nm}(s) =\frac{B^{(m)}(1)}{B^{(n)}(1)}
 \dfrac{\bigintss_\Delta \dfrac{B^{(n)}(x) q^{ (m)}(x)}{1-sx}\d \mu(x)}{
 \bigintss_\Delta \dfrac{B^{(m)}(x) q^{(m)}(x)}{1-sx}\d \mu(x)}, && 
 F_{I,nm}(s) =\frac{Q^{(m)}(1)}{q^{(n)}(1)}
 \dfrac{\bigintss_\Delta \dfrac{B^{(m)}(x) q^{(n)}(x)}{1-sx}\d \mu(x)}{
 \bigintss_\Delta \dfrac{B^{(m)}(x) q^{(m)}(x)}{1-sx}\d \mu(x)},
\end{align*}
for $n \neq m$, and for $n =m$ by 
\begin{align*}
F_{nn}(s)=1-\frac{1}{\bigintss_\Delta \dfrac{B^{(n)}(x) q^{(n)}(x)}{1-sx}\d \mu(x)} .
\end{align*}
The last expression $F_{nn}(s)$ is valid for both Markov chains of types I and II.
\end{teo}

\section{Ratio asymptotics}\label{S:ratio_asymptotics}

In this section we assume \eqref{eq:region_parameters_pochhammer_perfect}, so that the system of weights is perfect.

\subsection{On the linear forms at unity and Rodrigues type formula}

In \cite{lima_loureiro} multiple orthogonality for the measures $\tilde w_1\d x $ and $\tilde w_2(x)\d x$ supported on $[0,1]$ for the step-line were considered, i.e., one takes $\vec n=(1,1)$ (in the notation of \cite{afm,bfm,bfmaf}).
Here we use the weights $w_1$, $w_2$ defined on \eqref{eq:measures}.

\begin{teo}[Type I linear forms at unity]\label{teo:linear_forms_at_unity}
For $n\in\N_0$, the linear forms \eqref{eq:polinomio} at unity have the following values
\begin{align*}
 q^{(2n)}(1) = \frac{1}{(2n)!}\frac{(c)_{3n}(d)_{3n}}{(a)_{2n}(b)_{2n}},
&&
 q^{(2n+1)}(1) = \frac{1}{(2n+1)!}\frac{(c)_{3n+2}(d)_{3n+1}}{(a)_{2n+1}(b)_{2n+1}} .
 \end{align*}
\end{teo}
\begin{proof}
Recall that
\begin{align*}
 Q^{(n)}(x)=\frac{\Gamma(c)\Gamma(d)}{\Gamma(a)\Gamma(b)\Gamma(\delta)}x^{a-1}(1-x)^{\delta-1}q^{(n)}(x)
\end{align*}
so that the mentioned Rodrigues type formulas \eqref{eq:LL_Rodrigues} can be rewritten as
\begin{align*}
q^{(2n)}(x) &=
 \frac{1}{(2n)!}\frac{(c)_{3n}(d)_{3n}}{(a)_{2n}(b)_{2n}(\delta)_{2n}}
 \frac{1}{x^{a-1}(1-x)^{\delta-1}}\frac{\d^{2n}}{\d x^{2n}}\Big(x^{a-1+2n}(1-x)^{\delta-1+2n}\tensor[_2]{F}{_1}\hspace*{-3pt}\left[{\begin{NiceArray}{c}[small]c-b+n,d-b+n \\
 \delta+2n\end{NiceArray}};1-x\right]\Big),
 \\
q^{(2n+1)}(x) &=
 -\frac{1}{(2n+1)!}\frac{(c)_{3n+2}(d)_{3n+1}}{(a)_{2n+1}(b)_{2n+1}(\delta)_{2n+1}}
 \frac{1}{x^{a-1}(1-x)^{\delta-1}}\frac{\d^{2n+1}}{\d x^{2n+1}}\Big(x^{a+2n}(1-x)^{\delta+2n}\;\tensor[_2]{F}{_1}\hspace*{-3pt}\left[{\begin{NiceArray}{c}[small]c-
 b+n+1,d-b +n\\\delta+2n+1\end{NiceArray}};1-x\right]\Big).
\end{align*}
Let us expand the derivatives to get
\begin{subequations}\label{eq:Q}
	\begin{align}\label{eq:Qodd}
 & \begin{multlined}[t][.9\textwidth]
q^{(2n)}(x) =
 \frac{1}{(2n)!}\frac{(c)_{3n}(d)_{3n}}{(a)_{2n}(b)_{2n}(\delta)_{2n}}
 \\
 \times \sum_{k=0}^{2n}\sum_{l=0}^k
(-1)^{l+k}\binom{2n}{k}\binom{k}{l} (a+2n-k+l)_{k-l}(\delta+2n-l)_l x^{2n-k+l}(1-x)^{2n-l}
\tensor*[_2]{F}{_1^{(2n-k)}}\hspace*{-3pt}\left[{\begin{NiceArray}{c}[small]c-
 b+n,d-b+n \\\delta+2n\end{NiceArray}};1-x\right],
 \end{multlined}
 \\ \label{eq:Qeven}
& \begin{multlined}[t][.9\textwidth]
q^{(2n+1)}(x) =
 \frac{1}{(2n+1)!}\frac{(c)_{3n+2}(d)_{3n+1}}{(a)_{2n+1}(b)_{2n+1}(\delta)_{2n+1}}\\\times
 \sum_{k=0}^{2n+1}\sum_{l=0}^k
 (-1)^{l+k}\binom{2n+1}{k}\binom{k}{l} (a+2n+1-k+l)_{k-l}(\delta+2n+1-l)_l x^{2n+1-k+l}(1-x)^{2n+1-l}\\\times
 \tensor*[_2]{F}{_1^{(2n+1-k)}}\hspace*{-3pt}\left[{\begin{NiceArray}{c}[small]c-
 b+n+1,d-b+n \\\delta+2n+1\end{NiceArray}};1-x\right],
\end{multlined}
\end{align}
\end{subequations}
where $\tensor*[_2]{F}{_1^{(m)}}\hspace*{-3pt}\left[{\begin{NiceArray}{c}[small]\alpha,\; \beta \\\gamma \end{NiceArray}};1-x\right]$, stands for the $m$-th $x$-derivative of $\tensor[_2]{F}{_1}\hspace*{-3pt}\left[{\begin{NiceArray}{c}[small]\alpha,\; \beta \\\gamma \end{NiceArray}};x\right]$ when evaluated at~$1-x$.

Hence, when evaluating at unity we see that almost all terms cancel due to the presence of the factor $(1-x)^N$ for some natural number $N$. Then, in the evaluation at unity only survive the summands corresponding to $l=k=2n$ and $l=k=2n+1$, respectively.
In doing that one obtains for the factor with the double summations the values $(\delta)_{2n}$ and $(\delta)_{2n+1}$, respectively. Then, the announced result follows immediately.
\end{proof}

\begin{coro}\label{coro:positivity}
Whenever $a,b,c,d>0$, the linear forms when evaluated at unity are positive numbers
 \begin{align*}
 q^{(n)}(1)&>0, & n&\in\N_0.
 \end{align*}
\end{coro}

\begin{coro}\label{cor:expressions_type_I}
 The linear forms \eqref{eq:polinomio} can be expressed as 
 \begin{align*}
 & \begin{multlined}[t][0.95\textwidth] 
 q^{(2n)}(x) =
 \frac{1}{(2n)!}\frac{(c)_{3n}(d)_{3n}}{(a)_{2n}(b)_{2n}(\delta)_{2n}}
 \sum_{k=0}^{2n}\sum_{l=0}^k
 (-1)^{l+k}\binom{2n}{k}\binom{k}{l} (a+2n-k+l)_{k-l}(\delta+2n-l)_l x^{2n-k+l}(1-x)^{2n-l}\\
 \times \frac{(c - b+n)_{2n-k}(d-b+n)_{2n-k}}{(\delta+2n)_{2n-k}}\,\tensor[_2]{F}{_1}\hspace*{-3pt}\left[{\begin{NiceArray}{c}[small]c- b+3n-k , d-b+3n -k\\
 \delta+4n-k\end{NiceArray}};1-x\right],
 \end{multlined}
 \\
 & \begin{multlined}[t][0.95\textwidth] 
 q^{(2n+1)}(x) = 
 \frac{1}{(2n+1)!}\frac{(c)_{3n+2}(d)_{3n+1}}{(a)_{2n+1}(b)_{2n+1}(\delta)_{2n+1}}\\ \times
 \sum_{k=0}^{2n+1}\sum_{l=0}^k
 (-1)^{l+k}\binom{2n+1}{k}\binom{k}{l} (a+2n+1-k+l)_{k-l}(\delta+2n+1-l)_l x^{2n+1-k+l}(1-x)^{2n+1-l}\\
 \times \frac{(c- b+n+1)_{2n-k+1}(d-b+n)_{2n-k+1}}{(\delta+2n+1)_{2n-k+1}}\,
 \,\tensor[_2]{F}{_1}\hspace*{-3pt}\left[{\begin{NiceArray}{c}[small]
 c-b+3n+2-k,d-b+3n+1-k \\\delta+4n+2-k\end{NiceArray}};1-x\right] .
 \end{multlined}
 \end{align*}
\end{coro}
\begin{proof}
 The result follows from the well known equation
 \begin{align*}
 \frac{\d^m}{\d z^m}\tensor[_2]{F}{_1}\!\!\left[{\begin{NiceArray}{c}[small]\alpha,\; \beta \\\gamma \end{NiceArray}};x\right]=
 \frac{(\alpha)_m(\beta)_m}{(\gamma)_m}\,{}_2F_1\hspace*{-3pt}\left[{\begin{NiceArray}{c}[small]\alpha+m , \beta +m\\\gamma +m\end{NiceArray}};x\right] ,
 \end{align*}
 applied to \eqref{eq:Q}.
\end{proof}

For the next Corollary we require the system of weights is such that the zeros are confined in the support of the measure. For example, this happens whenever the system of weights is an AT system.
This holds when the min-max constraint \eqref{eq:min-max} is assumed, so the system is Nikishin \cite{lima_loureiro}, but it also holds for another system of hypergeometric parameters, as those leading to the Piñeiro multiple orthogonal polynomials. 
\begin{coro}\label{coro:psoitivity>1}
Whenever $a,b,c,d>0$, the linear forms fulfill $q_n(x)>0$, for all $n \in \N_0$ and $x\geq 1$.
\end{coro}
\begin{proof}
 As the zeros of $q_n(x)$ belong to $(0,1)$, the result follows from Corollary \ref{coro:positivity}.
\end{proof}
\subsection{Ratio asymptotics at unity of linear forms of type I }

\begin{coro}\label{coro:ratio_asymptotics_typeI}
 The large $n$ ratio asymptotics for the type I hypergeometric linear forms at unity is 
 \begin{align*}
 \lim_{n\to\infty} \frac{q^{(n+1)}(1)}{q^{(n)}(1)}=\frac{1}{2\kappa}=\frac{27}{8}.
 \end{align*}
\end{coro}
\begin{proof}
It follows from the explicit values of the linear form at unity. Indeed,
 \begin{align*}
\hspace{-.25cm}
\lim_{n\to\infty} \frac{q^{(2n+1)}(1)}{q^{(2n)}(1)}&=\lim_{n\to\infty}\frac{(2n)!}{(2n+1)!} \frac{(c)_{3n+2}(d)_{3n+1}(a)_{2n}(b)_{2n}}{(c)_{3n}(d)_{3n}(a)_{2n+1}(b)_{2n+1}} =\lim_{n\to\infty}\frac{(c+3n)(c+3n+1)(d+3n)}{(2n+1)(a+2n)(b+2n)} =\frac{27}{8},\\
\hspace{-.25cm}
\lim_{n\to\infty} \frac{q^{(2n+2)}(1)}{q^{(2n+1)}(1)}&=\lim_{n\to\infty}\frac{(2n+1)!}{(2n+2)!} \frac{(c)_{3n+3}(d)_{3n+3}(a)_{2n+1}(b)_{2n+1}}{(c)_{3n+2}(d)_{3n+1}(a)_{2n+2}(b)_{2n+2}} =\lim_{n\to\infty}\frac{(c+3n+2)(d+3n+1)(d+3n+2)}{(2n+2)(a+2n+1)(b+2n+1)} =\frac{27}{8}.
 \end{align*}
Which is what we wanted to prove.
\end{proof}

\subsection{Ratio asymptotics at unity of multiple orthogonal polynomials of type II}

A well known classical tool in the theory of orthogonal polynomials is the Christoffel--Darboux (CD) kernel, see~\cite{simon}. 
Multiple orthogonal polynomials have also Christoffel--Darboux kernels, 
Sorokin and Van Iseghem~\cite{sorokin} derived a formula that can be applied to multiple orthogonal polynomials, see also~\cite{Coussement__VanAssche} and~\cite{tesis}. Daems and Kuijlaars derived a~CD formula for the mixed multiple case~\cite{daems-kuijlaars,daems-kuijlaars2} using Riemann--Hilbert approach in the context of nonintersecting Brownian motions. Later on, the article~\cite{afm} reproduces the same result with an algebraic approach 
in the framework of mixed multiple orthogonal polynomial sequences.
The~CD formula derived in~\cite{CD} suits particularly well to our problem because it is expressed only in terms of a unique polynomial sequence.

\begin{pro}\label{pro:ratio_asymptotics_typeII}
 The large $n$ ratio asymptotics for the type II hypergeometric multiple orthogonal polynomials at unity~is 
\begin{align*}
 \lim_{n\to\infty} \frac{B^{(n+1)}(1)}{B^{(n)}(1)}=2\kappa=\frac{8}{27}.
\end{align*}
\end{pro}
\begin{proof}
From the Christoffel--Darboux formula on the sequence, see \cite[Proposition 6]{bfmaf}, we find 
\begin{multline}\label{eq:CD_QP_1}
q^{(n+1)}(1)B^{(n+2)}(1)
 -q^{(n+2)}(1)J_{n+2,n}B^{(n)}(1) -q^{(n+2)}(1) J_{n+2,n+1}B^{(n+1)}(1)
 -q^{(n+3)}(1)J_{n+3,n+1}B^{(n+1)}(1)=0,
\end{multline}
that allows 
us the application of the ideas in \cite[Theorem 4]{bfmaf}.
Notice that \eqref{eq:CD_QP_1} can be written as
\begin{align}\label{eq:CD_QP_2}
B^{(n+2)}(1) =
 a_nB^{(n+1)}(1)+b_nB^{(n)}(1), 
\end{align}
with
$
 a_n :=\frac{q^{(n+2)}(1)}{q^{(n+1)}(1)} \alpha_{n+2}
 +\frac{q^{(n+3)}(1)}{q^{(n+1)}(1)}\gamma_{n+2}$ and $
 b_n :=\frac{q^{(n+2)}(1)}{q^{(n+1)}(1)}\gamma_{n+1}$.
 From the explicit expressions for the hypergeometric type I linear forms at unity we find that
\begin{align*}
 \lim_{n\to\infty}a_n&=\frac{1}{2\kappa}3\kappa^2+\frac{1}{4\kappa^2}\kappa^3=\frac{7\kappa}{4},&
 \lim_{n\to\infty}b_n&=\frac{1}{2\kappa}\kappa^3=\frac{\kappa^2}{2}.
\end{align*}
Hence, we can apply the Poincaré's theory \cite{poincare,montel,Norlund, Elaydi} for homogeneous linear recurrence relations with coefficients having finite large $n$ limits. The characteristic polynomial 
$ r^2-\frac{7\kappa}{4}r-\frac{\kappa^2}{2}=\Big(r+\frac{1}{27}\Big)\Big(r-\frac{8}{27}\Big)$
has two simple roots, $-\frac{1}{27}$ and $\frac{8}{27}$, with distinct absolute value. Consequently, being a sequence of positive numbers, following Poincaré we find that it converges to the positive root~$\frac{27}{8}$.
\end{proof}

\begin{coro}[Ratio asymptotics for ${}_3F_2$]
We have the following large $n$ limit for the ratio of two generalized hypergeometric functions
\begin{align*}
 \lim_{n\to\infty}\frac{\tensor[_3]{F}{_2}\hspace*{-3pt}\left[{\begin{NiceArray}{c}[small]-2n-2,\; c+n+1,\;d+n \\a,\;b\end{NiceArray}};1\right]}{\tensor[_3]{F}{_2}\hspace*{-3pt}\left[{\begin{NiceArray}{c}[small]-2n-1,\; c+n,\;d+n \\a,\;b\end{NiceArray}};1\right] }=-2.
\end{align*}
\end{coro}

\begin{proof}
As we have
\begin{align*}
	 B^{(2n)}(x)&=
 \frac{(a)_{2n}(b)_{2n}}{(c+n)_{2n}(d+n-1)_{2n}}
 \tensor[_3]{F}{_2}\hspace*{-3pt}\left[{\begin{NiceArray}{c}[small]-2n,\; c+n,\;d+n-1 \\a,\;b\end{NiceArray}};x\right],&
B^{(2n+1)}(x)&=-
\frac{(a)_{2n+1}(b)_{2n+1}}{(c+n)_{2n+1}(d+n)_{2n+1}}
\tensor[_3]{F}{_2}\hspace*{-3pt}\left[{\begin{NiceArray}{c}[small]-2n-1,\; c+n,\;d+n \\a,\;b\end{NiceArray}};x\right]. 
\end{align*}
The ratios we are interested in are
\begin{align*}
 \frac{B^{(2n+1)}(1)}{B^{(2n)}(1)}&=-\frac{(a+2n)(b+2n)(d+n-1)}{(c+3n)(d+3n)(d+3n-1)}\frac{\tensor[_3]{F}{_2}\hspace*{-3pt}\left[{\begin{NiceArray}{c}[small]-2n-1,\; c+n,\;d+n \\a,\;b\end{NiceArray}};1\right]}{\tensor[_3]{F}{_2}\hspace*{-3pt}\left[{\begin{NiceArray}{c}[small]-2n,\; c+n,\;d+n-1 \\a,\;b\end{NiceArray}};1\right]},\\
 \frac{B^{(2n+2)}(1)}{B^{(2n+1)}(1)}&=-\frac{(a+2n+1)(b+2n+1)(c+n)}{(c+3n+1)(c+3n+2)(d+3n+1)}\frac{\tensor[_3]{F}{_2}\hspace*{-3pt}\left[{\begin{NiceArray}{c}[small]-2n-2,\; c+n+1,\;d+n \\a,\;b\end{NiceArray}};1\right]}{\tensor[_3]{F}{_2}\hspace*{-3pt}\left[{\begin{NiceArray}{c}[small]-2n-1,\; c+n,\;d+n \\a,\;b\end{NiceArray}};1\right] }.
\end{align*}
Notice that if the ratio asymptotics at unity, $ \lim\limits_{n\to\infty}\frac{B^{(2n+1)}(1)}{B^{(2n)}(1)}$, exists, in terms of $\kappa=\frac{4}{27}$, we have
\begin{align*}
\begin{aligned}
 \lim_{n\to\infty} \frac{B^{(2n+1)}(1)}{B^{(2n)}(1)}&=- \kappa\lim_{n\to\infty}\frac{
 	\tensor[_3]{F}{_2}\hspace*{-3pt}\left[{\begin{NiceArray}{c}[small]-2n-1,\; c+n,\;d+n \\a,\;b\end{NiceArray}};1\right]}{\tensor[_3]{F}{_2}\hspace*{-3pt}\left[{\begin{NiceArray}{c}[small]-2n,\; c+n,\;d+n-1 \\a,\;b\end{NiceArray}};1\right]} , &
 \lim_{n\to\infty} \frac{B^{(2n+2)}(1)}{B^{(2n+1)}(1)}&=- \kappa\lim_{n\to\infty}\frac{\tensor[_3]{F}{_2}\hspace*{-3pt}\left[{\begin{NiceArray}{c}[small]-2n-2,\; c+n+1,\;d+n \\a,\;b\end{NiceArray}};1\right]}{\tensor[_3]{F}{_2}\hspace*{-3pt}\left[{\begin{NiceArray}{c}[small]-2n-1,\; c+n,\;d+n \\a,\;b\end{NiceArray}};1\right]} ,
\end{aligned}
 \end{align*}
and using Proposition \ref{pro:ratio_asymptotics_typeII} the result follows.
\end{proof}

\subsection{Ratio asymptotics in compact subsets}

For this section we need that the system of weights is such that the zeros are confined in the support of the measure. For example, this happens whenever the system of weights is an AT system. 

\begin{rem}
 In \cite[\S3.4]{lima_loureiro} it was shown that
\begin{align}\label{eq:B_F}
 \lim_{n\to\infty}\frac{B^{(n+1)}(x)}{B^{(n)}(x)}&=-\frac{\kappa}{F(x)},
\end{align}
with
\begin{align}\label{eq:F}
F(x)&:=1-\frac{3}{2}\sqrt[3]{x}
 \Big(\omega^2\sqrt[3]{\sqrt{1-x}-1}+\omega\sqrt[3]{-\sqrt{1-x}-1}\Big), &\omega&:=\Exp{\frac{2\pi\ii}{3}},
\end{align}
uniformly in compact subsets of $\C\setminus[0,1]$. Taking $\sqrt[3]{-1}=-1$ we get
\begin{align*}
 \lim_{x\to 1^+} F(x)=1+\frac{3}{2}(\omega^2+\omega)=1-\frac{3}{2}=-\frac{1}{2}.
\end{align*}
Consequently, we have
\begin{align*}
 \lim_{x\to 1^+} \lim_{n\to\infty}\frac{B^{(n+1)}(x)}{B^{(n)}(x)}=2\kappa,
\end{align*}
so that using Proposition \ref{pro:ratio_asymptotics_typeII} we conclude
\begin{align*}
 \lim_{x\to 1^+} \lim_{n\to\infty}\frac{B^{(n+1)}(x)}{B^{(n)}(x)}= \lim_{n\to\infty}\lim_{x\to 1^+}\frac{B^{(n+1)}(x)}{B^{(n)}(x)}.
\end{align*}
\end{rem}

\begin{rem}
As $B^{(n)}(x)>0$ for $x\geq1$, from \eqref{eq:B_F} we get that $F(x)<0$ for $x>1$.
Moreover, from \cite[Lemma~3.5]{Coussement_Coussment_VanAssche}, $F(x)$ is analytic in $\C\setminus[0,1]$.
\end{rem}


\begin{teo}\label{teo:ratio asymptotics_linear forms}
 For $F(x)$ as in \eqref{eq:F}, we have
 \begin{align*}
 \lim_{n\to\infty}\frac{q^{(n+1)}(x)}{q^{(n)}(x)}=\frac{1}{2\kappa}\left(
 F(x)-3+\sqrt{\frac{x}{\kappa}\frac{F(x)-4}{F(x)-1}}\,
 \right),
 \end{align*}
pointwise in $\C\setminus (-\infty,1]$. There exists an open dense subset $\Omega \subset \{x\in \C\setminus (-\infty,1]: F(x)\not\in [0,4] \}$ in where the convergence is
 uniform in compact sets.
\end{teo}
\begin{proof}From \cite[Proposition 6]{bfmaf} we get
\begin{align*}
 q^{(n+1)}(x)\frac{B^{(n+2)}(x)}{B^{(n+1)}(x)}
 -q^{(n+2)}(x)\Big(\gamma_{n+1}\frac{B^{(n)}(x) }{B^{(n+1)}(x)}+\alpha_{n+2}\Big)
 -q^{(n+3)}(x)\gamma_{n+2}=0,
\end{align*}
that can be written as follows 
\begin{align}\label{eq:CD_Q}
a_{n+1}(x) q^{(n+1)}(x)+b_{n+2}(x)q^{(n+2)}(x)+c_{n+3}q^{(n+3)}(x)=0,
\end{align}
where
\begin{align*}
 \hspace*{-.25cm}
\begin{aligned}
 a_n(x)&:=\frac{B^{(n+1)}(x)}{B^{(n)}(x)}\xrightarrow[n\to\infty]{}-\frac{\kappa}{F(x)},&
 b_{n+2}(x)&:=-\gamma_{n+1}\frac{B^{(n)}(x) }{B^{(n+1)}(x)}- \alpha_{n+2}\xrightarrow[n\to\infty]{}\kappa^2(F(x)-3),&
 c_{n+3}&:=-\gamma_{n+2}\xrightarrow[n\to\infty]{}-\kappa^3 ,
\end{aligned}
\end{align*}
which holds uniformly on compact sets of $\C\setminus[0,1]$ (in the $a_n(x)$ and $b_n(x)$ cases).

The characteristic equation for the third order linear homogeneous recurrence \eqref{eq:CD_Q} reads as
\begin{align*}
\kappa^2r^2-(F(x)-3)\kappa r+ \frac{1}{F(x)}=0 ,
\end{align*}
with roots given by
\begin{align*}
 r_\pm(x)=\frac{F(x)-3\pm
 \sqrt{
 (F(x)-3)^2-\frac{4}{F(x)}
 }}
 {2 \kappa} .
\end{align*}
As $F(x)<0$ for $x>1$ we get $r_\pm(x)\gtrless 0$ for $x>1$. Given that $q^{(n)} (x)>0$ for $x\geq 1$, see Corollary \ref{coro:psoitivity>1}, we must have, according to Poincaré \cite{poincare} that
$
\lim_{n\to\infty}\frac{q^{(n+1)}(x)}{q^{(n)}(x)}=r_+(x)$
pointwise for $x>1$.
For the rest $x\in\C\setminus \R$, Poincaré theory ensures that the pointwise limit exists and its value must be $r_+(x)$ or $r_-(x)$.

Notice that the linear forms $q^{(n)}(x)$ have branch with a cut at $(-\infty,0]$; moreover, being $(w_1,w_2)$ an~AT~system has all its zeros in $(0,1)$. Hence,
$f_n(x):=\frac{q^{(n+1)}(x)}{q^{(n)}(x)}$ 
is holomorphic in the domain $\Omega_1:= \C\setminus(-\infty,1]$.
Therefore, we have a sequence of holomorphic functions $\{f_n\}_{n=0}^\infty$ with
pointwise convergence in the domain~$\Omega_1$.

Now, according to Osgood theorem (cf. \cite{Osgood,Beardon_Minda,Krantz}) the sequence converges pointwise in an open dense subset of $\Omega_1$ to a holomorphic function.
Hence, as in $(1,+\infty)$ converges to $r_+(x)$, by analytical continuation it converges pointwise in $\Omega_0=\{x\in \Omega_1: F(x)\not\in[0,4]\}$ to $r_+(x)$, which is holomorphic in that domain. Notice that the set $0\leq F(x)\leq 4$ is the cut of $r_+$, determined by the cut of the square root in the function $r_+(x)$, thus in that set the function $r_+(x)$ ceases to be holomorphic. 

 From Osgood's theorem we also know that 
 there is an open dense subset $\Omega\subset \Omega_0\subset \Omega_1$, $\overline \Omega= \overline \Omega_0= \overline \Omega_1=\C$, in where the convergence can be taken to be uniform in compact sets.

Now, as $F=-\kappa\phi$ with $\phi$ given in \cite[Equation (3.5)]{Coussement_Coussment_VanAssche} and, according to \cite[Lemma 3.5]{Coussement_Coussment_VanAssche}, we have $x\phi=(1+\kappa\phi)^3$, we are led to $xF=\kappa(F-1)^3$, and, consequently, we find for the radicand in $r_+(x)$ the alternative expression
\begin{align*}
 (F-3)^2-\frac{4}{F}=\frac{F^3-6F^2+9F-4}{F}=\frac{(F-1)^3-3(F-1)^2}{F}=\frac{x}{\kappa}\frac{F-4}{F-1} ,
\end{align*}
and from this we get the desired representations for the limit function.
\end{proof}



\section{Multiple hypergeometric random walks}\label{S:randowm_walks}

In this section we assume \eqref{eq:positivity_jacobi}, i.e. $a,b>0$, $d>\max(a,b)$, $c\geq a$ and $c+1\geq b$, 
so that the Jacobi matrix is a nonnegative matrix. Being the Jacobi matrix nonnegative, the hypergeometric multiple orthogonal polynomials of Lima and Loureiro are in fact random walk multiple orthogonal polynomials. 

\subsection{The stochastic transition matrices}
Following \cite{bfmaf} one concludes that there are two stochastic matrices
\begin{teo}\label{teo:JPII_stochastic}
 Let us assume for the hypergeometric system that \eqref{eq:region_parameters_pochhammer_perfect} holds.~Then, 
 \begin{enumerate}
 \item If the zeros of the polynomial sequence $\{ B^{(n)} \}_{n\in\N}$ are confined in $(0,1)$, the semi-infinite matrix
 \begin{align}\label{eq:stochastic_hypergeometric_II}
 P_{II}=
 \left( \begin{NiceMatrix}[columns-width = .5cm]
 P_{II,0,0} & P_{II,0,1}& 0 & \Cdots & & \\
 P_{II,1,0}&P_{II,1,1}& P_{II,1,2}& \Ddots& & \\[5pt]
 P_{II,2,0}& P_{II,2,1}& P_{II,2,2} & P_{II,2,3}&&\\[5pt]
 0& P_{II,3,1}&P_{II,3,2}& P_{II,3,3} & P_{II,3,4}& \\
 \Vdots&\Ddots&\Ddots&\Ddots&\Ddots&\Ddots
 \end{NiceMatrix}\right)
 \end{align}
with coefficients given in terms of the Jacobi matrix~\eqref{eq:jacobi_hyper_coeff} and the multiple orthogonal polynomials of type~II evaluated at unity, $B^{(n)}(1)$ for $n \in\N_0$, as~follows
 \begin{align*}
 P_{II,n,n+1} & = \frac{B^{(n+1)}(1)}{B^{(n)}(1)}, &
 P_{II,n,n} & = \beta_{n}, &
 P_{II,n+1,n} & = \frac{B^{(n)}(1)}{B^{(n+1)}(1)}\alpha_{n+1}, &
 P_{II,n+2,n}&= \frac{B^{(n)}(1)}{B^{(n+2)}(1)}\gamma_{n+1}.
 \end{align*}
is a multiple stochastic matrix of type~II. 
 \item 
The semi-infinite matrix
 \begin{align}\label{eq:stochastic_Jacobi_Piñeiro_I}
 P_I=\left(\begin{NiceMatrix}[columns-width = .5cm]
 P_{I,0,0} & P_{I,0,1}& P_{I,0,2} & 0 & \Cdots & \\
 P_{I,1,0}&P_{I,1,1}& P_{I,1,2}& P_{I,1,3}& \Ddots& \\[5pt]
 0& P_{I,2,1}& P_{I,2,2} & P_{I,2,3} & P_{I,2,4}&\\
 \Vdots & \Ddots&P_{I,3,2}& P_{I,3,3} & P_{I,3,4}& \Ddots\\
 &&\Ddots&\Ddots&\Ddots&\Ddots
 \end{NiceMatrix}\right)
 \end{align}
with coefficients expressed in terms 
of the Jacobi matrix~\eqref{eq:jacobi_hyper_coeff} and the linear forms of type~I evaluated at unity, $Q^{(n)}(1)$, for $n\in\N_0$, as~follows
 \begin{align*}
 P_{I,n,n+2}&= \frac{q^{(n+2)}(1)}{q^{(n)}(1)}\gamma_{n+1}, &
 P_{I,n,n+1}&= \frac{q^{(n+1)}(1)}{q^{(n)}(1)}\alpha_{n+1}, &
 P_{I,n,n}&=\beta_{n}, &
 P_{I,n+1,n}&=\frac{q^{(n)}(1)}{q^{(n+1)}(1)},
 \end{align*}
 is a multiple stochastic matrix of type~I. 
 \end{enumerate}
 
\end{teo}
\begin{proof}
i) As the zeros of the polynomials $B^{(n)}$ belong to $(0,1)$ and $B^{(n)}$ is monic and therefore diverges to $+\infty$ when $x\to+\infty$, we deduce that $B^{(n)}(1)>0$, for $n\in\N_0$. 
 
 
ii) It follows from Theorem \ref{teo:linear_forms_at_unity}.
\end{proof}


The corresponding diagrams for these Markov chains are
\begin{center}
 \tikzset{decorate sep/.style 2 args={decorate,decoration={shape backgrounds,shape=circle,shape size=#1,shape sep=#2}}}
 \begin{tikzpicture}[start chain = going right,
 -latex, every loop/.append style = {-latex}]
 \foreach \i in {0,...,5}
 \node[state, on chain,fill=gray!50!white] (\i) {\i};
 \foreach
 \i/\txt in {0/$P_{01}$,1/$P_{12}$/,2/$P_{23}$,3/$P_{34}$,4/$P_{45}$}
 \draw let \n1 = { int(\i+1) } in
 (\i) edge[bend left,"\txt",color=Periwinkle] (\n1);
 \foreach
 \i/\txt in {0/$P_{10}$,1/$P_{21}$/,2/$P_{32}$,3/$P_{34}$,4/$P_{54}$}
 \draw let \n1 = { int(\i+1) } in
 (\n1) edge[bend left, "\txt",color=Mahogany,auto=right] (\i);
 
 \foreach
 \i/\txt in {0/$P_{20}$,1/$P_{31}$/,2/$P_{42}$,3/$P_{53}$}
 \draw let \n1 = { int(\i+2) } in
 (\n1) edge[auto,bend left=50,"\txt",color=RawSienna] (\i);
 
 \foreach \i/\txt in {1/$P_{11}$,2/$P_{22}$/,3/$P_{33}$,4/$P_{44}$,5/$P_{55}$}
 \draw (\i) edge[loop above,color=NavyBlue, "\txt"] (\i);
 \draw (0) edge[loop left, color=NavyBlue,"$P_{00}$"] (0);
 
 \draw[decorate sep={1mm}{4mm},fill] (10.5,0) -- (12,0);
 \end{tikzpicture}
 \begin{tikzpicture}
 \draw (4,-1.8) node
 {\begin{minipage}{0.8\textwidth}
 \begin{center}\small
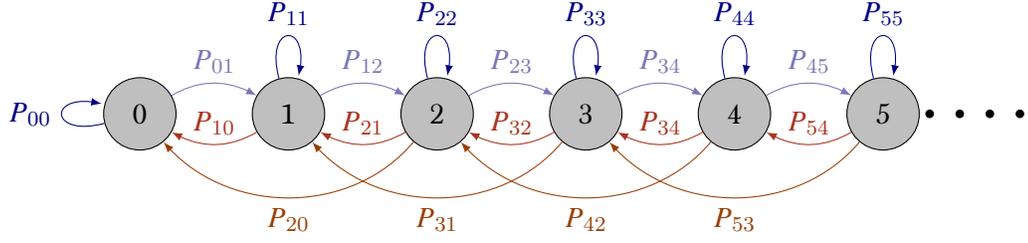
 \textbf{Type II Markov chain diagram}
 \end{center}
 \end{minipage}};
 \end{tikzpicture}

 \begin{tikzpicture}[start chain = going right,
 -latex, every loop/.append style = {-latex}]
 \foreach \i in {0,...,5}
 \node[state, on chain,fill=gray!50!white] (\i) {\i};
 \foreach
 \i/\txt in {0/$P_{01}$,1/$P_{12}$/,2/$P_{23}$,3/$P_{34}$,4/$P_{45}$}
 \draw let \n1 = { int(\i+1) } in
 (\i) edge[bend left,"\txt",below,color=Periwinkle] (\n1);
 \foreach
 \i/\txt in {0/$P_{10}$,1/$P_{21}$/,2/$P_{32}$,3/$P_{34}$,4/$P_{54}$}
 \draw let \n1 = { int(\i+1) } in
 (\n1) edge[bend left,below, "\txt",color=Mahogany,auto=right] (\i);
 
 \foreach
 \i/\txt in {0/$P_{02}$,1/$P_{13}$/,2/$P_{24}$,3/$P_{35}$}
 \draw let \n1 = { int(\i+2) } in
 (\i) edge[bend left=60,color=MidnightBlue,"\txt"](\n1);
 
 \foreach \i/\txt in {1/$P_{11}$,2/$P_{22}$/,3/$P_{33}$,4/$P_{44}$,5/$P_{55}$}
 \draw (\i) edge[loop below, color=NavyBlue,"\txt"] (\i);
 \draw (0) edge[loop left, color=NavyBlue,"$P_{00}$"] (0);
 
 \draw[decorate sep={1mm}{4mm},fill] (10.5,0) -- (12,0);
 \end{tikzpicture}
 \begin{tikzpicture}
 \draw (4,-1.8) node
 {\begin{minipage}{0.8\textwidth}
 \begin{center}\small
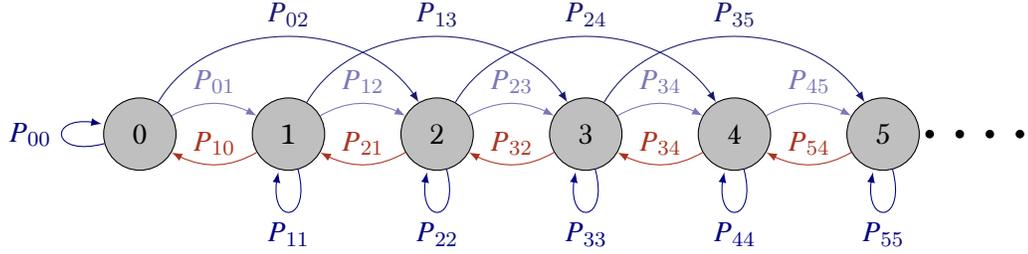
 \textbf{Type I Markov chain diagram}
 \end{center}
 \end{minipage}};
 \end{tikzpicture}
\end{center}

\begin{coro}
 The coefficients of the type I stochastic matrix can be written, for all $n \in \N_0$, as follows
 \begin{align*}
 \begin{aligned}
 P_{I,2n,2n+2} & = 
 \frac{(d+n-1) (d-a+n) (d-b+n)}{
 (d+3 n-1)_3} ,
 &
 P_{I,2n+1,2n} & =\frac{(2n+1)(a+2n-1)(b+2n-1)}{(c+3n)(d+3n-1)_2}, 
 \\
 P_{I,2n+1,2n+3} & =
 \frac{(c+n) (c-a+n+1) (c-b+n+1)}{
 (c + 3 n+1)_3} ,
 &
 P_{I,2n+2,2n+1} & =\frac{(2n+1)(a+2n)(b+2n)}{(c+3n)_2(d+3n)},
 \end{aligned} \phantom{olaolaolaolaol}
 \\
 \begin{aligned}
 P_{I,2n,2n} & = 
 \frac{(2 n+1) (a+2 n) (b+2 n)}{(c+3 n) (d+3 n)}-\frac{2 n (a+2 n-1) (b+2 n-1)}{(c+3 n-1) (d+3 n-2)} ,
 \end{aligned} \phantom{olaolaolaolaolaolaolaolaolaolaolaolaola}
 \\
 \begin{aligned}
 P_{I,2n+1,2n+1} & = 
 \frac{2 (n+1) (a+2 n+1) (b+2 n+1)}{(c+3 n+2) (d+3 n+1)}-\frac{(2 n+1) (a+2 n) (b+2 n)}{(c+3 n) (d+3 n)} ,
 \end{aligned}\phantom{olaolaolaolaolaolaolaolaolaolaolaol}
 \\
 \begin{multlined}[t][1\textwidth]
 P_{I,2n,2n+1} 
 = \frac{n (a+2 n-1) (b+2 n-1) (c+3 n+1)}{(c+3 n-1) (d+3 n-1)} \\
 -\frac{(2 n+1) (a+2 n) (b+2 n) (c+3 n+1)}{(c+3 n) (d+3 n)}+\frac{(n+1) (a+2 n+1) (b+2 n+1)}{d+3 n+1} , \phantom{olaolaolaolaola}
 \end{multlined}
 \\
 \begin{multlined}[t][1\textwidth]
 P_{I,2n+1,2n+2} 
 =
 (d+3 n+2) \left( -\frac{2 (n+1) (a+2 n+1) (b+2 n+1)}{(c+3 n+2) (d+3 n+1)} \right.
 \\
 \left. +\frac{(2 n+3) (a+2 (n+1)) (b+2 (n+1))}{2 (c+3 n+3) (d+3 n+2)}+\frac{(2 n+1) (a+2 n) (b+2 n)}{2 (c+3 n+1) (d+3 n)}\right). \phantom{olaolaolaolaola}
 \end{multlined}
 \end{align*}
\end{coro}

 \begin{pro}[Large $n$ limit for the dual hypergeometric Markov chains]\label{pro:hypergeometric_stochastic_dual}
 The large $n$ limit of the hypergeometric stochastic matrices of type~I and II are the same after transposition, i.e.,
 \begin{align*}
 \lim_{n\to\infty} P_{I,n,n+k}
 &=\lim_{n\to\infty}P_{II,n+k,n}, & k\in\{-2,-1,0,1\}.
 \end{align*}
\end{pro}
\begin{proof}
We need to show that
 \begin{align}\label{eq:other_relations}
 \frac{B^{(n-k)}(1)q^{(n-k)}(1)}{B^{(n)}(1)q^{(n)}(1)}&\xrightarrow[n\to\infty]{}1, &k&=2,1,-1.
 \end{align}
But these relations follow from Corollary \ref{coro:ratio_asymptotics_typeI} and Proposition \ref{pro:ratio_asymptotics_typeII}.
\end{proof}

\begin{rem}[Karlin--McGregor representation formulas]
 The Karlin--McGregor representation formulas in terms of multiple orthogonal polynomials are given by Theorems \ref{teo:KMcG} and \ref{teo:KMcG2}, that hold, as we showed above, for these hypergeometric multiple random walks. 
\end{rem}

\begin{teo}[Recurrent and transient hypergeometric random walks]\label{Theorem:recurrent-transient}
 Both dual hypergeometric random walks are recurrent whenever $0<\delta\leq1$ and transient for $\delta> 1$.
\end{teo}
\begin{proof}
 From \cite[Theorem 8]{bfmaf} we know that
 both dual Markov chains are recurrent if and only if the integral
 \begin{align}\label{eq:integral}
 \int_0^1 \frac{w_1(x)}{1-x}\d \mu(x),
 \end{align}
 diverges, and both dual Markov chains are transient whenever the integral converges.
 
 According to \eqref{eq:measures} and \eqref{eq:mu} we have to discuss the convergence
of the integral
 \begin{align}
 \int_0^1 {}_2F_1\hspace*{-3pt}\left[{\begin{NiceArray}{c}[small]c-b,d-b \\\delta\end{NiceArray}};1-x\right]\frac{x^{a-1}(1-x)^{\delta-1}}{1-x}\d x.
\end{align}
Hence, the divergence appears linked to the behavior of the integrand near the point $x=1$. The divergence of the integral appears for $\delta-2\leq-1$ and the integral converges for $\delta-2>-1$. That is, for $\delta\leq1$ the integral diverges and for $\delta >1$ the integral converges.
\end{proof}

\begin{rem}
 As we discussed in \cite{bfmaf}, given that there is no mass points, we conjecture that this recurrence is a null recurrence, and the mean of the return times is infinity. 
\end{rem}

\begin{rem}
 Notice that $\displaystyle\int_0^1 \frac{w_1(x)}{1-x}\d \mu(x)$ is the Stieltjes--Markov transform of the measure $w_1\d\mu$ at unity.
\end{rem}

Following \cite{bfmaf} we get
\begin{pro}\label{pro:steady}
The state
$ \boldsymbol \pi=(\begin{NiceMatrix}
 B^{(0)}(1) q^{(0)}(1)&B^{(1)}(1) q^{(1)}(1) &\Cdots
 \end{NiceMatrix})$
satisfies the steady state conditions
$ \boldsymbol \pi P_{I}=\boldsymbol \pi$ and $\boldsymbol \pi P_{II}=\boldsymbol \pi$.
\end{pro}

\begin{rem}
In \cite{bfmaf} it was conjectured that the state
$ \boldsymbol \pi=\begin{pNiceMatrix}
 B^{(0)}(1) q^{(0)}(1)&B^{(1)}(1) q^{(1)}(1) &\Cdots
 \end{pNiceMatrix}$
does not belong to~$\ell^1$,~i.e. not being a proper steady state.
\end{rem}


\subsection{Stochastic $LU$ factorizations of the Markov transition matrices}

Inspired by \cite{grunbaum_de la iglesia} in this section we find stochastic factorizations of the Jacobi matrix.

Following \cite{Barrios_Branquinho_Foulquie} we easily find
\begin{lemma}\label{lem:Jacobi_LLU}
 The Jacobi matrix $J$ given in \eqref{eq:Jacobi} with coefficients given in \eqref{eq:jacobi_hyper_coeff} and \eqref{eq:lambdas} has the following Gauss--Borel factorization
\begin{align*}
 J = L_1 L_2 U
\end{align*}
where
\begin{align*}
 L_1 & := \left( 
 \begin{NiceMatrix}
 1 & 0&\Cdots & \\[-5pt]
 \lambda_3 & 1 & \Ddots& \\
 0 & \lambda_6 & 1 & \\
 \Vdots & \Ddots & \Ddots & \Ddots 
 \end{NiceMatrix}
 \right),
 &
 L_2 & := \left( \begin{NiceMatrix}
 1 & 0 &\Cdots & \\[-5pt]
 \lambda_4 & 1 & \Ddots& \\
 0 & \lambda_7 & 1 & \\
 \Vdots& \Ddots & \Ddots & \Ddots 
 \end{NiceMatrix}
 \right) 
&
 U & := \left( \begin{NiceMatrix}
 \lambda_2 & 1 & 0& \Cdots& \\[-5pt]
 0 & \lambda_5 & 1 &\Ddots & \\[-5pt]
 \Vdots& \Ddots & \lambda_8 & 1 &\\
 & & & \Ddots & \Ddots
 \end{NiceMatrix}
 \right).
 \end{align*}
\end{lemma}

This simple $LU$ factorization induces a corresponding $LU$ factorization with stochastic factors of the stochastic matrices $P_{II}$ and $P_I$, i.e. stochastic Gauss--Borel factorizations.

\begin{teo}[Stochastic $LU$ factorization]\label{teo:stochastic _factorization}
\begin{enumerate}
 \item The stochastic matrix $P_{II}$ has the following stochastic $LU$ factorization
 \begin{align}\label{eq:stochastic_factorization_II}
 P_{II}=P_{II,1}^L P_{II,2}^L P_{II}^U
 \end{align} 
in terms of the stochastic matrices 
\begin{align*}
P_{II,1}^L &:=\sigma_{II} L_1 D_{II,2}^{-1}, &
P_{II,2}^L &:=D_{II,2} L_2 D_{II,1}^{-1} ,&
P_{II}^U&:=D_{II,1}U \sigma_{II}^{-1},
\end{align*}
where
\begin{align*}
 \sigma_{II}& = \diag \left( \begin{NiceMatrix} \frac{1}{B^{(0)}(1)} & \frac{1}{B^{(1)}(1)} & \Cdots \end{NiceMatrix}\right), &
 D_{II,i }&= \operatorname{diag} \left( \begin{NiceMatrix} \frac{1}{ d_{II,i}^{(0)}}& \frac{1}{d_{II,i}^{(1)}} & \Cdots \end{NiceMatrix} \right), &i&\in\{1,2\},
\end{align*}
with
\begin{align*}
 d_{II,1}^{(n)} &= \lambda_{3n+2}B^{(n)}(1) + B^{(n+1)}(1) , & n&\in \N_0\\
 d_{II,2}^{(n)} & = \lambda_{3n+1}\lambda_{3n-1}B^{(n-1)}(1) + (\lambda_{3n+1} +\lambda_{3n+2} )B^{(n)}(1) + B^{(n+1)}(1)& n&\in \N,
\end{align*}
and $d_{II,2}^{(0)}=d_{II,1}^{(0)}$.

\item The stochastic matrix $P_I$ has the following stochastic $LU$ factorization
 \begin{align}\label{eq:stochastic_factorization_I}
 P_{I}=P_{I}^L P_{I,2}^U P_{I,1}^U
\end{align} 
in terms of stochastic matrices
\begin{align*}
 P_{I}^L &:=\sigma_I U^\top D_{I,2}^{-1},&
 P_{I,2}^U&:= D_{I,2} L_2^\top D_{I,1}^{-1},&
 P_{I,1}^U&:=D_{I,1} L_1^\top \sigma_I ^{-1},
\end{align*}
where
\begin{align*}
 \sigma_{I}& = \diag \left( \begin{NiceMatrix} \frac{1}{q^{(0)}(1)} & \frac{1}{q^{(1)}(1)} & \Cdots \end{NiceMatrix}\right), &
 D_{I,i }&= \operatorname{diag} \left( \begin{NiceMatrix} \frac{1}{ d_{I,i}^{(0)}}& \frac{1}{d_{I,i}^{(1)}} & \Cdots \end{NiceMatrix} \right), &i&\in\{1,2\},
\end{align*}
with
\begin{align*}
 d_{I,1}^{(n)} &= q^{(n)}(1)+\lambda_{3n+3}q^{(n+1)}(1) , & 
 d_{I,2}^{(n)} & = q^{(n)}(1) + (\lambda_{3n+3} + \lambda_{3n+4})q^{(n+1)}(1) 
 + \lambda_{3n+4} \lambda_{3n+6} q^{(n+2)}(1), 
\end{align*}
for $n\in\N_0$.
\end{enumerate}
\end{teo}

\begin{proof}
\begin{enumerate}
 \item 
Lemma \ref{lem:Jacobi_LLU} leads to the following factorization
\begin{align*}
 P_{II} = \sigma_{II} L_1 D_{II,2}^{-1} D_{II,2} L_2 D_{II,1}^{-1} D_{II,1} U \sigma_{II}^{-1},
\end{align*}
where we are going to determine $D_{1,II}$ and $D_{II,2}$ so that the three factors are stochastic matrices; i.e., the $LU$ factorization is stochastic.
We now seek for $D_{II,1}$, in order that the matrix $D_{II,1} U \sigma^{-1}$ be stochastic, that is
$ D_{II,1 }U \sigma_{II}^{-1}\1= \1$, 
so it holds that
\begin{align*}
 \left( \begin{NiceMatrix}
 \lambda_2 & 1 & 0& \Cdots& \\
 0 & \lambda_5 & 1 &\Ddots & \\
 \Vdots& \Ddots & \lambda_8 & 1 &\\
 & & & \Ddots & \Ddots
 \end{NiceMatrix}
 \right) \left( \begin{NiceMatrix}
 B^{(0)}(1) \\
 B^{(1)}(1) \\
 \Vdots \\
 \end{NiceMatrix}
 \right) = \left( \begin{NiceMatrix}
 d_{II,1}^{(0)} \\[2pt]
 d_{II,1}^{(1)}\\
 \Vdots \\
 \end{NiceMatrix}
 \right)
\end{align*}
and we get
$ d_{II,1}^{(n)}= \lambda_{3j+2}B^{(n)}(1) + B^{(n+1)}(1)$.
An important fact to notice here is that the RHS is strictly positive, $d_{II,1}^{(n)}>0$.

Let us find $D_{II,2}$, in order that the matrix $D_{II,2} L_2 D_{II,1}^{-1}$ be stochastic, that is
$D_{II,2} L_2 D_{II,1}^{-1}\1= \1$,
and we get
\begin{align*}
\left( \begin{NiceMatrix}
 1 & 0 &\Cdots & \\[-5pt]
 \lambda_4 & 1 & \Ddots& \\
 0 & \lambda_7 & 1 & \\
 \Vdots& \Ddots & \Ddots & \Ddots 
\end{NiceMatrix}
\right) 
 \left( \begin{NiceMatrix}
 d_{II,1}^{(0)} \\[2pt]
 d_{II,1}^{(1)}\\
 \Vdots \\
 \end{NiceMatrix} \right) 
= \left( \begin{NiceMatrix}
 d_{II,2}^{(0)} \\[2pt]
 d_{II,2}^{(1)}\\
 \Vdots \\
 \end{NiceMatrix}\right).
\end{align*}
Therefore,
\begin{align*}
 d_{II,2}^{(n) } = \lambda_{3n+1} d_{II,1}^{(n-1) } + d_{II,1}^{(n) } 
 & = \lambda_{3n+1}\lambda_{3n-1}B^{(n-1)}(1) + (\lambda_{3n+1} +\lambda_{3n+2} )B^{(n)}(1) + B^{(n+1)}(1).
\end{align*}
Notice that the RHS is strictly positive, $ d_{II,2}^{(n)}>0$.

Finally, $ \sigma_{II} L_1 D_{II,2}^{-1}$ is an stochastic matrix, because all its entries are nonnegative and
\begin{align*}
 \1=P_{II}\1=P_{II,1}^L P_{II,2}^L P_{II}^U\1=P_{II,1}^L\1.
\end{align*}
%
%

\item
Again, from Lemma \ref{lem:Jacobi_LLU} and a transposition we get
\begin{align*}
 P_{I} = \sigma_{I} U^\top L_2^\top L_1^\top \sigma_{I} ^{-1}= \sigma_{I} U^\top D_{I,2}^{-1}D_{I,2} L_2^\top D_{I,1}^{-1}D_{I,1}L_1^\top \sigma_{I} ^{-1}.
\end{align*}
Let us find $D_{I,1}$, in order that the matrix $D_{I,1} L_1^\top \sigma_{I}^{-1}$ be stochastic, that is
$ \tilde{D}_1 L_1^\top \tilde{\sigma}^{-1}\1= \1$.
This is equivalent to
$L_1^\top \tilde{\sigma}^{-1}\1= \tilde{D}_1^{-1}\1$,
and taking the transpose we get
\begin{align*}
 \left( \begin{NiceMatrix} q^{(0)}(1) & q^{(1)}(1) & \Cdots \end{NiceMatrix} \right)
 \left( \begin{NiceMatrix}
 1 & 0 & \Cdots& \\[-5pt]
 \lambda_3 & 1 &\Ddots & \\
 0 & \lambda_6 & 1 & \\[-5pt]
 \Vdots& \Ddots & \lambda_8 & 1 \\
 & & \Ddots & \Ddots & \Ddots 
 \end{NiceMatrix}
 \right) =
 \left( \begin{NiceMatrix}
 d_{I,1}^{(0)} &
 d_{I,1}^{(1)} &
 \Cdots &
 \end{NiceMatrix}
 \right)
\end{align*}
Hence, we deduce that
\begin{align*}
 d_{I,1}^{(n)}= q^{(n)}(1)+\lambda_{3n+3}q^{(n+1)}(1).
\end{align*}
An important fact to notice here is that the RHS is strictly positive, so that $ d_{I,1}^{(n)}>0$.

Let us find $D_{I,2}$, in order that the matrix $\tilde{D}_2 L_2^\top \tilde{D}_1^{-1}$ be stochastic, that is
$D_{I,2} L_2^\top D_{I,1}^{-1}\1=\1$.
This is equivalent to
$ L_2^\top \tilde{D}_1^{-1}\1= \tilde{D}_2^{-1} \1$, and taking the transpose we obtain
\begin{align*}
 \left( \begin{NiceMatrix} d_{I,1}^{(0)} & d_{I,1}^{(1)} & \Cdots \end{NiceMatrix} \right)
 \left( \begin{NiceMatrix}
 1 & 0 & \Cdots& \\[-5pt]
 \lambda_4 & 1 & \Ddots& \\
 0 & \lambda_7 & 1 & \\[-5pt]
 \Vdots & \Ddots& \lambda_{10} & 1 \\
 & & \Ddots & \Ddots & \Ddots 
 \end{NiceMatrix}
 \right)= \left( \begin{NiceMatrix} d_{I,2}^{(0)} & d_{I,2}^{(1)} & \Cdots \end{NiceMatrix} \right).
\end{align*}
Consequently, we find
\begin{align*}
d_{I,2}^{(n)}& = d_{I,1}^{(n)} +\lambda_{3n+4}d_{I,1}^{(j+1)}
 = q^{(n)}(1) + (\lambda_{3n+3} + \lambda_{3n+4})q^{(n+1)}(1) 
 + \lambda_{3n+4} \lambda_{3n+6} q^{(n+2)}(1).
\end{align*}
Notice that the RHS of this identity is strictly positive and $d_{I,2}^{(n)}>0$.

\end{enumerate}
Using the same reasoning as in the previous case we get that the matrix
$ \sigma_I U^\top {D}_{I,2}^{-1} $ is stochastic.
\end{proof}

\begin{rem}
 This stochastic factorization has been motivated by \cite{grunbaum_de la iglesia2} in where an urn model was proposed for the Jacobi--Piñeiro random walks in \cite{bfmaf}. In \cite{grunbaum_de la iglesia2} for the Jacobi--Piñeiro situation they give a factorization $P_{II}=P_LP_U$ with $P_{U}$ an stochastic upper triangular matrix with only the first superdiagonal nonzero, i.e. an stochastic matrix describing a pure birth Markov chain, and a matrix $P_L$ an stochastic lower triangular matrix, with zero as an absorbing state and only the two first subdiagonals nonzero. As there is a nonzero second subdiagonal, following \cite{Gallager,Bremaud} this is not a pure death Markov chain, and it goes beyond it as there are transitions beyond near neighbors.
 
 The stochastic factorization provided here for the hypergeometric situation in three simple stochastic factors is in terms of a pure birth factor and two pure death factors for the type II, and in terms of one pure death factor and two pure birth factors for the type I case.
\end{rem}
\begin{rem}
 The construction of the corresponding urn models, once the stochastic factorization is provided, will be given by an appropriate choice of the hypergeometric parameters $(a,b,c,d)$ and three urns, one urn per factor, with three different experiments. 
\end{rem}
\subsection{Stochastic factorization of the type I Markov matrix}

Surprisingly, the stochastic factorization of the hypergeometric Markov matrix of type I can be carried out leading to extremely simple expressions in terms of quotients of arithmetic progressions in $n$.
\begin{teo}\label{theorem:stochastic factorization}
For $n\in\N_0$, the stochastic factorization \eqref {eq:stochastic_factorization_I}, 
 $P_{I}=P_{I,1}^L P_{I,2}^U P_{I,1}^U$, is explicitly given as follows
 \begin{align*}
\hspace*{-.5cm}\begin{aligned}
	 (P_{I}^L )_{2n,2n-1}&=\frac{2n}{3n-1+d}, & (P_{I}^L )_{2n,2n}&=\frac{n-1+d}{3n-1+d},&
 (P_{I}^L )_{2n+1,2n}&=\frac{2n+1}{3n+1+c}, & (P_{I}^L )_{2n+1,2n+1}&=\frac{n+c}{3n+1+c},\\
 (P_{I,2}^U )_{2n,2n+1}&=\frac{n+d-b}{3n+d}, & (P_{I,2}^U)_{2n,2n}&=\frac{2n+b}{3n+d},&
 (P_{I,2}^U )_{2n+1,2n+2}&=\frac{n+1+c-b}{3n+2+c}, & (P_{I,2}^U)_{2n+1,2n+1}&=\frac{2n+1+b}{3n+2+c},\\
 (P_{I,1}^U )_{2n,2n+1}&=\frac{n+c-a}{3n+c}, & (P_{I,1}^U)_{2n,2n}&=\frac{2n+a}{3n+c},&
 (P_{I,1}^U )_{2n+1,2n+2}&=\frac{n+d-a}{3n+1+d}, & (P_{I,1}^U)_{2n+1,2n+1}&=\frac{2n+1+a}{3n+1+d}.
\end{aligned}
 \end{align*}
Notice that $(P_{I}^L )_{2n,2n-1}$ does not exist. 
The matrices read
{\scriptsize\begin{align*}
 P_{I}^L&=\left(
 \begin{NiceMatrix}[columns-width = 0.3cm]
 1 & 0 &\Cdots & & & & & \\
 \frac{1}{c+1} & \frac{c}{c+1} & \Ddots & & & & & \\[4pt]
 0 & \frac{2}{d+2} & \frac{d}{d+2} & & & & & \\[2pt]
 \Vdots & \Ddots & \frac{3}{c+4} & \frac{c+1}{c+4} & & && \\[2pt]
 & & & \frac{4}{d+5} & \frac{d+1}{d+5} & & & \\[2pt]
 & & & & \frac{5}{c+7} & \frac{c+2}{c+7} & & \\[2pt]
 & & & & & \frac{6}{d+8} & \frac{d+2}{d+8}& \\[2pt]
 &&&&&\Ddots&\Ddots&\Ddots
 \end{NiceMatrix}
 \right),
\end{align*}}
for which zero is an absorbent state, and 
{\scriptsize\begin{align*}
\hspace*{-.25cm} 
\begin{aligned}
 P_{I,2}^U&=\left(
 \begin{NiceMatrix}[columns-width = 0.4cm]
 \frac{b}{d} &\frac{d-b}{d} & 0 & \Cdots& & & & \\
 0 & \frac{b+1}{c+2} & \frac{c-b+1}{c+2} & \Ddots & & & &\\
 \Vdots& \Ddots& \frac{b+2}{d+3} & \frac{d-b+1}{d+3}& & & &\\[2pt]
 & & & \frac{b+3}{c+5} & \frac{c-b+2}{c+5} & & & \\[2pt]
 & & & & \frac{b+4}{d+6} & \frac{d-b+2}{d+6} & & \\[2pt]
 & & & & & \frac{b+5}{c+8} & \frac{c-b+3}{c+8} & \\[2pt]
 & & & & & & \Ddots & \Ddots
 \end{NiceMatrix}
 \right),&
 P_{I,1}^U&=\left(
 \begin{NiceMatrix}[columns-width = 0.4cm]
 \frac{a}{c} &\frac{c-a}{c} & 0 & \Cdots& & & & \\
 0 & \frac{a+1}{d+1} & \frac{d-a}{d+1} & \Ddots & & & & \\
 \Vdots& \Ddots& \frac{a+2}{c+3} & \frac{c-a+1}{c+3}& & & & \\[2pt]
 & & & \frac{a+3}{d+4} & \frac{d-a+1}{d+4} & & & \\[2pt]
 & & & & \frac{a+4}{c+6} & \frac{c-a+2}{c+6} & & \\[2pt]
 & & & & & \frac{a+5}{d+7} & \frac{d-a+2}{d+7} & \\[2pt]
 & & & & & & \Ddots & \Ddots
 \end{NiceMatrix}
 \right).
 \end{aligned}
\end{align*}}
\end{teo}
\begin{proof}
 The proof is an algebraic calculation involving Theorems \ref{teo:linear_forms_at_unity} and \ref{teo:stochastic _factorization} and Equations \eqref{eq:lambdas}.
\end{proof}

\section{Uniform matrices}\label{S:uniform}

In this section we find and study real uniform or almost uniform Jacobi matrices. A uniform Jacobi matrix is understood in this context as a banded Toeplitz matrix, that is constant or uniform by diagonals as in~\eqref{eq:Jacobi_Jacobi_Piñeiro_uniform},~i.e.,
 \begin{align*}
 J= \left(\begin{NiceMatrix}
 3\kappa& 1 & 0 & \Cdots & & \\[-5pt]
 3\kappa^2&3\kappa& 1 & \Ddots& & \\ 
 \kappa^3& 3\kappa^2& 3\kappa& 1 &&\\ 
 0& \kappa^3&3\kappa^2& 3\kappa& 1 & \\
 \Vdots&\Ddots&\Ddots&\Ddots&\Ddots&\Ddots
\end{NiceMatrix}\right).
 \end{align*}
with $\kappa = \frac 4 {27} $, 
 that we called \emph{asymptotic uniform Jacobi matrix}. The \emph{almost asymptotic uniform Jacobi matrix} is understood as the asymptotic uniform Jacobi matrix but for the first two columns. Asymptotic uniform Jacobi matrix appeared in the context of multiple orthogonal polynomials \cite{Coussement_Coussment_VanAssche}, see also \cite{bfmaf} for a discussion for Jacobi--Piñeiro polynomials and random walks, as limiting cases. 
\begin{teo}\label{teo:12}
Almost asymptotic uniform Jacobi matrices happen if and only if the hypergeometric tuple $(a,b,c,d)$ is one of the following twelve tuples 
\begin{align}\label{eq:almost uniform tuples}
\begin{aligned}
 &\Big(\frac 1 3, \frac 2 3, \frac1 2, 1\Big), & &\Big(\frac 2 3, \frac 1 3, \frac1 2, 1\Big), 
 & &\Big(\frac 1 3, \frac 2 3, 1, \frac 3 2\Big), & &\Big( \frac 2 3, \frac 1 3, 1, \frac 3 2\Big) , & &\Big(\frac 2 3, \frac 4 3, 1, \frac 3 2 \Big), & &\Big(\frac 4 3, \frac 2 3, 1, \frac 3 2\Big), 
\\
&\Big(\frac 2 3, \frac 4 3, \frac 3 2 , 2\Big), &&\Big(\frac 4 3, \frac 2 3, \frac 3 2 , 2 \Big), &&
\Big(\frac 4 3, \frac 5 3, \frac 3 2 , 2\Big), &&\Big(\frac 5 3, \frac 4 3, \frac 3 2 , 2\Big ), &&
\Big(\frac 4 3, \frac 5 3, 2 , \frac 5 2\Big) , &&\Big(\frac 5 3, \frac 4 3, 2 , \frac 5 2 \Big),
\end{aligned}
\end{align}
that we will called \emph{uniform  tuples}.
\end{teo}

\begin{proof}
The twelve uniform cases follow, we used Mathematica, from the real positive solution $(a,b,c,d)$ of the system of equations
\begin{align*}
 \beta_1& = 3\kappa , &
 \alpha_2& = 3\kappa^2 , &
 \gamma_2& = \kappa^3 , &
 \beta_2& = 3\kappa , &&,
\end{align*}
where the sequences $(\beta_n)$, $(\alpha_n)$, $(\gamma_n)$ are given in \eqref{eq:jacobi_hyper_coeff} and \eqref{eq:lambdas}.
That is
\begin{align*}
 \frac{2 (a+1) (b+1)}{(c+2) (d+1)}-\frac{a b}{c d} & = 3\kappa , \\
 \frac{2 (a+1) (b+1}{(c+2) (d+1)} ) \left(\frac{a b}{2 (c+1) d}-\frac{2 (a+1) (b+1)}{(c+2) (d+1)}+\frac{3 (a+2) (b+2)}{2 (c+3) (d+2)}\right)& = 3\kappa^2 , \\
 \frac{6 (a+1) (a+2) (b+1) (b+2) c (-a+c+1) (-b+c+1)}{(c+1) (c+2)^2 (c+3)^2 (c+4) (d+1) (d+2) (d+3)} & = \kappa^3 , \\
 \frac{3 (a+2) (b+2)}{(c+3) (d+3)}-\frac{2 (a+1) (b+1)}{(c+2) (d+1)} & = 3\kappa.
\end{align*}
Then, one can check, replacing the given tuples, that for all different cases and for $n\in\N$ one has $\beta_n=3\kappa$, $\alpha_{n+1}=3\kappa^2$ and $\gamma_{n+1}=\kappa^3$
\end{proof}

\begin{coro}
If the  tuple $(a,b,c,d)$ is uniform, 
i.e. it is in the list \eqref{eq:almost uniform tuples}, then so is $(b,a,c,d)$.
\end{coro}

\begin{rem}
Notice that when $a$ and $b$ are permuted in the hypergeometric tuple the weight $w_1$ is not altered but the weight $w_2$ changes to a new weight $\check w_2$. As we discussed in
Theorem~\ref{teo:gauge_freedom}
a given pair of weights $(w_1,w_2)$ uniquely determines the corresponding sequences of orthogonal polynomials of type $II$, $\{B^{(n)}(x)\}_{n=0}^\infty$, the sequence of linear forms of type $I$, $\{Q^{(n)}(x)\}_{n=0}^\infty$ and the Jacobi matrix $J$, up to the transformations $(w_1,w_2)\mapsto (w_1,\check w_2=\alpha w_1+\beta w_2)$ with $\beta\neq 0$ and $\alpha+\beta=1$, what we called a \emph{gauge} transformation. We will show in the sequel that the transformation given by the permutation $(a,b,c,d)\leftrightarrow (b,a,c,d)$ is precisely a manifestation of this \emph{gauge} symmetry.
\end{rem}

\begin{rem}
 According to the previous remark we have only six different almost cubic Toeplitz--Jacobi matrices, sequences of type II orthogonal polynomials and type I linear forms. To each of these sets we have two different sets of weights $(w_1,w_2)$ and $(w_1,\check w_2)$ and correspondingly two different uniform tuples $(a,b,c,d)$and $(b,a,c,d)$.
\end{rem}

To study this \emph{gauge} symmetry is convenient to introduce the following algebraic functions
\begin{align*}
 \vartheta_\pm (x)&:= \sqrt[\leftroot{2}\uproot{2}\scriptstyle 3]{1 \pm \sqrt{1-x}} , & x &\in (0,1).
\end{align*}
\begin{pro}
 The functions $\vartheta_\pm$ satisfies the following relations
\begin{align*}
 \vartheta_+^3 + \vartheta_-^3 &= 2 , & \vartheta_+^3 - \vartheta_-^3 &= 2 \sqrt{1-x} , &
 \vartheta_+ \vartheta_- &= \sqrt[\leftroot{2}\uproot{2}\scriptstyle 3]{x},
\end{align*}
and also
\begin{align}\label{eq:algebraic_functions2}
 (\vartheta_+ + \vartheta_-)(\vartheta_+^2 -\vartheta_+\vartheta_- + \vartheta_-^2) &= 2 ,&
 (\vartheta_+ - \vartheta_-)(\vartheta_+^2 +\vartheta_+\vartheta_- + \vartheta_-^2) &= 2 \sqrt{1-x} ,
\end{align}
\end{pro}
\begin{proof}
Equation\eqref {eq:algebraic_functions2} follow from 
 \begin{align*}
 (\vartheta_+ + \vartheta_-)(\vartheta_+^2 -\vartheta_+\vartheta_- + \vartheta_-^2) &= \vartheta_+^3 + \vartheta_-^3 , &
 (\vartheta_+ - \vartheta_-)(\vartheta_+^2 +\vartheta_+\vartheta_- + \vartheta_-^2) &= \vartheta_+^3 - \vartheta_-^3 ,
 \end{align*}
and the first ones follows directly form the definition of $\vartheta_\pm$.
\end{proof}
These identities will be instrumental in proving some identities between $(W_1,W_2)$ and 
$(W_1,\hat W_2)$.

We now proceed with the study of these matrices. We divide this family of twelve tuples in two sets, the six \emph{stochastic uniform tuples} are
 \begin{align}\label{eq:stochastic uniform tuples}
&\Big(\frac 1 3, \frac 2 3, \frac1 2, 1\Big), &&\Big(\frac 2 3, \frac 1 3, \frac1 2, 1\Big), &&
 \Big( \frac 2 3, \frac 4 3, 1, \frac 3 2\Big) ,&& \Big(\frac 4 3, \frac 2 3, 1, \frac 3 2\Big) , &&
\Big( \frac 4 3, \frac 5 3, \frac 3 2 , 2\Big) , && \Big(\frac 5 3, \frac 4 3, \frac 3 2 , 2 \Big),
 \end{align}
 and the six \emph{semi-stochastic tuples} are
 \begin{align}\label{eq:semi-stochastic uniform tuples}
 & &\Big(\frac 1 3, \frac 2 3, 1, \frac 3 2\Big), & &( \frac 2 3, \frac 1 3, 1, \frac 3 2\Big) , &
 &\Big(\frac 2 3, \frac 4 3, \frac 3 2 , 2\Big), &&\Big(\frac 4 3, \frac 2 3, \frac 3 2 , 2 \Big), &&
 \Big(\frac 4 3, \frac 5 3, 2 , \frac 5 2\Big) , &&\Big(\frac 5 3, \frac 4 3, 2 , \frac 5 2 \Big).
 \end{align}
We will see that the stochastic tuples lead to three stochastic matrices of type I (taking into account the permutation \emph{gauge} symmetry $a\leftrightarrow b$), with corresponding recurrent Markov chains and to semi-stochastic matrices of type II that describe Markov chains with sinks and sources. The semi-stochastic tuples lead to three double semi-stochastic matrices of type I and three more of type II corresponding Markov chains with~sinks.

In what follows we will give for each case, the explicit expressions of the system of weights, the type II multiple orthogonal polynomials, and its values at unity. We will also give the stochastic factorization, and discuss whether there are uniform pure death o birth factors. 
Notice that, as we know the values at unity of type II multiple orthogonal polynomials and type I linear forms, we can apply Theorem \ref{pro:sigma_spectral} in order to construct type II and I stochastic matrices, but unfortunately not uniform, but for three the type I stochastic matrices related to stochastic uniform tuples. This is why we follow an alternative path leading to semi-stochastic uniform matrices. 

 It is also remarkable that the knowledge of the explicit value at unity of the type II multiple orthogonal polynomials leads to following nontrivial summation formulas for the generalized hypergeometric function $\tensor[_3]{F}{_2}(1)$ that, for the reader convenience, we collect together here.
\begin{pro}[Summation formulas at unity]
	The following summation formulas at unity for the generalized hypergeometric function $\tensor[_3]{F}{_2}$ hold true
 \begin{align*}
 		\tensor[_3]{F}{_2}\hspace*{-3pt}\left[{\begin{NiceArray}{c}[small]-n,\; \frac{n+1}{2} , \;\frac{n}{2} \\[3pt]
 			\frac{1}{3},\;\frac{2}{3}\end{NiceArray}};1\right]&=\frac{1+2(-8)^{n}}{3 \times 4^{n}}, &
 		\tensor[_3]{F}{_2}\hspace*{-3pt}\left[{\begin{NiceArray}{c}[small]-n,\; \frac{n+1}{2} , \;\frac{n+2}{2} \\[3pt]
 			\frac{2}{3},\;\frac{4}{3}\end{NiceArray}};1\right]&=
\frac{4 (-8)^{n} - 1}{3 (3 n+1) 4^{n}} ,
 \\
 		\tensor[_3]{F}{_2}\hspace*{-3pt}\left[{\begin{NiceArray}{c}[small]-n,\; \frac{n+3}{2} ,\; \frac{n+2}{2} \\[3pt]
 				\frac{4}{3},\;\frac{5}{3}\end{NiceArray}};1\right]&
 		=\frac{2(1-(-8)^{n+1} )}{9(n+1)(3n+2)4^n},&
 		\tensor[_3]{F}{_2}\hspace*{-3pt}\left[{\begin{NiceArray}{c}[small]-n,\; \frac{n+1}{2} , \; \frac{n+2}{2} \\[3pt]
 				\frac{1}{3},\;\frac{2}{3}\end{NiceArray}};1\right]&
 		=\frac{1+2 (9 n+4)(-8)^n}{9 \times 4^{n}},\\
 		\tensor[_3]{F}{_2}\hspace*{-3pt}\left[{\begin{NiceArray}{c}[small]-n,\; \frac{n+3}{2} , \; \frac{n+2}{2} \\[3pt]
 				\frac{2}{3},\;\frac{4}{3}\end{NiceArray}};1\right]&
 		=\frac{ 4 (9 n+7) (-8)^n -1 }{27 (n+1) 4^{n}},&	\tensor[_3]{F}{_2}\hspace*{-3pt}\left[{\begin{NiceArray}{c}[small]-n,\; \frac{n+4}{2} , \; \frac{n+3}{2} \\[3pt]\frac{4}{3},\;\frac{5}{3}\end{NiceArray}};1\right]&
 		=\frac{2(1- (9 n+10)(-8)^{n+1})}{81 (n+1) (n+2)4^n}.
 \end{align*}
\end{pro}
\subsection{Stochastic uniform tuples. Uniform recurrent random walks }\label{S:stochastic_uniform}

\begin{pro}
 Among all the uniform tuples \eqref{eq:almost uniform tuples} only the stochastic uniform tuples \eqref{eq:stochastic uniform tuples} are such that the type I linear forms at unity satisfies
 \begin{align}\label{eq:linear form at unity uniform}
 Q^{(n)}(1)=\frac 1 {(2\kappa)^n} , && n \in \N .
 \end{align}
\end{pro}
\begin{proof}
 The stochastic uniform tuples are the real positive solution, $a,b,c,d$, of the system of equations
\begin{align*}
 Q^{(1)}(1) & = \frac 1 {2 \kappa} , &
 Q^{(2)}(1) & = \frac 1 {(2\kappa)^2} , &
 Q^{(3)}(1) & = \frac 1 {(2 \kappa)^3} , &
 Q^{(4)}(1) & = \frac 1 {(2\kappa)^4} , 
\end{align*}
where $\{ Q^{(n)} (1)\}_{n=0}^\infty$ is the sequence of type I linear forms at unity given in Theorem \ref{teo:linear_forms_at_unity}. 
One can check that for these stochastic uniform tuples Equation \eqref{eq:linear form at unity uniform} holds.
\end{proof}

\begin{rem}
 From \eqref{eq:linear form at unity uniform} and the procedure in \eqref{eq:stochastic_Jacobi_Piñeiro_I} we obtain a corresponding stochastic matrix of type I, an as all the objects are almost uniform as a result we get almost uniform stochastic matrices. Using the same matrix 
 \begin{align}\label{eq:norma}
 \sigma_{II} := \diag \left(\begin{NiceMatrix} \frac 1{2 \kappa} & \left( \frac 1{2 \kappa} \right)^2 & \Cdots\end{NiceMatrix}\right)
 \end{align}
we can obtain semi-stochastic matrices of type II.
\end{rem}

We now analyze the three cases. In all of them the Jacobi matrix differs from \eqref{eq:Jacobi_Jacobi_Piñeiro_uniform} only in the first column, and $\delta=\frac{1}{2}$ so that the corresponding Markov chains are recurrent. The min-max property \eqref{eq:min-max}, $\max(a,b)>\min(c,d)$ does not hold. Moreover,
the $\alpha_1$ in the proof of \cite[Theorem 2.1]{lima_loureiro} now is not positive, and the proof of the Nikishin property of the weight fails.
The coefficients $H_{2n}$ are positive and $H_{2n+1}$ are negative, but they never cancel, therefore the Gauss--Borel factorization of the moment matrix holds. Therefore, the system of weights $(W_1,W_2,\d x)$ and $(W_1,\check W_2,\d\mu)$ are perfect. From \eqref{eq:lambdas} we find that the Jacobi matrix is nonnegative whenever
$a<d$, $b<d$, $a<c$ and $b-1<c$. When the condition $\max(a,b)<\min(c,d)$ is meet the Jacobi matrix is a nonnegative case. However, for
$ a<d$, $b<d$, $a<c$ and $b-1<c<b$ we also get a nonnegative Jacobi matrix but \eqref{eq:min-max} is not fulfilled. 

The type I transition matrix $P_I$ is stochastic and uniform, and have stochastic factorizations. Only one pair of tuples, $\big(\frac 1 3, \frac 2 3, \frac1 2, 1\big), \big(\frac 2 3, \frac 1 3, \frac1 2, 1\big)$ have a uniform stochastic factorization. The type II transition matrix~$P_{II}$ is semi-stochastic having three states, the one with uniform stochastic factorization, or two states, the two remaining ones, which are sinks or sources. In a sink probability is destroyed, and in a source probability is created. In these case, the overall destroyed and created probability balance to zero. 

\paragraph{\textbf{The stochastic uniform tuples }$\big(\frac 1 3, \frac 2 3, \frac1 2, 1\big), \big(\frac 2 3, \frac 1 3, \frac1 2, 1\big)$}

In this case the Jacobi matrix and the type I stochastic matrix are
\begin{align*}
 J & =
 \left(\begin{NiceMatrix}
 3\kappa& 1 & 0 & \Cdots & & \\[-5pt]
 6\kappa^2&3\kappa& 1 & \Ddots& & \\ 
 3\kappa^3& 3\kappa^2& 3\kappa& 1 &&\\ 
 0& \kappa^3&3\kappa^2& 3\kappa& 1 & \\
 \Vdots&\Ddots&\Ddots&\Ddots&\Ddots&\Ddots
\end{NiceMatrix}\right),
&
 P_I&= \left(\begin{NiceMatrix}[columns-width =auto]
 \frac{12}{27} &\frac{12}{27} &\frac{3}{27}& 0 & \Cdots& & \\
 \frac{8}{27} & \frac{12}{27} &\frac{6}{27} &\frac{1}{27} & \Ddots & & \\
 0& \frac{8}{27} & \frac{12}{27} &\frac{6}{27} &\frac{1}{27} & \Ddots & \\
 \Vdots& \Ddots& \frac{8}{27} & \frac{12}{27} &\frac{6}{27} &\frac{1}{27} & \\
 & & & \Ddots& \Ddots & \Ddots&\Ddots
 \end{NiceMatrix}
 \right),
\end{align*}
with corresponding diagram 
\begin{center}
 \tikzset{decorate sep/.style 2 args={decorate,decoration={shape backgrounds,shape=circle,shape size=#1,shape sep=#2}}}
 \begin{tikzpicture}[start chain = going right,
 -latex, every loop/.append style = {-latex}]\small
 \foreach \i in {0,...,5}
 \node[state, on chain,fill=gray!50!white] (\i) {\i};
 \foreach
 \i/\txt in {1/$\frac{6}{27}$/,2/$\frac{6}{27}$,3/$\frac{6}{27}$,4/$\frac{6}{27}$}
 \draw let \n1 = { int(\i+1) } in
 (\i) edge[line width=.22 mm,bend left,"\txt",below,color=Periwinkle] (\n1);
 
 \draw (0) edge[line width=.44 mm,bend left,"$\frac{12}{27}$",below,color=Periwinkle] (1);
 
 \foreach
 \i/\txt in {0/$\frac{8}{27}$,1/$\frac{8}{27}$/,2/$\frac{8}{27}$,3/$\frac{8}{27}$,4/$\frac{8}{27}$}
 \draw let \n1 = { int(\i+1) } in
 (\n1) edge[line width=.3 mm,bend left,below, "\txt",color=Mahogany,auto=right] (\i);
 
 \foreach
 \i/\txt in {1/$\frac{1}{27}$,1/$\frac{1}{27}$/,2/$\frac{1}{27}$,3/$\frac{1}{27}$}
 \draw let \n1 = { int(\i+2) } in
 (\i) edge[line width=.04 mm,bend left=65,color=MidnightBlue,"\txt"](\n1);
 
 \draw (0) edge[line width=.11 mm,bend left=65,color=MidnightBlue,"$\frac{3}{27}$"](2);
 
 \foreach \i/\txt in {1/$\frac{12}{27}$,2/$\frac{12}{27}$/,3/$\frac{12}{27}$,4/$\frac{12}{27}$,5/$\frac{12}{27}$}
 \draw (\i) edge[line width=.44 mm,loop below, color=NavyBlue,"\txt"] (\i);
 \draw (0) edge[line width=.44 mm,loop left, color=NavyBlue,"$\frac{12}{27}$"] (0);
 
 \draw[decorate sep={1mm}{4mm},fill] (10,0) -- (12,0);
 \end{tikzpicture}
 \begin{tikzpicture}
 \draw (4,-1.8) node
 {\begin{minipage}{0.8\textwidth}
 \begin{center}\small
 \textbf{Uniform type I Markov chain }
 \end{center}
 \end{minipage}};
 \end{tikzpicture}
\end{center}

The system of weights $(W_1,W_2,\d x)$ and $(W_1,\hat W_2,\d x)$ corresponding to the stochastic uniform tuples $\big(\frac 1 3, \frac 2 3, \frac1 2, 1\big)$ and $\big(\frac 2 3, \frac 1 3, \frac1 2, 1\big)$ are, respectively, 
\begin{align*}
 W_1 (x) & = \frac{\sqrt{3} (\vartheta_+(x)+\vartheta_-(x))}{4 \pi \sqrt[\leftroot{2}\uproot{2} 3]{x^{2}} \sqrt{1-x} },& 
 W_2 (x) & = \frac{3 \sqrt{3} (\vartheta_+^4(x)+\vartheta_-^4(x))}{16 \pi \sqrt[\leftroot{2}\uproot{2}3]{x^{2}}\sqrt{1-x} } , &
 \hat W_2 (x) & = \frac{3 \sqrt{3} ( \vartheta_+^2(x)+\vartheta_-^2(x))}{8 \pi \sqrt[\leftroot{2}\uproot{2} 3]{x}\sqrt{1-x} } .
\end{align*}
Solving the system of equations \eqref{eq:rho}
we find that $ \alpha = \frac{3}{2}$ and $\beta = -\frac{1}{2}$.
After some simplifications we get 
\begin{align*}
 \frac{3}{2} W_1 -\frac{1}{2} \hat W_2 & =3 \sqrt{3} \,
 \frac{2 (\vartheta_+(x)+\vartheta_-(x))+\sqrt[\leftroot{2}\uproot{2} 3]{x}(\vartheta_+^2(x)-\vartheta_-^2(x)) }{16 \pi \sqrt[\leftroot{2}\uproot{2} 3]{x^2} \sqrt{1-x} } ,
\end{align*}
but the numerator can be written as
$ (\vartheta_+ + \vartheta_-)^2 (\vartheta_+^2 - \vartheta_+ \vartheta_- + \vartheta_-^2) - (\vartheta_+^2 + \vartheta_-^2) \vartheta_+ \vartheta_- = \vartheta_+^4 + \vartheta_-^4 $,
and, consequently, we obtain 
\begin{align*}
 \frac{3}{2} W_1 -\frac{1}{2} \hat W_2 & = \frac{3 \sqrt{3}}{16 \pi \sqrt[\leftroot{2}\uproot{2} 3]{x^2}\sqrt{1-x}} \big( \vartheta_+^4 + \vartheta_-^4 \big) = W_2,
\end{align*}
i.e., $\hat W_2=3 W_1-2W_2$, and both set of measures are in the same \emph{gauge} class.

The type II multiple orthogonal polynomials are
\begin{align}\label{eq:B_3F2_1}
 B^{(n)}(x) & = 3(-\kappa)^n
 \tensor[_3]{F}{_2}\hspace*{-3pt}\left[{\begin{NiceArray}{c}[small]-n,\; \frac{n+1}{2} , \;\frac{n}{2} \\[3pt]
 \frac{1}{3},\;\frac{2}{3}\end{NiceArray}};x\right] , && n \in \N . 
\end{align}
To evaluate these polynomials at unity notice that the following constant coefficients  order four  homogeneous recurrence is satisfied
$ B^{(n)} (1) = B^{(n+1)} (1)+3 \kappa B^{(n)} (1) + 3\kappa^2 B^{(n-1)} (1) + \kappa^3 B^{(n-2)} (1)$,
that has characteristic polynomial in $p(t)=(\kappa+t)^3-t^2=(t-2\kappa)^2\big(t+\frac{\kappa}{4}\big)$. Hence,
$	B^{(n) }(1) = \Big(-\frac{\kappa}{4}\Big)^n c_1+\big(2\kappa\big)^n (c_2+n c_3)$,
for appropriate constants $\{c_1,c_2,c_3\}$ such that the initial conditions 
$B^{(1)} (1) = \frac{5}{9}$, $B^{(2)} (1) = \frac{43}{243}$ and $B^{(3)} (1)=\frac{341}{6561}$
are satisfied. Hence,
\begin{align}\label{eq:Bn(1)}
	B^{(n)}(1)&= \frac{2\times 8^n+(-1)^n}{27^{n}}, & 
	\tensor[_3]{F}{_2}\hspace*{-3pt}\left[{\begin{NiceArray}{c}[small]-n,\; \frac{n+1}{2} , \;\frac{n}{2} \\[3pt]
			\frac{1}{3},\;\frac{2}{3}\end{NiceArray}};1\right]&=\frac{1+2(-8)^{n}}{3 \times 4^{n}}, &
	n&\in\N.
\end{align}

The stochastic factorization of the type I Markov matrix $P_I=P_{I}^L P_{I,2}^U P_{I,1}^U$
\begin{align*}P_{I}^L&=\left(
 \begin{NiceMatrix}[columns-width = 0.1cm]
 1 & 0 &\Cdots & & \\
 \frac{2}{3} & \frac{1}{3} & \Ddots & & \\
 0 & \frac{2}{3} & \frac{1}{3} & & \\
 \Vdots & \Ddots & \frac{2}{3} & \frac{1}{3} & \\
 &&\Ddots&\Ddots&\Ddots
 \end{NiceMatrix}
 \right), & P_{I,1}^U&=P_{I,2}^U= \left(\begin{NiceMatrix}[columns-width =auto]
 \frac{2}{3} &\frac{1}{3} & 0 & \Cdots& & \\
 0 & \frac{2}{3} & \frac{1}{3} & \Ddots & & \\
 \Vdots& \Ddots& \frac{2}{3} & \frac{1}{3} & & \\[2pt]
 & & &\frac{2}{3} & \frac{1}{3}& \\[2pt]
 & & & & \Ddots & \Ddots
 \end{NiceMatrix}\right).
\end{align*}
We have a uniform stochastic factorization, all factors are uniform. This is a very peculiar property among the hypergeometric type I Markov matrices. 
\begin{teo}\label{teo:uniform_stochastic_factorization}
There is only one hypergeometric type stochastic matrix $P_I$ with a uniform stochastic factorization, and its hypergeometric tuples are the uniform stochastic tuples $\big(\frac 1 3, \frac 2 3, \frac1 2, 1\big)$ and $\big(\frac 2 3, \frac 1 3, \frac1 2, 1\big)$.
\end{teo}
\begin{proof}
 Here we use Theorem \ref{theorem:stochastic factorization} where the factorization into pure birth and death factors $P_I=LU_1, U_2$ was given. 
 The even rows of the pure death factor $P^{L}_I$ are uniform if and only if $d=1$, while the odd rows are uniform if and only if $c=\frac{1}{2}$:
 \begin{align*} 
 \begin{aligned}
 P_I^L\big|_{d=1}&=\left(\begin{NiceMatrix}
 1 & 0 &\Cdots & & & & & \\
 \frac{1}{c+1} & \frac{c}{c+1} & \Ddots & & & & & \\[1pt]
 0 & \frac{2}{3} & \frac{1}{3} & & & & & \\
 \Vdots & \Ddots & \frac{3}{c+4} & \frac{c+1}{c+4} & & && \\[1pt]
 & & & \frac{2}{3} & \frac{1}{3} & & & \\[1pt]
 & & & & \frac{5}{c+7} & \frac{c+2}{c+7} && \\[1pt]
 & & & & & \frac{2}{3} & \frac{1}{3}& \\[1pt]
 &&&&&\Ddots&\Ddots&\Ddots
 \end{NiceMatrix}
 \right), & P_I^L\big|_{c=\frac{1}{2}}&=\left(
 \begin{NiceMatrix}
 1 & 0 &\Cdots & & & & & \\
 \frac{2}{3} & \frac{1}{3} & \Ddots & & & & & \\[1pt]
 0 & \frac{2}{d+2} & \frac{d}{d+2} & & & & & \\
 \Vdots & \Ddots & \frac{2}{3} & \frac{1}{3} & & & & \\[1pt]
 & & & \frac{4}{d+5} & \frac{d+1}{d+5} & & & \\[1pt]
 & & & & \frac{2}{3} & \frac{1}{3} & & \\[1pt]
 & & & & & \frac{6}{d+8} & \frac{d+2}{d+8}& \\
 &&&&&\Ddots&\Ddots&\Ddots
 \end{NiceMatrix}
 \right).
 \end{aligned}
 \end{align*}
 Therefore, the pure death factor is uniform if and only if $c=\frac{1}{2}$ and $d=1$ and
 $ P_L:= P_I^L\big|_{c=\frac{1}{2},d=1}=\left(
 \begin{NiceMatrix}[small]
 1 & 0 &\Cdots & \\[-3pt]
 \frac{2}{3} & \frac{1}{3} & \Ddots & \\[3pt]
 0 & \frac{2}{3} & \frac{1}{3} & &\\
 \Vdots & \Ddots & \Ddots& \Ddots \\
 \end{NiceMatrix}
 \right)$.
 For the second pure birth factor $P_{I,2}^U$ we have uniformity on even rows if and only if $3b-2d=0$ and on the odd rows if and only if $3b-2c=1$:
 {\scriptsize\begin{align*}
 \hspace*{-.5cm}\begin{aligned}
 P_{I,2}^U\big|_{3b-2d=0}&=\left(
 \begin{NiceMatrix}[columns-width = 0.3cm]
 \frac{2}{3} &\frac{1}{3} & 0 & \Cdots& & & & \\
 0 & \frac{b+1}{c+2} & \frac{c-b+1}{c+2} & \Ddots & & & &\\
 \Vdots& \Ddots& \frac{2}{3} & \frac{1}{3}& & & &\\[2pt]
 & & & \frac{b+3}{c+5} & \frac{c-b+2}{c+5} & & & \\[2pt]
 & & & & \frac{2}{3}& \frac{1}{3}& & \\[2pt]
 & & & & & \frac{b+5}{c+8} & \frac{c-b+3}{c+8} & \\[2pt]
 & & & & & & \Ddots & \Ddots
 \end{NiceMatrix}
 \right),&
 P_{I,2}^U\big|_{3b-2c=1}&=\left(
 \begin{NiceMatrix}[columns-width = 0.3cm]
 \frac{b}{d} &\frac{d-b}{d} & 0 & \Cdots& & & & \\
 0 & \frac{2}{3} & \frac{1}{3} & \Ddots & & & & \\
 \Vdots& \Ddots& \frac{b+2}{d+3} & \frac{d-b+1}{d+3}& & & &\\[2pt]
 & & &\frac{2}{3} & \frac{1}{3}& & & \\[2pt]
 & & & & \frac{b+4}{d+6} & \frac{d-b+2}{d+6} & & \\[2pt]
 & & & & & \frac{2}{3} & \frac{1}{3}& \\[2pt]
 & & & & & & \Ddots & \Ddots
 \end{NiceMatrix}
 \right).
 \end{aligned}
 \end{align*}}
 In fact, this factor is uniform if only if $3b-2c=1$ and $3b-2d=0$, and
 $ P_U:= P_{I,2}^U\bigg|_{\substack{3b-2d=0\\
 3b-2c=1 }}=\left(
 \begin{NiceMatrix}[small]
 \frac{2}{3} &\frac{1}{3} & 0 & \Cdots \\
 0 & \frac{2}{3} & \frac{1}{3} & \Ddots \\
 \Vdots& \Ddots& \Ddots& \Ddots 
 \end{NiceMatrix}
 \right)$.
 For the first pure birth factor we have uniformity on the even rows if and only if $3a-2c=0$ and on odd rows if and only if $3a-2d=-1$:
 {\scriptsize\begin{align*}
 \hspace*{-.35cm}\begin{aligned}
 P_{I,1}^U\big|_{3a-2c=0}&=\left(
 \begin{NiceMatrix}[columns-width = 0.3cm]
 \frac{2}{3} & \frac{1}{3}& 0 & \Cdots& & & & \\
 0 & \frac{a+1}{d+1} & \frac{d-a}{d+1} & \Ddots & & & &\\
 \Vdots& \Ddots& \frac{2}{3} & \frac{1}{3}& & & & \\[2pt]
 & & & \frac{a+3}{d+4} & \frac{d-a+1}{d+4} & & & \\[2pt]
 & & & & \frac{2}{3} & \frac{1}{3}& & \\[2pt]
 & & & & & \frac{a+5}{d+7} & \frac{d-a+2}{d+7} & \\[2pt]
 & & & & & & \Ddots & \Ddots
 \end{NiceMatrix}
 \right),& P_{I,1}^U\big|_{3a-2d=-1}&=\left(
 \begin{NiceMatrix}[columns-width = 0.3cm]
 \frac{a}{c} &\frac{c-a}{c} & 0 & \Cdots& & & & \\
 0 & \frac{2}{3} & \frac{1}{3}& \Ddots & & & &\\
 \Vdots& \Ddots& \frac{a+2}{c+3} & \frac{c-a+1}{c+3}& & & &\\[2pt]
 & & & \frac{2}{3} & \frac{1}{3}& & & \\[2pt]
 & & & & \frac{a+4}{c+6} & \frac{c-a+2}{c+6} & & \\[2pt]
 & & & & & \frac{2}{3} & \frac{1}{3}& \\[2pt]
 & & & & & & \Ddots & \Ddots
 \end{NiceMatrix}
 \right).
 \end{aligned}
 \end{align*}}
 This factor is uniform if and only if $3b-2c=1$ and $3b-2d=0$ with 
$ P_{I,2}^U\Big|_{\substack{3a-2c=0,\\
 3a-2d=-1 }}=P_U$.
 Hence, both pure birth factors $P^U_{I,1}$ and $P^U_{I,2}$ are uniform whenever we have 
 $3b-2d=0$, $3b-2c=1$, $3a-2c=0$ and $3a-2d=-1$
 whose solution is
 $(a,b,c,d)= \big( a, a+\frac{1}{3}, \frac{3}{2}a, \frac{3}{2}a+\frac{1}{2}\big)$. Then, to have $d=1$ we need $a=\frac{1}{3}$ and we get the result.
\end{proof}

The corresponding type II transition matrix 
$ P_{\text{II}} 
 = \left(\begin{NiceMatrix}[small]
 \frac{12}{27}& \frac{8}{27} & 0 & \Cdots & & \\ 
 \frac {12} {27}& \frac{12}{27}& \frac{8}{27} & \Ddots & & \\[5pt]
 \frac{3}{27} & \frac {6} {27}& \frac{12}{27}& \frac{8}{27} &&\\[5pt] 
 0& \frac{1}{27} &\frac {6} {27}& \frac{12}{27}& \frac{8}{27} & \\ 
 \Vdots&\Ddots&\Ddots&\Ddots&\Ddots&\Ddots
 \end{NiceMatrix}\right)
$
is stochastic but for the first three rows, this models a recurrent random walk where the first state is a sink and the second and third states are sources. The diagram if this random walk is
\begin{center}
 \tikzset{decorate sep/.style 2 args={decorate,decoration={shape backgrounds,shape=circle,shape size=#1,shape sep=#2}}}
 \begin{tikzpicture}[start chain = going right,
 -latex, every loop/.append style = {-latex}]\small
 \foreach \i in {0,...,5}
 \node[state, on chain,fill=gray!50!white] (\i) {\i};
 \foreach
 \i/\txt in {0/$\frac{8}{27}$,1/$\frac{8}{27}$/,2/$\frac{8}{27}$,3/$\frac{8}{27}$,4/$\frac{8}{27}$}
 \draw let \n1 = { int(\i+1) } in
 (\i) edge[line width=.3 mm,bend left,"\txt",color=Periwinkle] (\n1);
 \foreach
 \i/\txt in {1/$\frac{6}{27}$/,2/$\frac{6}{27}$,3/$\frac{6}{27}$,4/$\frac{6}{27}$}
 \draw let \n1 = { int(\i+1) } in
 (\n1) edge[line width=.22 mm,bend left=50,above, "\txt",color=Mahogany,auto=right] (\i);
 \draw (1) edge[line width=.44 mm,bend left=50,above, auto=right,color=Mahogany,"$\frac{12}{27}$"] (0);
 
 \foreach
 \i/\txt in {1/$\frac{1}{27}$/,2/$\frac{1}{27}$,3/$\frac{1}{27}$}
 \draw let \n1 = { int(\i+2) } in
 (\n1) edge[line width=.04 mm,bend left=60,color=RawSienna,"\txt"] (\i);
 \draw (2) edge[line width=.12 mm,bend left=50,above,color=RawSienna,"$\frac{12}{27}$"] (0);
 
 \foreach \i/\txt in {1/$\frac{12}{27}$,2/$\frac{12}{27}$/,3/$\frac{12}{27}$,4/$\frac{12}{27}$,5/$\frac{12}{27}$}
 \draw (\i) edge[line width=.44 mm,loop above,color=NavyBlue, "\txt"] (\i);
 \draw (0) edge[line width=.44 mm,loop left, color=NavyBlue,"$\frac{12}{27}$"] (0);
 
 \draw[decorate sep={1mm}{4mm},fill] (10,0) -- (12,0);
 \end{tikzpicture}
 \begin{tikzpicture}
 \draw (4,-1.8) node
 {\begin{minipage}{0.8\textwidth}
 \begin{center}\small
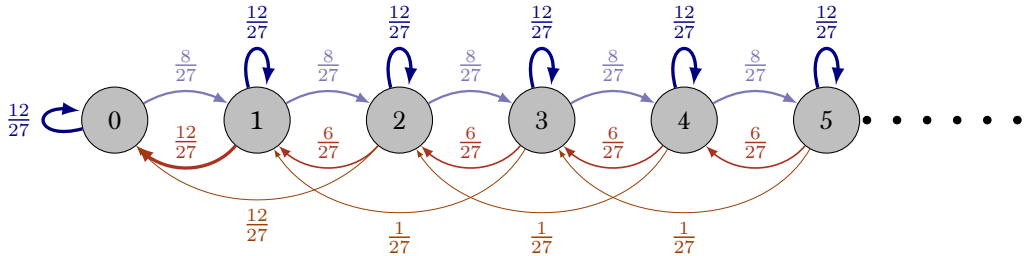
 \textbf{Uniform type II Markov chain with one sink and two sources}
 \end{center}
 \end{minipage}};
 \end{tikzpicture}
\end{center}

\paragraph{\textbf{The stochastic uniform tuples }
$\Big(\frac 2 3, \frac 4 3, 1, \frac 3 2 \Big)$ and $\Big( \frac 4 3, \frac 2 3, 1, \frac 3 2 \Big)$}

In this case the Jacobi matrix and the type I stochastic matrix are
\begin{align*}
J & =
\left(\begin{NiceMatrix}
 4\kappa& 1 & 0 & \Cdots & & \\[-5pt]
 5\kappa^2&3\kappa& 1 & \Ddots& & \\ 
 \kappa^3& 3\kappa^2& 3\kappa& 1 &&\\ 
 0& \kappa^3&3\kappa^2& 3\kappa& 1 & \\
 \Vdots&\Ddots&\Ddots&\Ddots&\Ddots&\Ddots
\end{NiceMatrix}\right),
&
 P_I&= \left(\begin{NiceMatrix}[columns-width =auto]
 \frac{16}{27} &\frac{10}{27} &\frac{1}{27}& 0 & \Cdots& & \\
 \frac{8}{27} & \frac{12}{27} &\frac{6}{27} &\frac{1}{27} & \Ddots & & \\
 0& \frac{8}{27} & \frac{12}{27} &\frac{6}{27} &\frac{1}{27} & \Ddots & \\
 \Vdots& \Ddots& \frac{8}{27} & \frac{12}{27} &\frac{6}{27} &\frac{1}{27} & \\
 & & & \Ddots& \Ddots & \Ddots&\Ddots
 \end{NiceMatrix}
 \right),
\end{align*}
 
with corresponding diagram 
\begin{center}
 \tikzset{decorate sep/.style 2 args={decorate,decoration={shape backgrounds,shape=circle,shape size=#1,shape sep=#2}}}
 \begin{tikzpicture}[start chain = going right,
 -latex, every loop/.append style = {-latex}]\small
 \foreach \i in {0,...,5}
 \node[state, on chain,fill=gray!50!white] (\i) {\i};
 \foreach
 \i/\txt in {1/$\frac{6}{27}$/,2/$\frac{6}{27}$,3/$\frac{6}{27}$,4/$\frac{6}{27}$}
 \draw let \n1 = { int(\i+1) } in
 (\i) edge[line width=.22 mm,bend left,"\txt",below,color=Periwinkle] (\n1);
 
 \draw (0) edge[line width=.37 mm,bend left,"$\frac{10}{27}$",below,color=Periwinkle] (1);
 
 \foreach
 \i/\txt in {0/$\frac{8}{27}$,1/$\frac{8}{27}$/,2/$\frac{8}{27}$,3/$\frac{8}{27}$,4/$\frac{8}{27}$}
 \draw let \n1 = { int(\i+1) } in
 (\n1) edge[line width=.3 mm,bend left,below, "\txt",color=Mahogany,auto=right] (\i);
 
 \foreach
 \i/\txt in {0/$\frac{1}{27}$,1/$\frac{1}{27}$,1/$\frac{1}{27}$/,2/$\frac{1}{27}$,3/$\frac{1}{27}$}
 \draw let \n1 = { int(\i+2) } in
 (\i) edge[line width=.04 mm,bend left=65,color=MidnightBlue,"\txt"](\n1); 
 
 \foreach \i/\txt in {1/$\frac{12}{27}$,2/$\frac{12}{27}$/,3/$\frac{12}{27}$,4/$\frac{12}{27}$,5/$\frac{12}{27}$}
 \draw (\i) edge[line width=.44 mm,loop below, color=NavyBlue,"\txt"] (\i);
 \draw (0) edge[line width=.49 mm,loop left, color=NavyBlue,"$\frac{16}{27}$"] (0);
 
 \draw[decorate sep={1mm}{4mm},fill] (10,0) -- (12,0);
 \end{tikzpicture}
 \begin{tikzpicture}
 \draw (4,-1.8) node
 {\begin{minipage}{0.8\textwidth}
 \begin{center}\small
 \textbf{Uniform type I Markov chain }
 \end{center}
 \end{minipage}};
 \end{tikzpicture}
\end{center}
The system of weights $(W_1,W_2,\d x)$ and $(W_1,\hat W_2,\d x)$ corresponding to the stochastic uniform tuples $\big(\frac 2 3, \frac 4 3, 1,\frac3 2\big)$ and $\big(\frac 4 3, \frac 2 3, 1, \frac3 2,\big)$ are, respectively, 
\begin{align*}
 W_1 (x) & = \frac{3\sqrt{3} (\vartheta_+^2(x)+\vartheta_-^2(x))}{8 \pi \sqrt[\leftroot{2}\uproot{2} 3]{x} \sqrt{1-x} },& 
 W_2 (x) & = \frac{9\sqrt{3} (\vartheta_+^5(x)+\vartheta_-^5(x))}{32 \pi \sqrt[\leftroot{2}\uproot{2}3]{x}\sqrt{1-x} } , &
 \hat W_2 (x) & = \frac{9 \sqrt{3} \sqrt[\leftroot{2}\uproot{2} 3]{x}( (\vartheta_+(x)+\vartheta_-(x)))}{16\pi \sqrt{1-x} } .
\end{align*}
Solving the system of equations
\eqref{eq:rho}
we find that $ \alpha = \frac{3}{2}$ and $\beta = -\frac{1}{2}$.
After some simplifications we get 
\begin{align*}
 \frac{3}{2} W_1 -\frac{1}{2} \hat W_2 & =9 \sqrt{3} \,
 \frac{2 (\vartheta_+^2(x)+\vartheta_-^2(x))-\sqrt[\leftroot{2}\uproot{2} 3]{x^2}(\vartheta_+(x)+\vartheta_-(x)) }{32 \pi \sqrt[\leftroot{2}\uproot{2} 3]{x} \sqrt{1-x} } ,
\end{align*}
but the numerator can be written as
$ (\vartheta_+ + \vartheta_-) (\vartheta_+^2 - \vartheta_+ \vartheta_- + 
\vartheta_-^2) (\vartheta_-^2 + \vartheta_+^2) - (\vartheta_- + 
\vartheta_+) (\vartheta_+ \vartheta_-)^2$,
and, consequently, we obtain 
\begin{align*}
 \frac{3}{2} W_1 -\frac{1}{2} \hat W_2 & = \frac{9 \sqrt{3}}{32 \pi \sqrt[\leftroot{2}\uproot{2} 3]{x}\sqrt{1-x}} \big( \vartheta_+^5+ \vartheta_-^5\big) = W_2,
\end{align*}
i.e., $\hat W_2=3 W_1-2W_2$, and both sets of weights are in the same \emph{gauge} class.

The type II multiple orthogonal polynomials are
\begin{align}\label{eq:B_3F2_2}
 B^{(n)}(x) & = 
(3 n+1) \left( 
- \kappa \right)^n
 \, \tensor[_3]{F}{_2}\hspace*{-3pt}\left[{\begin{NiceArray}{c}[small]-n,\; \frac{n+1}{2} , \;\frac{n+2}{2} \\[3pt]
 \frac{2}{3},\;\frac{4}{3}\end{NiceArray}};x\right], && n \in \N . 
\end{align}

Following similar arguments as for the deduction of \eqref{eq:Bn(1)}, we get the following values at unity of the type II multiple orthogonal polynomials and generalized hypergeometric functions 
\begin{align*}
	B^{(n)}(1)&= \frac{(-1)^{n+1}+4 \times 8^{n}}{3 \times 27^{n}}
 , 
& 
\tensor[_3]{F}{_2}\hspace*{-3pt}\left[{\begin{NiceArray}{c}[small]-n,\; \frac{n+1}{2} , \;\frac{n+2}{2} \\[3pt]
		\frac{2}{3},\;\frac{4}{3}\end{NiceArray}};1\right]&=
\frac{4 (-1)^{-n} 8^{n} - 1}{3 (3 n+1) 4^{n}}
 , 
 &
	n&\in\N.
\end{align*}

The stochastic factorization of the type I Markov matrix $P_I=P_{I}^L P_{I,2}^U P_{I,1}^U$
\begin{align*}
\hspace{-.35cm}
P_{I}^L&=\left(
 \begin{NiceMatrix}[columns-width = 0.1cm]
 1 & 0 &\Cdots & & \\
 \frac{1}{2} & \frac{1}{2} & \Ddots & & \\
 0 & \frac{4}{7} & \frac{3}{7} & & \\
 \Vdots & \Ddots & \frac{3}{5} & \frac{2}{5} & \\
 &&\Ddots&\Ddots&\Ddots
 \end{NiceMatrix}
 \right), & P_{I,2}^U&= \left(\begin{NiceMatrix}[columns-width =auto]
 \frac{2}{3} &\frac{1}{3} & 0 & \Cdots& & \\
 0 & \frac{2}{3} & \frac{1}{3} & \Ddots & & \\
 \Vdots& \Ddots& \frac{2}{3} & \frac{1}{3} & & \\[2pt]
 & & &\frac{2}{3} & \frac{1}{3}& \\[2pt]
 & & & & \Ddots & \Ddots
 \end{NiceMatrix}\right),& P_{I,I}^U&= \left(\begin{NiceMatrix}[columns-width =auto]
 \frac{8}{9} &\frac{1}{9} & 0 & \Cdots& & \\
 0 & \frac{7}{9}&\frac{2}{9}& \Ddots & & \\
 \Vdots& \Ddots& \frac{20}{27} & \frac{7}{27} & & \\[2pt]
 & & &\frac{13}{18} & \frac{5}{18}& \\[2pt]
 & & & & \Ddots & \Ddots
\end{NiceMatrix}\right).
\end{align*}
There is only a uniform factor, one of the two pure births.
The type II transition matrix 
 $P_{\text{II}} 
 = \left(\begin{NiceMatrix}[small]
 \frac{16}{27}& \frac{8}{27} & 0 & \Cdots & & \\ 
 \frac {10} {27}& \frac{12}{27}& \frac{8}{27} & \Ddots & & \\[5pt]
 \frac{1}{27} & \frac {6} {27}& \frac{12}{27}& \frac{8}{27} &&\\[5pt] 
 0& \frac{1}{27} &\frac {6} {27}& \frac{12}{27}& \frac{8}{27} & \\ 
 \Vdots&\Ddots&\Ddots&\Ddots&\Ddots&\Ddots
 \end{NiceMatrix}\right)
$
 is stochastic but for the two first rows, representing a recurrent random walk with the first state being a sink, and the second a source. The diagram of this random walk is
\begin{center}
 \tikzset{decorate sep/.style 2 args={decorate,decoration={shape backgrounds,shape=circle,shape size=#1,shape sep=#2}}}
 \begin{tikzpicture}[start chain = going right,
 -latex, every loop/.append style = {-latex}]\small
 \foreach \i in {0,...,5}
 \node[state, on chain,fill=gray!50!white] (\i) {\i};
 \foreach
 \i/\txt in {0/$\frac{8}{27}$,1/$\frac{8}{27}$/,2/$\frac{8}{27}$,3/$\frac{8}{27}$,4/$\frac{8}{27}$}
 \draw let \n1 = { int(\i+1) } in
 (\i) edge[line width=.3 mm,bend left,"\txt",color=Periwinkle] (\n1);
 \foreach
 \i/\txt in {1/$\frac{6}{27}$/,2/$\frac{6}{27}$,3/$\frac{6}{27}$,4/$\frac{6}{27}$}
 \draw let \n1 = { int(\i+1) } in
 (\n1) edge[line width=.22 mm,bend left=50,above, "\txt",color=Mahogany,auto=right] (\i);
 \draw (1) edge[line width=.37 mm,bend left=50,above, auto=right,color=Mahogany,"$\frac{10}{27}$"] (0);
 
 \foreach
 \i/\txt in {1/$\frac{1}{27}$/,2/$\frac{1}{27}$,3/$\frac{1}{27}$}
 \draw let \n1 = { int(\i+2) } in
 (\n1) edge[line width=.04 mm,bend left=60,color=RawSienna,"\txt"] (\i);
 \draw (2) edge[line width=.12 mm,bend left=50,above,color=RawSienna,"$\frac{12}{27}$"] (0);
 
 \foreach \i/\txt in {1/$\frac{12}{27}$,2/$\frac{12}{27}$/,3/$\frac{12}{27}$,4/$\frac{12}{27}$,5/$\frac{12}{27}$}
 \draw (\i) edge[line width=.44 mm,loop above,color=NavyBlue, "\txt"] (\i);
 \draw (0) edge[line width=.59 mm,loop left, color=NavyBlue,"$\frac{16}{27}$"] (0);
 
 \draw[decorate sep={1mm}{4mm},fill] (10,0) -- (12,0);
 \end{tikzpicture}
 \begin{tikzpicture}
 \draw (4,-1.8) node
 {\begin{minipage}{0.8\textwidth}
 \begin{center}\small
 \textbf{Uniform type II Markov chain with one sink and one source}
 \end{center}
 \end{minipage}};
 \end{tikzpicture}
\end{center}

\paragraph{\textbf{The stochastic uniform tuples }
 $\Big(\frac 4 3, \frac 5 3, \frac 3 2, 2 \Big)$ and $\Big( \frac 5 3, \frac 4 3, \frac 3 2, 2 \Big)$}
In this case the Jacobi matrix and the type I stochastic matrix are
\begin{align*}
 J & =
 \left(\begin{NiceMatrix}[]
 5\kappa& 1 & 0 & 0 & 0 & \Cdots\\ 
 3\kappa^2&3\kappa& 1 & 0 & 0 & \Ddots \\ 
 \kappa^3& 3\kappa^2& 3\kappa& 1 &0&\Ddots\\ 
 0& \kappa^3&3\kappa^2& 3\kappa& 1 & \Ddots\\
 \Vdots&\Ddots&\Ddots&\Ddots&\Ddots&\Ddots
 \end{NiceMatrix}\right),&
 P_I&= \left(\begin{NiceMatrix}[columns-width =auto]
 \frac{20}{27} &\frac{6}{27} &\frac{1}{27}& 0 & \Cdots& & \\
 \frac{8}{27} & \frac{12}{27} &\frac{6}{27} &\frac{1}{27} & \Ddots & & \\
 0& \frac{8}{27} & \frac{12}{27} &\frac{6}{27} &\frac{1}{27} & \Ddots & \\
 \Vdots& \Ddots& \frac{8}{27} & \frac{12}{27} &\frac{6}{27} &\frac{1}{27} & \\
 & & & \Ddots& \Ddots & \Ddots&\Ddots
 \end{NiceMatrix}
 \right),
 \end{align*}
with corresponding diagram 
\begin{center}
 \tikzset{decorate sep/.style 2 args={decorate,decoration={shape backgrounds,shape=circle,shape size=#1,shape sep=#2}}}
 \begin{tikzpicture}[start chain = going right,
 -latex, every loop/.append style = {-latex}]\small
 \foreach \i in {0,...,5}
 \node[state, on chain,fill=gray!50!white] (\i) {\i};
 \foreach
 \i/\txt in {0/$\frac{6}{27}$,1/$\frac{6}{27}$,2/$\frac{6}{27}$,3/$\frac{6}{27}$,4/$\frac{6}{27}$}
 \draw let \n1 = { int(\i+1) } in
 (\i) edge[line width=.22 mm,bend left,"\txt",below,color=Periwinkle] (\n1);

 \foreach
 \i/\txt in {0/$\frac{8}{27}$,1/$\frac{8}{27}$/,2/$\frac{8}{27}$,3/$\frac{8}{27}$,4/$\frac{8}{27}$}
 \draw let \n1 = { int(\i+1) } in
 (\n1) edge[line width=.3 mm,bend left,below, "\txt",color=Mahogany,auto=right] (\i);
 
 \foreach
 \i/\txt in {0/$\frac{1}{27}$,1/$\frac{1}{27}$,1/$\frac{1}{27}$/,2/$\frac{1}{27}$,3/$\frac{1}{27}$}
 \draw let \n1 = { int(\i+2) } in
 (\i) edge[line width=.04 mm,bend left=65,color=MidnightBlue,"\txt"](\n1); 
 
 \foreach \i/\txt in {1/$\frac{12}{27}$,2/$\frac{12}{27}$/,3/$\frac{12}{27}$,4/$\frac{12}{27}$,5/$\frac{12}{27}$}
 \draw (\i) edge[line width=.44 mm,loop below, color=NavyBlue,"\txt"] (\i);
 \draw (0) edge[line width=.74mm,loop left, color=NavyBlue,"$\frac{20}{27}$"] (0);
 
 \draw[decorate sep={1mm}{4mm},fill] (10,0) -- (12,0);
 \end{tikzpicture}
 \begin{tikzpicture}
 \draw (4,-1.8) node
 {\begin{minipage}{0.8\textwidth}
 \begin{center}\small
 \textbf{Uniform type I Markov chain}
 \end{center}
 \end{minipage}};
 \end{tikzpicture}
\end{center}
The system of weights $(W_1,W_2,\d x)$ and $(W_1,\hat W_2,\d x)$ corresponding to the stochastic uniform tuples $\big(\frac 4 3, \frac 5 3, \frac 3 2, 2 \big)$ and $\big( \frac 5 3, \frac 4 3, \frac 3 2, 2 \big)$ are, respectively, 
\begin{align*}
 W_1 (x) & = \frac{9\sqrt{3} \sqrt[\leftroot{2}\uproot{2} 3]{x} (\vartheta_+(x)+\vartheta_-(x))}{16\pi \sqrt{1-x} },& 
 W_2 (x) & = \frac{81\sqrt{3}\sqrt[\leftroot{2}\uproot{2} 3]{x} (\vartheta_+^4(x)+\vartheta_-^4(x))}{160 \pi \sqrt{1-x} } , &
 \hat W_2 (x) & = \frac{81 \sqrt{3} \sqrt[\leftroot{2}\uproot{2} 3]{x^2} (\vartheta_+^2(x)+\vartheta_-^2(x))}{128\pi \sqrt{1-x} } .
\end{align*}
Solving the system of equations \eqref{eq:rho}
we find that $ \alpha = \frac{9}{5}$ and $\beta = -\frac{4}{5}$.
After some simplifications we get 
\begin{align*}
 \frac{9}{5} W_1 -\frac{4}{5} \hat W_2 & =-81\sqrt{3}\sqrt[\leftroot{2}\uproot{2} 3]{x} \,
 \frac{\vartheta_+^2(x)+\vartheta_-^2(x)-2(\vartheta_+(x)+\vartheta_-(x)) }{32 \pi \sqrt[\leftroot{2}\uproot{2} 3]{x} \sqrt{1-x} } ,
\end{align*}
but the numerator can be written as
$ -(-(\vartheta_+ + \vartheta_-) (\vartheta_+^2 - \vartheta_+ \vartheta_- + 
\vartheta_-^2) (\vartheta_- + \vartheta_+) + (\vartheta_-^2 + \vartheta_+^2) \
(\vartheta_+ \vartheta_-)) (\vartheta_+ \vartheta_-)$,
and, consequently, we obtain 
\begin{align*}
 \frac{9}{5} W_1 -\frac{4}{5} \hat W_2 & = \frac{81 \sqrt{3}}{160 \pi \sqrt{1-x}} \vartheta_+\vartheta_-\big( \vartheta_+^4+ \vartheta_-^4\big) = W_2,
\end{align*}
i.e., $\hat W_2=\frac{9}{4}W_1-\frac{5}{4}W_2$, and both sets of weights are in the same \emph{gauge} class.

The type II multiple orthogonal polynomials are
\begin{align}\label{eq:B_3F2_3}
 B^{(n)}(x)= 
 \frac{(n+1) (3 n+2)}{2} ( -\kappa )^n
\, \tensor[_3]{F}{_2}\hspace*{-3pt}\left[{\begin{NiceArray}{c}[small]-n,\; \frac{n+3}{2} ,\; \frac{n+2}{2} \\[3pt]
\frac{4}{3},\;\frac{5}{3}\end{NiceArray}};x\right] , && n \in \N . 
\end{align}
Following similar arguments as for the deduction of \eqref{eq:Bn(1)}, we get the following values at unity of the type II multiple orthogonal polynomials and generalized hypergeometric functions 
\begin{align*}
	B^{(n)}(1)&= \frac{(-1)^n+8^{n+1} }{9\times 27^n }, & 
	\tensor[_3]{F}{_2}\hspace*{-3pt}\left[{\begin{NiceArray}{c}[small]-n,\; \frac{n+3}{2} ,\; \frac{n+2}{2} \\[3pt]
			\frac{4}{3},\;\frac{5}{3}\end{NiceArray}};1\right]&
		=\frac{2(1-(-8)^{n+1} )}{9(n+1)(3n+2)4^n}, &
	n&\in\N.
\end{align*}
The stochastic factorization of the type I Markov matrix $P_I=P_{I}^L P_{I,2}^U P_{I,1}^U$
\begin{align*}P_{I}^L&=\left(
 \begin{NiceMatrix}[columns-width = 0.1cm]
 1 & 0 &\Cdots & & \\
 \frac{2}{5} & \frac{3}{5} & \Ddots & & \\
 0 & \frac{1}{2} & \frac{1}{2} & & \\
 \Vdots & \Ddots & \frac{6}{11} & \frac{5}{11} & \\
 &&\Ddots&\Ddots&\Ddots
 \end{NiceMatrix}
 \right), & P_{I,2}^U&= \left(\begin{NiceMatrix}[columns-width =auto]
 \frac{5}{6} &\frac{1}{6} & 0 & \Cdots& & \\
 0 & \frac{16}{21} & \frac{5}{21} & \Ddots & & \\
 \Vdots& \Ddots& \frac{11}{15} & \frac{4}{15} & & \\[2pt]
 & & &\frac{28}{39} & \frac{11}{39}& \\[2pt]
 & & & & \Ddots & \Ddots
 \end{NiceMatrix}\right),& P_{I,I}^U&= \left(\begin{NiceMatrix}[columns-width =auto]
 \frac{8}{9} &\frac{1}{9} & 0 & \Cdots& & \\
 0 & \frac{7}{9}&\frac{2}{9}& \Ddots & & \\
 \Vdots& \Ddots& \frac{20}{27} & \frac{7}{27} & & \\[2pt]
 & & &\frac{13}{18} & \frac{5}{18}& \\[2pt]
 & & & & \Ddots & \Ddots
 \end{NiceMatrix}\right).
\end{align*}
There is no uniform factor. The corresponding type II transition matrix 
$ P_{\text{II}} 
 = \left(\begin{NiceMatrix}[small]
 \frac{20}{27}& \frac{8}{27} & 0 & \Cdots & & \\ 
 \frac {6} {27}& \frac{12}{27}& \frac{8}{27} & \Ddots & & \\ 
 \frac{1}{27} & \frac {6} {27}& \frac{12}{27}& \frac{8}{27} &&\\ 
 0& \frac{1}{27} &\frac {6} {27}& \frac{12}{27}& \frac{8}{27} & \\ 
 \Vdots&\Ddots&\Ddots&\Ddots&\Ddots&\Ddots
 \end{NiceMatrix}\right)$
 is stochastic but for the two first rows, this models a recurrent random walk with the first state being a sink and the second a source. The diagram of this random walk is
\begin{center}
 \tikzset{decorate sep/.style 2 args={decorate,decoration={shape backgrounds,shape=circle,shape size=#1,shape sep=#2}}}
 \begin{tikzpicture}[start chain = going right,
 -latex, every loop/.append style = {-latex}]\small
 \foreach \i in {0,...,5}
 \node[state, on chain,fill=gray!50!white] (\i) {\i};
 \foreach
 \i/\txt in {0/$\frac{8}{27}$,1/$\frac{8}{27}$/,2/$\frac{8}{27}$,3/$\frac{8}{27}$,4/$\frac{8}{27}$}
 \draw let \n1 = { int(\i+1) } in
 (\i) edge[line width=.3 mm,bend left,"\txt",color=Periwinkle] (\n1);
 \foreach
 \i/\txt in {0/$\frac{6}{27}$,1/$\frac{6}{27}$,2/$\frac{6}{27}$,3/$\frac{6}{27}$,4/$\frac{6}{27}$}
 \draw let \n1 = { int(\i+1) } in
 (\n1) edge[line width=.22 mm,bend left=50,above, "\txt",color=Mahogany,auto=right] (\i);
 
 \foreach
 \i/\txt in {1/$\frac{1}{27}$/,2/$\frac{1}{27}$,3/$\frac{1}{27}$}
 \draw let \n1 = { int(\i+2) } in
 (\n1) edge[line width=.04 mm,bend left=60,color=RawSienna,"\txt"] (\i);
 \draw (2) edge[line width=.12 mm,bend left=50,above,color=RawSienna,"$\frac{12}{27}$"] (0);
 
 \foreach \i/\txt in {1/$\frac{12}{27}$,2/$\frac{12}{27}$/,3/$\frac{12}{27}$,4/$\frac{12}{27}$,5/$\frac{12}{27}$}
 \draw (\i) edge[line width=.44 mm,loop above,color=NavyBlue, "\txt"] (\i);
 \draw (0) edge[line width=.74 mm,loop left, color=NavyBlue,"$\frac{20}{27}$"] (0);
 
 \draw[decorate sep={1mm}{4mm},fill] (10,0) -- (12,0);
 \end{tikzpicture}
 \begin{tikzpicture}
 \draw (4,-1.8) node
 {\begin{minipage}{0.8\textwidth}
 \begin{center}\small
 \textbf{Uniform type II Markov chain with one sink and one source}
 \end{center}
 \end{minipage}};
 \end{tikzpicture}
\end{center}

\subsection{Semi-stochastic uniform tuples. Uniform transient random walks with sinks}\label{S:semi-stochastic uniform tuples}

We now proceed as for the type II stochastic uniform tuples, that is we normalize the almost asymptotic uniform Jacobi matrices using the matrices
$\sigma_{II}$ and $\sigma_{I}=\sigma_{II}^{-1}$, see \eqref{eq:norma}. For these semi-stochastic tuples $\delta = \frac 32$, so that they are transient random walks with sinks. All of them satisfy the min-max property \eqref{eq:min-max}. Moreover, the corresponding uniform Jacobi matrices differ from \eqref{eq:Jacobi_Jacobi_Piñeiro_uniform} by at most one column.

\paragraph{\textbf{The semi-stochastic uniform tuples} $\Big(\frac 1 3, \frac 2 3, 1, \frac 3 2\Big)$ and $\Big( \frac 2 3, \frac 1 3, 1, \frac 3 2 \Big)$}

In this case the Jacobi matrix and semi-stochastic matrices are
\begin{align*}
\hspace*{-1cm}\begin{aligned}
 J & =
 \left(\begin{NiceMatrix}
 \kappa& 1 & 0 & \Cdots & & \\[-5pt]
 2\kappa^2&3\kappa& 1 & \Ddots& & \\ 
 \kappa^3& 3\kappa^2& 3\kappa& 1 &&\\ 
 0& \kappa^3&3\kappa^2& 3\kappa& 1 & \\
 \Vdots&\Ddots&\Ddots&\Ddots&\Ddots&\Ddots
 \end{NiceMatrix}\right),&
 P_I&= \left(\begin{NiceMatrix}[columns-width =auto]
 \frac{4}{27} &\frac{4}{27} &\frac{1}{27}& 0 & \Cdots& & \\
 \frac{8}{27} & \frac{12}{27} &\frac{6}{27} &\frac{1}{27} & \Ddots & & \\
 0& \frac{8}{27} & \frac{12}{27} &\frac{6}{27} &\frac{1}{27} & \Ddots & \\
 \Vdots& \Ddots& \frac{8}{27} & \frac{12}{27} &\frac{6}{27} &\frac{1}{27} & \\
 & & & \Ddots& \Ddots & \Ddots&\Ddots
 \end{NiceMatrix}
 \right), & P_{\text{II}} 
& = \left(\begin{NiceMatrix}
 \frac{4}{27}& \frac{8}{27} & 0 & \Cdots& & \\ 
 \frac {4} {27}& \frac{12}{27}& \frac{8}{27} & \Ddots& & \\[5pt]
 \frac{1}{27} & \frac {6} {27}& \frac{12}{27}& \frac{8}{27} &&\\ 
 0& \frac{1}{27} &\frac {6} {27}& \frac{12}{27}& \frac{8}{27} & \Ddots\\ 
 \Vdots&\Ddots&\Ddots&\Ddots&\Ddots&\Ddots
 \end{NiceMatrix}\right),
\end{aligned}
\end{align*}
with corresponding diagrams
\begin{center}
 \tikzset{decorate sep/.style 2 args={decorate,decoration={shape backgrounds,shape=circle,shape size=#1,shape sep=#2}}}
 \begin{tikzpicture}[start chain = going right,
 -latex, every loop/.append style = {-latex}]\small
 \foreach \i in {0,...,5}
 \node[state, on chain,fill=gray!50!white] (\i) {\i};
 \foreach
 \i/\txt in {1/$\frac{6}{27}$,2/$\frac{6}{27}$,3/$\frac{6}{27}$,4/$\frac{6}{27}$}
 \draw let \n1 = { int(\i+1) } in
 (\i) edge[line width=.22 mm,bend left,"\txt",below,color=Periwinkle] (\n1);
 \draw (0) edge[line width=.15 mm,bend left,color=Periwinkle,"$\frac{4}{27}$"](1);

 \foreach
 \i/\txt in {0/$\frac{8}{27}$,1/$\frac{8}{27}$/,2/$\frac{8}{27}$,3/$\frac{8}{27}$,4/$\frac{8}{27}$}
 \draw let \n1 = { int(\i+1) } in
 (\n1) edge[line width=.3 mm,bend left,below, "\txt",color=Mahogany,auto=right] (\i);
 
 \foreach
 \i/\txt in {0/$\frac{1}{27}$,1/$\frac{1}{27}$,1/$\frac{1}{27}$/,2/$\frac{1}{27}$,3/$\frac{1}{27}$}
 \draw let \n1 = { int(\i+2) } in
 (\i) edge[line width=.04 mm,bend left=65,color=MidnightBlue,"\txt"](\n1); 
 
 \foreach \i/\txt in {1/$\frac{12}{27}$,2/$\frac{12}{27}$/,3/$\frac{12}{27}$,4/$\frac{12}{27}$,5/$\frac{12}{27}$}
 \draw (\i) edge[line width=.44 mm,loop below, color=NavyBlue,"\txt"] (\i);
 \draw (0) edge[line width=.15mm,loop left, color=NavyBlue,"$\frac{4}{27}$"] (0);
 
 \draw[decorate sep={1mm}{4mm},fill] (10,0) -- (12,0);
 \end{tikzpicture}
\end{center}

\begin{center}
 \tikzset{decorate sep/.style 2 args={decorate,decoration={shape backgrounds,shape=circle,shape size=#1,shape sep=#2}}}
 \begin{tikzpicture}
 \draw (4,-1.8) node
 {\begin{minipage}{0.8\textwidth}
 \begin{center}\small
 \textbf{Uniform type I Markov chain with one sink}
 \end{center}
 \end{minipage}};
 \end{tikzpicture}
 \begin{tikzpicture}[start chain = going right,
 -latex, every loop/.append style = {-latex}]\small
 \foreach \i in {0,...,5}
 \node[state, on chain,fill=gray!50!white] (\i) {\i};
 \foreach
 \i/\txt in {0/$\frac{8}{27}$,1/$\frac{8}{27}$/,2/$\frac{8}{27}$,3/$\frac{8}{27}$,4/$\frac{8}{27}$}
 \draw let \n1 = { int(\i+1) } in
 (\i) edge[line width=.3 mm,bend left,"\txt",color=Periwinkle] (\n1);
 \foreach
 \i/\txt in {1/$\frac{6}{27}$,2/$\frac{6}{27}$,3/$\frac{6}{27}$,4/$\frac{6}{27}$}
 \draw let \n1 = { int(\i+1) } in
 (\n1) edge[line width=.22 mm,bend left=50,above, "\txt",color=Mahogany,auto=right] (\i);
 \draw (1) edge[line width=.15 mm,bend left=50,above, auto=right,color=Mahogany,"$\frac{4}{27}$"] (0);
 
 \foreach
 \i/\txt in {1/$\frac{1}{27}$/,2/$\frac{1}{27}$,3/$\frac{1}{27}$}
 \draw let \n1 = { int(\i+2) } in
 (\n1) edge[line width=.04 mm,bend left=60,color=RawSienna,"\txt"] (\i);
 \draw (2) edge[line width=.12 mm,bend left=50,above,color=RawSienna,"$\frac{12}{27}$"] (0);
 
 \foreach \i/\txt in {1/$\frac{12}{27}$,2/$\frac{12}{27}$/,3/$\frac{12}{27}$,4/$\frac{12}{27}$,5/$\frac{12}{27}$}
 \draw (\i) edge[line width=.44 mm,loop above,color=NavyBlue, "\txt"] (\i);
 \draw (0) edge[line width=.15 mm,loop left, color=NavyBlue,"$\frac{4}{27}$"] (0);
 
 \draw[decorate sep={1mm}{4mm},fill] (10,0) -- (12,0);
\end{tikzpicture}
\begin{tikzpicture}
 \draw (4,-1.8) node
 {\begin{minipage}{0.8\textwidth}
 \begin{center}\small
 \textbf{Uniform type II Markov chain with two sinks}
 \end{center}
 \end{minipage}};
\end{tikzpicture}
\end{center}
The systems of weights $(W_1,W_2,\d x)$ and $(W_1,\hat W_2,\d x)$ corresponding to the semi-stochastic uniform tuples $\big(\frac 1 3, \frac 2 3, 1, \frac 3 2\big)$ and $\big( \frac 2 3, \frac 1 3, 1, \frac 3 2 \big)$ are, respectively, 
\begin{align*}
 W_1 (x) & = \frac{3\sqrt{3} (\vartheta_+(x)-\vartheta_-(x))}{4\pi \sqrt[\leftroot{2}\uproot{2} 3]{x^2} },& 
 W_2 (x) & = \frac{9\sqrt{3} (\vartheta_+^4(x)-\vartheta_-^4(x))}{32 \pi \sqrt[\leftroot{2}\uproot{2} 3]{x^2} } , &
 \hat W_2 (x) & = \frac{9 \sqrt{3} ( (\vartheta_+^2(x)-\vartheta_-^2(x)))}{8\pi \sqrt[\leftroot{2}\uproot{2} 3]{x} } .
\end{align*}
From \eqref{eq:rho}
we find that $ \alpha = \frac{3}{4}$ and $\beta = \frac{1}{4}$.
We get
\begin{align*}
 \frac{3}{4} W_1 +\frac{1}{4} \hat W_2 & =9\sqrt{3} \,
 \frac{(\vartheta_+^2(x)-\vartheta_-^2(x))\big(2+ \sqrt[\leftroot{2}\uproot{2} 3]{x} (\vartheta_+^3(x)+\vartheta_-^3(x)))\big) }{32 \pi \sqrt[\leftroot{2}\uproot{2} 3]{x^2} } ,
\end{align*}
but the numerator can be written as
$ (\vartheta_- - \vartheta_+) ((\vartheta_+ + \vartheta_-) (\vartheta_+^2 - \vartheta_+ \vartheta_- + 
\vartheta_-^2) + (\vartheta_+ + \vartheta_-) \vartheta_+ \vartheta_-) $,
and, consequently, we obtain 
\begin{align*}
 \frac{3}{4} W_1 +\frac{1}{4} \hat W_2 & = \frac{9 \sqrt{3}}{32 \pi \sqrt[\leftroot{2}\uproot{2} 3]{x^2}} \big( \vartheta_+^4- \vartheta_-^4\big) = W_2,
\end{align*}
i.e., $\hat W_2=-3W_1+4W_2$, and both sets of weights are in the same \emph{gauge} class.

The type II multiple orthogonal polynomials are
\begin{align}\label{eq:B_3F2_4}
 B^{(n)}(x)= (-\kappa)^n
 \tensor[_3]{F}{_2}\hspace*{-3pt}\left[{\begin{NiceArray}{c}[small]-n,\; \frac{n+1}{2} , \; \frac{n+2}{2} \\[3pt]
 \frac{1}{3},\;\frac{2}{3}\end{NiceArray}};x\right] , && n \in \N . 
\end{align}
Following similar arguments as for the deduction of \eqref{eq:Bn(1)}, we get the following values at unity of the type II multiple orthogonal polynomials and generalized hypergeometric functions 
\begin{align*}
	B^{(n)}(1)&= \frac{2(9 n+4)8^n+(-1)^n }{9\times 27^n}, & 
	\tensor[_3]{F}{_2}\hspace*{-3pt}\left[{\begin{NiceArray}{c}[small]-n,\; \frac{n+1}{2} , \; \frac{n+2}{2} \\[3pt]
			\frac{1}{3},\;\frac{2}{3}\end{NiceArray}};1\right]&
	=\frac{1+2(-1)^n (9 n+4)8^n}{9 \times 4^{n}}, &
	n&\in\N.
\end{align*}

The stochastic factorization of the type I Markov matrix $P_I=P_{I}^L P_{I,2}^U P_{I,1}^U$ is not uniform
\begin{align*}P_{I}^L&=\left(
 \begin{NiceMatrix}[columns-width = 0.1cm]
 1 & 0 &\Cdots & & \\
 \frac{1}{2} & \frac{1}{2} & \Ddots & & \\[5pt]
 0 & \frac{4}{7} & \frac{3}{7} & & \\
 \Vdots & \Ddots & \frac{3}{5} & \frac{2}{5} & \\
 &&\Ddots&\Ddots&\Ddots
 \end{NiceMatrix}
 \right), & P_{I,2}^U&= \left(\begin{NiceMatrix}[columns-width =auto]
 \frac{4}{9} &\frac{5}{9} & 0 & \Cdots& & \\
 0 & \frac{5}{9} & \frac{4}{9} & \Ddots & & \\
 \Vdots& \Ddots& \frac{16}{27} & \frac{11}{27} & & \\[5pt]
 & & &\frac{11}{18} & \frac{7}{18}& \\
 & & & & \Ddots & \Ddots
 \end{NiceMatrix}\right),& P_{I,1}^U&= \left(\begin{NiceMatrix}[columns-width =auto]
 \frac{1}{3} &\frac{2}{3} & 0 & \Cdots& & \\
 0 & \frac{8}{15}&\frac{7}{15}& \Ddots & & \\
 \Vdots& \Ddots& \frac{7}{12} & \frac{5}{12} & & \\[5pt]
 & & &\frac{20}{33} & \frac{13}{33}& \\
 & & & & \Ddots & \Ddots
 \end{NiceMatrix}\right).
\end{align*}

\paragraph{\textbf{The semi-stochastic uniform tuples} $\Big(\frac 2 3, \frac 4 3, \frac 3 2 ,2\Big)$ and $\Big( \frac 4 3, \frac 2 3, \frac 3 2 ,2 \Big)$}

In this case the Jacobi matrix and semi-stochastic matrices are
\begin{align*}
 \hspace*{-1cm}\begin{aligned}
 J & =
 \left(\begin{NiceMatrix}
 2\kappa& 1 & 0 & \Cdots & & \\[-5pt]
 3\kappa^2&3\kappa& 1 & \Ddots& & \\ 
 \kappa^3& 3\kappa^2& 3\kappa& 1 &&\\ 
 0& \kappa^3&3\kappa^2& 3\kappa& 1 & \\
 \Vdots&\Ddots&\Ddots&\Ddots&\Ddots&\Ddots
 \end{NiceMatrix}\right),&
 P_I&= \left(\begin{NiceMatrix}[columns-width =auto]
 \frac{8}{27} &\frac{6}{27} &\frac{1}{27}& 0 & \Cdots& & \\
 \frac{8}{27} & \frac{12}{27} &\frac{6}{27} &\frac{1}{27} & \Ddots & & \\
 0& \frac{8}{27} & \frac{12}{27} &\frac{6}{27} &\frac{1}{27} & \Ddots & \\
 \Vdots& \Ddots& \frac{8}{27} & \frac{12}{27} &\frac{6}{27} &\frac{1}{27} & \\
 & & & \Ddots& \Ddots & \Ddots&\Ddots
 \end{NiceMatrix}
 \right), & P_{\text{II}} 
 & = \left(\begin{NiceMatrix}
 \frac{8}{27}& \frac{8}{27} & 0 & \Cdots& & \\ 
 \frac {6} {27}& \frac{12}{27}& \frac{8}{27} & \Ddots& & \\[5pt]
 \frac{1}{27} & \frac {6} {27}& \frac{12}{27}& \frac{8}{27} &&\\ 
 0& \frac{1}{27} &\frac {6} {27}& \frac{12}{27}& \frac{8}{27} & \Ddots\\ 
 \Vdots&\Ddots&\Ddots&\Ddots&\Ddots&\Ddots
 \end{NiceMatrix}\right),
 \end{aligned}
\end{align*}
with corresponding diagrams 
\begin{center}
 \tikzset{decorate sep/.style 2 args={decorate,decoration={shape backgrounds,shape=circle,shape size=#1,shape sep=#2}}}
 \begin{tikzpicture}[start chain = going right,
 -latex, every loop/.append style = {-latex}]\small
 \foreach \i in {0,...,5}
 \node[state, on chain,fill=gray!50!white] (\i) {\i};
 \foreach
 \i/\txt in {0/$\frac{6}{27}$,1/$\frac{6}{27}$,2/$\frac{6}{27}$,3/$\frac{6}{27}$,4/$\frac{6}{27}$}
 \draw let \n1 = { int(\i+1) } in
 (\i) edge[line width=.22 mm,bend left,"\txt",below,color=Periwinkle] (\n1);

 \foreach
 \i/\txt in {0/$\frac{8}{27}$,1/$\frac{8}{27}$/,2/$\frac{8}{27}$,3/$\frac{8}{27}$,4/$\frac{8}{27}$}
 \draw let \n1 = { int(\i+1) } in
 (\n1) edge[line width=.3 mm,bend left,below, "\txt",color=Mahogany,auto=right] (\i);
 
 \foreach
 \i/\txt in {0/$\frac{1}{27}$,1/$\frac{1}{27}$,1/$\frac{1}{27}$/,2/$\frac{1}{27}$,3/$\frac{1}{27}$}
 \draw let \n1 = { int(\i+2) } in
 (\i) edge[line width=.04 mm,bend left=65,color=MidnightBlue,"\txt"](\n1); 
 
 \foreach \i/\txt in {1/$\frac{12}{27}$,2/$\frac{12}{27}$/,3/$\frac{12}{27}$,4/$\frac{12}{27}$,5/$\frac{12}{27}$}
 \draw (\i) edge[line width=.44 mm,loop below, color=NavyBlue,"\txt"] (\i);
 \draw (0) edge[line width=.3mm,loop left, color=NavyBlue,"$\frac{8}{27}$"] (0);
 
 \draw[decorate sep={1mm}{4mm},fill] (10,0) -- (12,0);
 \end{tikzpicture}
 \begin{tikzpicture}
 \draw (4,-1.8) node
 {\begin{minipage}{0.8\textwidth}
 \begin{center}\small
 \textbf{Uniform type I Markov chain with one sink}
 \end{center}
 \end{minipage}};
 \end{tikzpicture}
 \begin{tikzpicture}[start chain = going right,
 -latex, every loop/.append style = {-latex}]\small
 \foreach \i in {0,...,5}
 \node[state, on chain,fill=gray!50!white] (\i) {\i};
 \foreach
 \i/\txt in {0/$\frac{8}{27}$,1/$\frac{8}{27}$/,2/$\frac{8}{27}$,3/$\frac{8}{27}$,4/$\frac{8}{27}$}
 \draw let \n1 = { int(\i+1) } in
 (\i) edge[line width=.3 mm,bend left,"\txt",color=Periwinkle] (\n1);
 \foreach
 \i/\txt in {0/$\frac{6}{27}$,1/$\frac{6}{27}$,2/$\frac{6}{27}$,3/$\frac{6}{27}$,4/$\frac{6}{27}$}
 \draw let \n1 = { int(\i+1) } in
 (\n1) edge[line width=.22 mm,bend left=50,above, "\txt",color=Mahogany,auto=right] (\i);
 
 \foreach
 \i/\txt in {0/$\frac{1}{27}$,1/$\frac{1}{27}$/,2/$\frac{1}{27}$,3/$\frac{1}{27}$}
 \draw let \n1 = { int(\i+2) } in
 (\n1) edge[line width=.04 mm,bend left=60,color=RawSienna,"\txt"] (\i);
 
 \foreach \i/\txt in {1/$\frac{12}{27}$,2/$\frac{12}{27}$/,3/$\frac{12}{27}$,4/$\frac{12}{27}$,5/$\frac{12}{27}$}
 \draw (\i) edge[line width=.44 mm,loop above,color=NavyBlue, "\txt"] (\i);
 \draw (0) edge[line width=.3 mm,loop left, color=NavyBlue,"$\frac{8}{27}$"] (0);
 
 \draw[decorate sep={1mm}{4mm},fill] (10,0) -- (12,0);
 \end{tikzpicture}
 \begin{tikzpicture}
 \draw (4,-1.8) node
 {\begin{minipage}{0.8\textwidth}
 \begin{center}\small
 \textbf{Uniform type II Markov chain with two sinks}
 \end{center}
 \end{minipage}};
 \end{tikzpicture}
\end{center}
The systems of weights $(W_1,W_2,\d x)$ and $(W_1,\hat W_2,\d x)$ corresponding to the semi-stochastic uniform tuples $\big(\frac 2 3, \frac 4 3, \frac 3 2 ,2\big)$ and $\big( \frac 4 3, \frac 2 3, \frac 3 2 ,2 \big)$ are, respectively, 
\begin{align*}
 W_1 (x) & = \frac{9\sqrt{3} (\vartheta_+^2(x)-\vartheta_-^2(x))}{8\pi \sqrt[\leftroot{2}\uproot{2} 3]{x} },& 
 W_2 (x) & = \frac{81\sqrt{3} (\vartheta_+^5(x)-\vartheta_-^5(x))}{160 \pi \sqrt[\leftroot{2}\uproot{2} 3]{x} } , &
 \hat W_2 (x) & = \frac{81 \sqrt{3} \sqrt[\leftroot{2}\uproot{2} 3]{x}(\vartheta_+(x)-\vartheta_-(x))}{16\pi } .
\end{align*}
From \eqref{eq:rho}
we find that $ \alpha = \frac{9}{10}$ and $\beta = \frac{1}{10}$.
We get
\begin{align*}
 \frac{9}{10} W_1 +\frac{1}{10} \hat W_2 & =81\sqrt{3} \,
 \frac{(\vartheta_+(x)-\vartheta_-(x))\big( \sqrt[\leftroot{2}\uproot{2} 3]{x}+2 (\vartheta_+(x)+\vartheta_-(x))\big) }{160 \pi \sqrt[\leftroot{2}\uproot{2} 3]{x} } ,
\end{align*}
but the numerator can be written as
$ ( \vartheta_- - \vartheta_+) ((\vartheta_+ + \vartheta_-) (\vartheta_+^2 - \vartheta_+ \vartheta_- + 
\vartheta_-^2) (\vartheta_- + \vartheta_+) + (\vartheta_+ \vartheta_-)^2)$,
and, consequently, we obtain 
\begin{align*}
 \frac{9}{10} W_1 +\frac{1}{10} \hat W_2 & = \frac{81 \sqrt{3}}{160 \pi \sqrt[\leftroot{2}\uproot{2} 3]{x}} \big( \vartheta_+^5- \vartheta_-^5\big) = W_2,
\end{align*}
i.e., $\hat W_2=-9W_1+10W_2$, and both sets of weights are in the same \emph{gauge} class.

The type II multiple orthogonal polynomials are
\begin{align}\label{eq:B_3F2_5}
 B^{(n)}(x)= (n+1) \left( -\kappa \right)^n
 \tensor[_3]{F}{_2}\hspace*{-3pt}\left[{\begin{NiceArray}{c}[small]-n,\; \frac{n+3}{2} , \; \frac{n+2}{2} \\[3pt]
 \frac{2}{3},\;\frac{4}{3}\end{NiceArray}};x\right] , && n \in \N . 
\end{align}
Following similar arguments as for the deduction of \eqref{eq:Bn(1)}, we get the following values at unity of the type II multiple orthogonal polynomials and generalized hypergeometric functions 
\begin{align*}
	B^{(n)}(1)&= \frac{4 (9 n+7)8^n+(-1)^{n+1} }{27^{n+1}}, & 
	\tensor[_3]{F}{_2}\hspace*{-3pt}\left[{\begin{NiceArray}{c}[small]-n,\; \frac{n+3}{2} , \; \frac{n+2}{2} \\[3pt]
			\frac{2}{3},\;\frac{4}{3}\end{NiceArray}};1\right]&
	=\frac{ 4 (9 n+7) (-8)^n -1 }{27 (n+1) 4^{n}}, &
	n&\in\N.
\end{align*}
The stochastic factorization of the type I Markov matrix $P_I=P_{I}^L P_{I,2}^U P_{I,1}^U$ is not uniform
\begin{align*}
\hspace{-.5cm}
P_{I}^L&=\left(
 \begin{NiceMatrix}[columns-width = 0.1cm]
 1 & 0 &\Cdots & & \\
 \frac{2}{5} & \frac{3}{5} & \Ddots & & \\[5pt]
 0 & \frac{1}{2} & \frac{1}{2} & & \\
 \Vdots & \Ddots & \frac{6}{11} & \frac{5}{11} & \\
 &&\Ddots&\Ddots&\Ddots
 \end{NiceMatrix}
 \right), & P_{I,2}^U&= \left(\begin{NiceMatrix}[columns-width =auto]
 \frac{2}{3} & \frac{1}{3} & 0 & \Cdots& & \\
 0 & \frac{2}{3} & \frac{1}{3}& \Ddots & & \\
 \Vdots& \Ddots& \frac{2}{3} & \frac{1}{3}& & \\[5pt]
 & & &\frac{2}{3} & \frac{1}{3}& \\
 & & & & \Ddots & \Ddots
 \end{NiceMatrix}\right),& P_{I,1}^U&= \left(\begin{NiceMatrix}[columns-width =auto]
 \frac{4}{9} & \frac{5}{9} & 0 & \Cdots& & \\
 0 & \frac{5}{9} & \frac{4}{9}& \Ddots & & \\
 \Vdots& \Ddots& \frac{16}{27} & \frac{11}{27}& & \\[5pt]
 & & & \frac{11}{18} & \frac{7}{18} & \\
 & & & & \Ddots & \Ddots
 \end{NiceMatrix}\right).
\end{align*}
Notice that there is a uniform factor.

\paragraph{\textbf{The semi-stochastic uniform tuples} $\Big(\frac 4 3, \frac 5 3, 2, \frac 5 2 \Big)$ and 
 $\Big( \frac 5 3, \frac 4 3, 2, \frac5 2 \Big)$}

Notice that $\big(\frac 4 3, \frac 5 3, 2, \frac 5 2 \big)$ is the uniform case discussed in \cite{lima_loureiro}. The Jacobi matrix and semi-stochastic matrices are
\begin{align*}
 \hspace*{-.85cm}\begin{aligned}
 J & =
 \left(\begin{NiceMatrix}
 3\kappa& 1 & 0 & \Cdots & & \\[-5pt]
 3\kappa^2&3\kappa& 1 & \Ddots& & \\ 
 \kappa^3& 3\kappa^2& 3\kappa& 1 &&\\ 
 0& \kappa^3&3\kappa^2& 3\kappa& 1 & \\
 \Vdots&\Ddots&\Ddots&\Ddots&\Ddots&\Ddots
 \end{NiceMatrix}\right),&
 P_I&= \left(\begin{NiceMatrix}[columns-width =auto]
 \frac{12}{27} &\frac{6}{27} &\frac{1}{27}& 0 & \Cdots& & \\
 \frac{8}{27} & \frac{12}{27} &\frac{6}{27} &\frac{1}{27} & \Ddots & & \\
 0& \frac{8}{27} & \frac{12}{27} &\frac{6}{27} &\frac{1}{27} & \Ddots & \\
 \Vdots& \Ddots& \frac{8}{27} & \frac{12}{27} &\frac{6}{27} &\frac{1}{27} & \\
 & & & \Ddots& \Ddots & \Ddots&\Ddots
 \end{NiceMatrix}
 \right), & P_{\text{II}} 
 & = \left(\begin{NiceMatrix}
 \frac{12}{27}& \frac{8}{27} & 0 & \Cdots& & \\ 
 \frac {6} {27}& \frac{12}{27}& \frac{8}{27} & \Ddots& & \\[5pt]
 \frac{1}{27} & \frac {6} {27}& \frac{12}{27}& \frac{8}{27} &&\\ 
 0& \frac{1}{27} &\frac {6} {27}& \frac{12}{27}& \frac{8}{27} & \Ddots\\ 
 \Vdots&\Ddots&\Ddots&\Ddots&\Ddots&\Ddots
 \end{NiceMatrix}\right),
 \end{aligned}
\end{align*}
with corresponding diagrams 
\begin{center}
 \tikzset{decorate sep/.style 2 args={decorate,decoration={shape backgrounds,shape=circle,shape size=#1,shape sep=#2}}}
 \begin{tikzpicture}[start chain = going right,
 -latex, every loop/.append style = {-latex}]\small
 \foreach \i in {0,...,5}
 \node[state, on chain,fill=gray!50!white] (\i) {\i};
 \foreach
 \i/\txt in {1/$\frac{6}{27}$/,2/$\frac{6}{27}$,3/$\frac{6}{27}$,4/$\frac{6}{27}$}
 \draw let \n1 = { int(\i+1) } in
 (\i) edge[line width=.22 mm,bend left,"\txt",below,color=Periwinkle] (\n1);
 
 \draw (0) edge[line width=.15 mm,bend left,"$\frac{4}{27}$",below,color=Periwinkle] (1);
 
 \foreach
 \i/\txt in {0/$\frac{8}{27}$,1/$\frac{8}{27}$/,2/$\frac{8}{27}$,3/$\frac{8}{27}$,4/$\frac{8}{27}$}
 \draw let \n1 = { int(\i+1) } in
 (\n1) edge[line width=.3 mm,bend left,below, "\txt",color=Mahogany,auto=right] (\i);
 
 \foreach
 \i/\txt in {0/$\frac{1}{27}$,1/$\frac{1}{27}$/,2/$\frac{1}{27}$,3/$\frac{1}{27}$}
 \draw let \n1 = { int(\i+2) } in
 (\i) edge[line width=.04 mm,bend left=65,color=MidnightBlue,"\txt"](\n1);
 
 \foreach \i/\txt in {1/$\frac{12}{27}$,2/$\frac{12}{27}$/,3/$\frac{12}{27}$,4/$\frac{12}{27}$,5/$\frac{12}{27}$}
 \draw (\i) edge[line width=.44 mm,loop below, color=NavyBlue,"\txt"] (\i);
 \draw (0) edge[line width=.15 mm,loop left, color=NavyBlue,"$\frac{4}{27}$"] (0);
 
 \draw[decorate sep={1mm}{4mm},fill] (10,0) -- (12,0);
 \end{tikzpicture}
 \begin{tikzpicture}
 \draw (4,-1.8) node
 {\begin{minipage}{0.8\textwidth}
 \begin{center}\small
 \textbf{Uniform type I Markov chain with one sink}
 \end{center}
 \end{minipage}};
 \end{tikzpicture}
 \begin{tikzpicture}[start chain = going right,
 -latex, every loop/.append style = {-latex}]\small
 \foreach \i in {0,...,5}
 \node[state, on chain,fill=gray!50!white] (\i) {\i};
 \foreach
 \i/\txt in {0/$\frac{8}{27}$,1/$\frac{8}{27}$/,2/$\frac{8}{27}$,3/$\frac{8}{27}$,4/$\frac{8}{27}$}
 \draw let \n1 = { int(\i+1) } in
 (\i) edge[line width=.3 mm,bend left,"\txt",color=Periwinkle] (\n1);
 \foreach
 \i/\txt in {0/$\frac{6}{27}$,1/$\frac{6}{27}$/,2/$\frac{6}{27}$,3/$\frac{6}{27}$,4/$\frac{6}{27}$}
 \draw let \n1 = { int(\i+1) } in
 (\n1) edge[line width=.22 mm,bend left=50,above, "\txt",color=Mahogany,auto=right] (\i);
 
 \foreach
 \i/\txt in {0/$\frac{1}{27}$,1/$\frac{1}{27}$/,2/$\frac{1}{27}$,3/$\frac{1}{27}$}
 \draw let \n1 = { int(\i+2) } in
 (\n1) edge[line width=.04 mm,bend left=60,color=RawSienna,"\txt"] (\i);
 
 \foreach \i/\txt in {1/$\frac{12}{27}$,2/$\frac{12}{27}$/,3/$\frac{12}{27}$,4/$\frac{12}{27}$,5/$\frac{12}{27}$}
 \draw (\i) edge[line width=.44 mm,loop above,color=NavyBlue, "\txt"] (\i);
 \draw (0) edge[line width=.44 mm,loop left, color=NavyBlue,"$\frac{12}{27}$"] (0);
 
 \draw[decorate sep={1mm}{4mm},fill] (10,0) -- (12,0);
 \end{tikzpicture}
 \begin{tikzpicture}
 \draw (4,-1.8) node
 {\begin{minipage}{0.8\textwidth}
 \begin{center}\small
 \textbf{Uniform type II Markov chain with two sinks}
 \end{center}
 \end{minipage}};
 \end{tikzpicture}

\end{center}
The systems of weights $(W_1,W_2,\d x)$ and $(W_1,\hat W_2,\d x)$ corresponding to the semi-stochastic uniform tuples 
$\big(\frac 4 3, \frac 5 3, 2, \frac 5 2 \big)$ and $\big( \frac 5 3, \frac 4 3, 2, \frac5 2 \big)$ are, respectively, 
\begin{align*}
 W_1 (x) & = \frac{81\sqrt{3} \sqrt[\leftroot{2}\uproot{2} 3]{x}(\vartheta_+(x)-\vartheta_-(x))}{16\pi },& 
 W_2 (x) & = \frac{243\sqrt{3} \sqrt[\leftroot{2}\uproot{2} 3]{x}(\vartheta_+^4(x)-\vartheta_-^4(x))}{160 \pi } , &
 \hat W_2 (x) & = \frac{243 \sqrt{3} \sqrt[\leftroot{2}\uproot{2} 3]{x^2}( (\vartheta_+^2(x)-\vartheta_-^2(x)))}{64\pi } .
\end{align*}
Solving \eqref{eq:rho}
we find that $ \alpha = \frac{3}{5}$ and $\beta = \frac{2}{5}$.
After some clearing we obtain
\begin{align*}
 \frac{3}{5} W_1 +\frac{2}{5} \hat W_2 & =243\sqrt{3} \sqrt[\leftroot{2}\uproot{2} 3]{x}\,
 \frac{(\vartheta_+^3(x)-\vartheta_-^3(x))\big( 2 +\sqrt[\leftroot{2}\uproot{2} 3]{x}(\vartheta_+^3(x)+\vartheta_-^3(x)))\big) }{160 \pi \sqrt[\leftroot{2}\uproot{2} 3]{x} } ,
\end{align*}
but the numerator can be written as
$ ( \vartheta_- - \vartheta_+) ((\vartheta_+ + \vartheta_-) (\vartheta_+^2 - \vartheta_+ \vartheta_- + 
\vartheta_-^2) + (\vartheta_- + \vartheta_+) (\vartheta_+ \vartheta_-)) (\vartheta_+ \vartheta_-) 
$,
and, consequently, we obtain 
\begin{align*}
 \frac{3}{5} W_1 +\frac{2}{5} \hat W_2 & =\frac{243 \sqrt{3}}{160 \pi } \vartheta_+ \vartheta_- \big( \vartheta_+^4 -\vartheta_-^4 \big) = W_2 
\end{align*}
i.e., $\hat W_2=-\frac{3}{2}W_1+\frac{5}{2}W_2$, and both sets of weights are in the same \emph{gauge} class.

The type II multiple orthogonal polynomials are \cite{lima_loureiro}
\begin{align}\label{eq:B_3F2_6}
 B^{(n)}(x)= \frac{(n+2) (n+1) }{2} 
 \left( -\kappa \right)^n
 \tensor[_3]{F}{_2}\hspace*{-3pt}\left[{\begin{NiceArray}{c}[small]-n,\; \frac{n+4}{2} , \; \frac{n+3}{2} \\[3pt]\frac{4}{3},\;\frac{5}{3}\end{NiceArray}};x\right] , && n \in \N . 
\end{align}
Following similar arguments as for the deduction of \eqref{eq:Bn(1)}, we get the following values at unity of the type II multiple orthogonal polynomials and generalized hypergeometric functions 
\begin{align*}
	B^{(n)}(1)&= \frac{8^{n+1} (9 n+10)+(-1)^n}{3 \times 27^{ n+1} }, & 
	\tensor[_3]{F}{_2}\hspace*{-3pt}\left[{\begin{NiceArray}{c}[small]-n,\; \frac{n+4}{2} , \; \frac{n+3}{2} \\[3pt]\frac{4}{3},\;\frac{5}{3}\end{NiceArray}};1\right]&
	=\frac{2(1- (9 n+10)(-8)^{n+1})}{81 (n+1) (n+2)4^n} &
	n&\in\N.
\end{align*}
The stochastic factorization of the type I Markov matrix $P_I=P_{I}^L P_{I,2}^U P_{I,1}^U$ is
\begin{align*}P_{I}^L&=\left(
 \begin{NiceMatrix}[columns-width = 0.1cm]
 1 & 0 &\Cdots & & \\
 \frac{1}{3} & \frac{2}{3} & \Ddots & & \\[5pt]
 0 & \frac{4}{9} & \frac{5}{9} & & \\
 \Vdots & \Ddots & \frac{1}{2} & \frac{1}{2} & \\
 &&\Ddots&\Ddots&\Ddots
 \end{NiceMatrix}
 \right), & P_{I,2}^U&= P_{I,1}^U= \left(\begin{NiceMatrix}[columns-width =auto]
 \frac{2}{3} & \frac{1}{3} & 0 & \Cdots& & \\
 0 & \frac{2}{3} & \frac{1}{3}& \Ddots & & \\
 \Vdots& \Ddots& \frac{2}{3} & \frac{1}{3}& & \\[5pt]
 & & &\frac{2}{3} & \frac{1}{3}& \\
 & & & & \Ddots & \Ddots
 \end{NiceMatrix}\right),
\end{align*}
with two uniform pure birth factors.

\subsection{Chain of Christoffel transformations. Contiguous relations for $_3F_2$}
Using the scaling transformations in Remark \ref{rem:scaling} and the Christoffel transformations described in Theorem \ref{teo:Christoffel} we consider the permuting Christoffel transformations 
$ (W_1,W_2)\xrightarrow[]{ C_\alpha} (W_2, \alpha x W_1)$,
among the stochastic uniform weights and the semi-stochastic uniform weights.
We will also consider the square of these transformations, that we called basic Christoffel transformations, that is 
$ (W_1,W_2)\xrightarrow[]{ C_{\alpha_1,\alpha_2}} (\alpha_1 x W_1,\alpha_2 x W_2 )$.
By inspection of the previous results we~get:
\begin{teo}[Christoffel chains for uniform tuples]\label{teo:Christoffel chains for uniform tuples}
 \begin{enumerate}
 \item The stochastic uniform weights are related through the following chain of permuting Christoffel transformations
 \begin{align}\label{eq:Christoffel_chain_1}
 \Big(\frac{2}{3},\frac{1}{3}, \frac{1}{2},1\Big)\xrightarrow[]{ C_{\frac{9}{4}}} 
 \Big(\frac{4}{3},\frac{2}{3}, 1,\frac{3}{2}\Big)\xrightarrow[]{ C_{\frac{27}{16}}} 
 \Big(\frac{5}{3},\frac{4}{3}, \frac{3}{2},{2}\Big).
 \end{align}
 \item The semi-stochastic uniform weights are connected by the following chain of permuting Christoffel transformations
 \begin{align}\label{eq:Christoffel_chain_2}
 \Big(\frac{2}{3},\frac{1}{3}, 1,\frac{3}{2}\Big)\xrightarrow[]{ C_{\frac{27}{4}}} 
 \Big(\frac{4}{3},\frac{2}{3}, \frac{3}{2},2\Big)\xrightarrow[]{ C_{\frac{27}{8}}} 
 \Big(\frac{5}{3},\frac{4}{3}, 
 2,\frac{5}{2}\Big).
 \end{align}
\item 
For the stochastic uniform weights we have the basic Christoffel transformations
\begin{subequations}\label{eq:Christoffel_basic_stochastic}
	 \begin{gather}\label{eq:Christoffel_basic_1}
 \Big(\frac{2}{3},\frac{1}{3}, \frac{1}{2},1\Big)\xrightarrow[]{ C_{\frac{9}{4},\frac{27}{16}}} 
 \Big(\frac{5}{3},\frac{4}{3}, \frac{3}{2},{2}\Big), \\\label{eq:Christoffel_basic_2}
\Big(\frac{1}{3},\frac{2}{3}, \frac{1}{2},1\Big)\xrightarrow[]{ C_{\frac{9}{4},\frac{27}{10}}} 
\Big(\frac{4}{3},\frac{5}{3}, \frac{3}{2},{2}\Big).
\end{gather}
\end{subequations}
\item 
For the semi-stochastic uniform weights we have the basic Christoffel transformations
\begin{subequations}\label{eq:Christoffel_basic_semistochastic}
	\begin{gather}\label{eq:Christoffel_basic_3}
 \Big(\frac{2}{3},\frac{1}{3}, 1, \frac{3}{2}\Big)\xrightarrow[]{ C_{\frac{27}{4},\frac{27}{8}}} 
 \Big(\frac{5}{3},\frac{4}{3}, 2,\frac{5}{2}\Big),\\\label{eq:Christoffel_basic_4}
 \Big(\frac{2}{3},\frac{1}{3}, 1, \frac{3}{2}\Big)\xrightarrow[]{ C_{\frac{27}{4},\frac{27}{5}}} 
 \Big(\frac{4}{3},\frac{5}{3}, 2,\frac{5}{2}\Big). 
\end{gather}
\end{subequations}
 \end{enumerate}
\end{teo}

\begin{rem}
\begin{enumerate}
 \item Notice that if compose the permuting Christoffel transformations in \eqref{eq:Christoffel_chain_1} we get \eqref{eq:Christoffel_basic_1}, and composing \eqref{eq:Christoffel_chain_2} we get \eqref{eq:Christoffel_basic_3}.
 \item The basic Christoffel transformations as in \eqref{eq:Christoffel_basic_2} and \eqref{eq:Christoffel_basic_4} connect uniform tuples subject in the same \emph{gauge} class as the tuples in \eqref{eq:Christoffel_basic_1} and \eqref{eq:Christoffel_basic_3}, respectively. However, we have not been able to find the corresponding permuting Christoffel chains.
 \item Using the \emph{gauge} transformations and the Christoffel chains we are able to connect all the stochastic uniform tuples, the same happens for the semi--stochastic tuples.
 \item The chains connect the stochastic uniform tuples and semi-stochastic uniform tuples separately. So far, we have not been able to see a connection between these two sets of uniform tuples.
 \item The Christoffel formulas in Theorems \ref{teo:Christoffel} and \ref{teo:Basic Christoffel formulas} can be used to connect the multiple orthogonal polynomials corresponding to different uniform tuples.
\end{enumerate}
\end{rem}

In the theory of hypergeometric functions relations among hypergeometric functions with shifted parameters by integers are known as contiguous relations, see \cite{Andrews,Rainville,Bailey}, and particularly \cite[\S 7]{Rainville0}.
From the previous Christoffel transformations and using the explicit form of the type II multiple orthogonal polynomials we get some non trivial relations among different instances of the generalized hypergeometric function $\tensor[_3]{F}{_2}$. As we will see, the six relations happens to be contiguous relations for $\tensor[_3]{F}{_2}$.

First we get contiguous relations with three hypergeometric functions.
\begin{pro}[Three terms $_3F_2$ contiguous relations]
	For $n\in\N$, the following relations are fulfilled
	\begin{align*}
\frac{ (n+1) (3 n+2)x}{6\kappa}
\, \tensor[_3]{F}{_2}\hspace*{-3pt}\left[{\begin{NiceArray}{c}[small]-n,\; \frac{n+3}{2} ,\; \frac{n+2}{2} \\[3pt]
		\frac{4}{3},\;\frac{5}{3}\end{NiceArray}};x\right]+
\tensor[_3]{F}{_2}\hspace*{-3pt}\left[{\begin{NiceArray}{c}[small]-n-1,\; \frac{n+2}{2} , \;\frac{n+1}{2} \\[3pt]
		\frac{1}{3},\;\frac{2}{3}\end{NiceArray}};x\right] -
		\tensor[_3]{F}{_2}\hspace*{-3pt}\left[{\begin{NiceArray}{c}[small]-n,\; \frac{n+1}{2} , \;\frac{n}{2} \\[3pt]
				\frac{1}{3},\;\frac{2}{3}\end{NiceArray}};x\right] &=0,\\
\frac{	(n+1) ( n+2)x}{2\kappa}
\tensor[_3]{F}{_2}\hspace*{-3pt}\left[{\begin{NiceArray}{c}[small]-n,\; \frac{n+4}{2} , \; \frac{n+3}{2} \\[3pt]\frac{4}{3},\;\frac{5}{3}\end{NiceArray}};x\right]+
\tensor[_3]{F}{_2}\hspace*{-3pt}\left[{\begin{NiceArray}{c}[small]-n-1,\; \frac{n+2}{2} , \; \frac{n+3}{2} \\[3pt]
		\frac{1}{3},\;\frac{2}{3}\end{NiceArray}};x\right]-
	\tensor[_3]{F}{_2}\hspace*{-3pt}\left[{\begin{NiceArray}{c}[small]-n,\; \frac{n+1}{2} , \; \frac{n+2}{2} \\[3pt]
		\frac{1}{3},\;\frac{2}{3}\end{NiceArray}};x\right]&=0.
\end{align*}
\end{pro}
\begin{proof}
	From the Equations \eqref{eq:B_3F2_1} and \eqref{eq:B_3F2_4} we get
	for the uniform tuples $ \big(\frac{2}{3},\frac{1}{3}, \frac{1}{2},1\big)$ 
	and $ \big(\frac{2}{3},\frac{1}{3}, 1,\frac{3}{2}\big)$ that
$		B^{(n)}(0)=3(-\kappa)^n$ and $B^{(n)}(0)=(-\kappa)^n$, 
	respectively. Then, from \eqref{teo:Basic Christoffel formulas}, i.e. $x	 B^{(n)}_{\underline{\vec w}}(x)=B_{\vec w}^{(n+1)}(x) 
		-\frac{B_{\vec w}^{(n+1)}(0)}{B_{\vec w}^{(n)}(0)}B_{\vec w}^{(n)}(x)$
and Equations \eqref{eq:B_3F2_3} and \eqref{eq:B_3F2_6} give the result.
\end{proof}
Second, we obtain contiguous relations with four hypergeometric functions.
\begin{pro}[Four terms $_3F_2$ contiguous relations]
For $n\in\N$, the following relations are satisfied
\begin{gather}	\label{eq:hyper_relation_4_1}
			\frac{(3 n+1)x }{\kappa}
		\, \tensor[_3]{F}{_2}\hspace*{-3pt}\left[{\begin{NiceArray}{c}[small]-n,\; \frac{n+1}{2} , \;\frac{n+2}{2} \\[3pt]
				\frac{2}{3},\;\frac{4}{3}\end{NiceArray}};x\right]
		+ 3 \, \tensor[_3]{F}{_2}\hspace*{-3pt}\left[{\begin{NiceArray}{c}[small]-n-1,\; \frac{n+2}{2} , \;\frac{n+1}{2} \\[3pt]
				\frac{1}{3},\;\frac{2}{3}\end{NiceArray}};x\right]
		-6\,\tensor[_3]{F}{_2}\hspace*{-3pt}\left[{\begin{NiceArray}{c}[small]-n,\; \frac{n+1}{2} , \;\frac{n}{2} \\[3pt]
				\frac{1}{3},\;\frac{2}{3}\end{NiceArray}};x\right]
		+3\,\tensor[_3]{F}{_2}\hspace*{-3pt}\left[{\begin{NiceArray}{c}[small]-n+1,\; \frac{n}{2} , \;\frac{n-1}{2} \\[3pt]
				\frac{1}{3},\;\frac{2}{3}\end{NiceArray}};x\right]=0,\\
	\label{eq:hyper_relation_4_2}	\begin{multlined}[t][.9\textwidth]
		\frac{(n+1) (3 n+2)x}{2\kappa} 
		\, \tensor[_3]{F}{_2}\hspace*{-3pt}\left[{\begin{NiceArray}{c}[small]-n,\; \frac{n+3}{2} ,\; \frac{n+2}{2} \\[3pt]
				\frac{4}{3},\;\frac{5}{3}\end{NiceArray}};x\right]
		+(3 n+4) 
		\, \tensor[_3]{F}{_2}\hspace*{-3pt}\left[{\begin{NiceArray}{c}[small]-n-1,\; \frac{n+2}{2} , \;\frac{n+3}{2} \\[3pt]
				\frac{2}{3},\;\frac{4}{3}\end{NiceArray}};x\right]\\-2	(3 n+1) 
		\, \tensor[_3]{F}{_2}\hspace*{-3pt}\left[{\begin{NiceArray}{c}[small]-n,\; \frac{n+1}{2} , \;\frac{n+2}{2} \\[3pt]
				\frac{2}{3},\;\frac{4}{3}\end{NiceArray}};x\right]	+ (3 n-2) 
		\, \tensor[_3]{F}{_2}\hspace*{-3pt}\left[{\begin{NiceArray}{c}[small]-n+1,\; \frac{n}{2} , \;\frac{n+1}{2} \\[3pt]
				\frac{2}{3},\;\frac{4}{3}\end{NiceArray}};x\right]=0,
	\end{multlined}\\
	\label{eq:hyper_relation_4_3}	\frac{(n+1) x}{\kappa}
	\tensor[_3]{F}{_2}\hspace*{-3pt}\left[{\begin{NiceArray}{c}[small]-n,\; \frac{n+3}{2} , \; \frac{n+2}{2} \\[3pt]
			\frac{2}{3},\;\frac{4}{3}\end{NiceArray}};x\right] 
	+
	\tensor[_3]{F}{_2}\hspace*{-3pt}\left[{\begin{NiceArray}{c}[small]-n-1,\; \frac{n+2}{2} , \; \frac{n+3}{2} \\[3pt]
			\frac{1}{3},\;\frac{2}{3}\end{NiceArray}};x\right] -2 \,	
	\tensor[_3]{F}{_2}\hspace*{-3pt}\left[{\begin{NiceArray}{c}[small]-n,\; \frac{n+1}{2} , \; \frac{n+2}{2} \\[3pt]
			\frac{1}{3},\;\frac{2}{3}\end{NiceArray}};x\right] +
	\tensor[_3]{F}{_2}\hspace*{-3pt}\left[{\begin{NiceArray}{c}[small]-n+1,\; \frac{n}{2} , \; \frac{n+1}{2} \\[3pt]
			\frac{1}{3},\;\frac{2}{3}\end{NiceArray}};x\right] =0,
	\\
	\label{eq:hyper_relation_4_4}
\begin{multlined}[t][.9\textwidth]
	\frac{(n+1)(n+2) x }{2\kappa} 
	\tensor[_3]{F}{_2}\hspace*{-3pt}\left[{\begin{NiceArray}{c}[small]-n,\; \frac{n+4}{2} , \; \frac{n+3}{2} \\[3pt]\frac{4}{3},\;\frac{5}{3}\end{NiceArray}};x\right]
	+(n+2) \,
	\tensor[_3]{F}{_2}\hspace*{-3pt}\left[{\begin{NiceArray}{c}[small]-n-1,\; \frac{n+4}{2} , \; \frac{n+3}{2} \\[3pt]
			\frac{2}{3},\;\frac{4}{3}\end{NiceArray}};x\right]\\	-2 (n+1) \,
	\tensor[_3]{F}{_2}\hspace*{-3pt}\left[{\begin{NiceArray}{c}[small]-n,\; \frac{n+3}{2} , \; \frac{n+2}{2} \\[3pt]
			\frac{2}{3},\;\frac{4}{3}\end{NiceArray}};x\right]	+ n \,
	\tensor[_3]{F}{_2}\hspace*{-3pt}\left[{\begin{NiceArray}{c}[small]-n+1,\; \frac{n+2}{2} , \; \frac{n+1}{2} \\[3pt]
			\frac{2}{3},\;\frac{4}{3}\end{NiceArray}};x\right]=0.
\end{multlined}
\end{gather}
\end{pro}
\begin{proof}
For the uniform tuples $\big(\frac{2}{3},\frac{1}{3}, \frac{1}{2},1\big)$, $\big(\frac{4}{3},\frac{2}{3}, 1,\frac{3}{2}\big)$,
$\big(\frac{2}{3},\frac{1}{3}, 1,\frac{3}{2}\big)$ and $\big(\frac{4}{3},\frac{2}{3}, \frac{3}{2},2\big)$, proceeding similarly as we did
for \eqref{eq:Bn(1)}, but instead of unity the origin, instead of the type II polynomials the type I, and using dual recurrences, we get
 \begin{align}\label{eq:A10}
 	A_1^{(n)}(0) &= \frac{1}{(-\kappa)^n},
 \end{align}
From Equations \eqref{eq:permuting_Christoffel} and \eqref{eq:Jacobi_Jacobi_Piñeiro_uniform} we get
 \begin{align}
x	B^{(n)}_{\underline{\vec w}}(x)
	=B_{\vec w}^{(n+1)}(x)
+\bigg(\frac{A_{1,\vec w}^{(n-1)}(0)}{A^{(n)}_{1,\vec w}(0)} +3\kappa\bigg)B_{\vec w}^{(n)}(x)-\frac{A_{1,\vec w}^{(n+1)}(0)}{A^{(n)}_{1,\vec w}(0)}\kappa^3B_{\vec w}^{(n-1)}(x).
\end{align}
and \eqref{eq:A10} leads to
 \begin{align}\label{eq:permuting_Christoffel_uniform}
	x	B^{(n)}_{\underline{\vec w}}(x)
= B_{\vec w}^{(n+1)}(x)	+2\kappa	B_{\vec w}^{(n)}(x)	+\kappa^2 B_{\vec w}^{(n-1)}(x).
\end{align}
From \eqref{eq:permuting_Christoffel_uniform}, \eqref{eq:B_3F2_1} and \eqref{eq:B_3F2_2} we conclude \eqref{eq:hyper_relation_4_1}, Equations 
\eqref{eq:permuting_Christoffel_uniform}, \eqref{eq:B_3F2_2} and \eqref{eq:B_3F2_3} leads to \eqref{eq:hyper_relation_4_2}, 
recalling \eqref{eq:permuting_Christoffel_uniform}, \eqref{eq:B_3F2_4} and \eqref{eq:B_3F2_5} we have \eqref{eq:hyper_relation_4_3}, and finally, from \eqref{eq:permuting_Christoffel_uniform}, \eqref{eq:B_3F2_5} and \eqref{eq:B_3F2_6} we get \eqref{eq:hyper_relation_4_4}.
\end{proof}
\subsection{Type I multiple orthogonal polynomials. Explicit expressions}
In contrast with the type II multiple orthogonal polynomials, for which we have expressions in terms of the  generalized hypergeometric function $\tensor[_3]{F}{_2}$, the best expression we have for the type I multiple orthogonal polynomials is provided in Corollary \ref{cor:expressions_type_I}.  These  are computationally rather more involved than the mentioned hypergeometric ones for the type II.
Fortunately, as we will show now, for the uniform tuples much  simpler expressions do exist for these hypergeometric type I multiple orthogonal polynomials. The idea is to use generating functions for constant coefficient recurrence relations.  
The uniform  order four  homogeneous linear recurrence relation satisfied 
 by  type I multiple orthogonal polynomials discussed here reads as follows
	\begin{align}
	v_n+3\kappa v_{n+1}+3\kappa^2 v_{n+2}+\kappa^3v_{n+3}&= x v_{n+1}, && n \in \mathbb N_0.
\label{eq:uniform_recurrenceI}
\end{align}
Notice that, if we denote 
\begin{align*}
	\kappa_*&:=\frac{1}{\kappa}, &x_*&:=\kappa_*^3x,
\end{align*}
  recurrence \eqref{eq:uniform_recurrenceI} can be written as
\begin{align}\label{eq:uniform_recurrenceI'}
	\kappa_*^3	v_n+3\kappa_*^2 v_{n+1}+3\kappa_*v_{n+2}+v_{n+3}&= x_* v_{n+1}, && n \in \mathbb N_0 .
\end{align}
Let us consider the corresponding generating function
\begin{align*}	V(t,x)&:=\sum_{n=0}^\infty v_n(x)t^n.
	\end{align*}
\begin{pro}\label{pro:V}
	The generating function is explicitly given by the rational function
		\begin{align}
\label{eq:V}
		V&=\frac{(v_2+3\kappa_*v_1+(3\kappa_*^2-x_*)v_0)t^2+(v_1+3\kappa_* v_0)t+v_0}{(1+\kappa_* t)^3-x_*t^2}.
	\end{align}
\end{pro}
\begin{proof}
For \eqref{eq:uniform_recurrenceI} we 	multiply \eqref{eq:uniform_recurrenceI'} by $t^{n+3}$ so that
\begin{align*}
(\kappa_* t)^3 t^n	v_n+3 (\kappa_* t)^2 t^{n+1} v_{n+1} + 3 (\kappa_* t) t^{n+2}v_{n+2} + t^{n+3}v_{n+3} & = x_*t^2 t^{n+1}v_{n+1}, 
\end{align*}
and sum up, for the type I generating function we get 
\begin{align*}
	(\kappa_* t)^3 V+3	(\kappa_* t)^2 (V-v_0)+3	(\kappa_* t) (V-v_0-v_1t)+V-v_0-v_1t-v_2t^2&= x_*t^2 (V-v_0), 
\end{align*}
so that \eqref{eq:V} immediately follows. 
\end{proof}

To get the $v_n(x)$ back we need to expand Equation \eqref{eq:V} in power series in $t$.

\begin{lemma}[Power series expansions]\label{lem:trinomial}
	For $|3\kappa t+3(\kappa^2-x) t^2+\kappa^3t^3|<1$, we find
	\begin{align*}
	\frac{1}{{(1+\kappa_* t)^3-x_*t^2}}&=\sum_{n=0}^\infty 
	e_n(x)t^n , &e_n(x)&:=(-\kappa)^{-n}\hspace*{-11pt}\sum_{m_1+2m_2+3m_3=n}	\hspace*{-10pt}(-1)^{m_2} \binom{m_1+m_2+m_3}{m_1,m_2,m_3}3^{m_1+m_2}\Big(1-\frac{ x}{3\kappa}\Big)^{m_2}.
\end{align*}
	Here $\binom{k}{k_1,k_2,k_3}=\frac{k!}{k_1!k_2!k_3!}$ is the trinomial coefficient.
\end{lemma}
\begin{proof}
As
	\begin{align*}
		{(1+\kappa_* t)^3-x_*t^2}= 1+3\kappa_* t+3\Big(1-\frac{x_*}{3\kappa_*^2}\Big) \kappa_*^2t^2+\kappa_*^3t^3,
	\end{align*}
	for $t,x$ such that $|3\kappa_* t+3(\kappa_*^2-x_*) t^2+\kappa_*^3t^3|<1$, we can write 
	\begin{align*}
		\frac{1}{(1+\kappa_* t)^3-x_*t^2}= \frac{1}{1+3\kappa_* t+3\big(1-\frac{x_*}{3\kappa_*^2}\big)\kappa_*^2 t^2+\kappa_*^3t^3}=\sum_{m=0}^\infty \Big(-3\kappa_* t-3\Big(1-\frac{x_*}{3\kappa_*^2}\Big)\kappa_*^2t^2-\kappa_*^3t^3\Big)^m,
	\end{align*}
	and he multinomial theorem gives result.
\end{proof}

Using the theory of linear Diophantine equations we find:
\begin{pro}\label{pro:the polynomials e_n}
The polynomials $e_n$, for $n=3p,3p+1,3p+2$ with $p$ a nonnegative integer, are given by
\begin{align*}
	e_{3p}&:=		\frac{1}{(-\kappa)^{3p}}\sum_{l=0}^p 27^l\sum_{k=0}^{\lfloor\frac{3l}{2}\rfloor}\frac{(-1)^k}{3^{k}} \binom{p+2l-k}{3l-2k,k,p-l}\Big(1-\frac{ x}{3\kappa}\Big)^{k},\\
	e_{3p+1}&:=\frac{3}{(-\kappa)^{3p+1}}	\sum_{l=0}^p27^l\sum_{k=0}^{\lfloor\frac{3l+1}{2}\rfloor}\frac{(-1)^k}{3^{k}} \binom{p+2l+1-k}{3l+1-2k,k,p-l}\Big(1-\frac{ x}{3\kappa}\Big)^{k},\\
	e_{3p+2}&:=\frac{9}{(-\kappa)^{3p+2}}	\sum_{l=0}^p27^l\sum_{k=0}^{\lfloor\frac{3l}{2}\rfloor+1}\frac{(-1)^k}{3^{k}} \binom{p+2l+2-k}{3l+2-2k,k,p-l}\Big(1-\frac{ x}{3\kappa}\Big)^{k}.
\end{align*}
\end{pro}
\begin{proof}
The restriction in the sum
$	m_1+2m_2+3m_3=n$
can be understood as a linear Diophantine equation for three unknown positive integers $(m_1,m_2,m_3)$. As $\text{gcd}(1,2,3)=1$, this can be solved as follows, depending whether $n=3p,3p+1,3p+2$ with $p$ a nonnegative integer:
\begin{align*}
	n&=3p: &\begin{pNiceMatrix}
		m_1\\m_2\\m_3
	\end{pNiceMatrix}&=\begin{pNiceMatrix}
		3l-2k\\k\\p-l
	\end{pNiceMatrix}, & l&\in\{0,1,\dots,p\}, & k\in\Big\{0,1,\dots,\Big\lfloor\frac{3l}{2}\Big\rfloor\Big\},\\
	n&=3p+1: &\begin{pNiceMatrix}
		m_1\\m_2\\m_3
	\end{pNiceMatrix}&=\begin{pNiceMatrix}
	3l+1-2k\\k\\p-l
	\end{pNiceMatrix}, & l&\in\{0,1,\dots,p\}, & k\in\Big\{0,1,\dots,\Big\lfloor\frac{3l+1}{2}\Big\rfloor\Big\},\\
	n&=3p+2: &\begin{pNiceMatrix}
		m_1\\m_2\\m_3
	\end{pNiceMatrix}&=\begin{pNiceMatrix}
		3l+2-2k\\k\\p-l
	\end{pNiceMatrix}, & l&\in\{0,1,\dots,p\}, & k\in\Big\{0,1,\dots,\Big\lfloor\frac{3l}{2}\Big\rfloor+1\Big\}.
\end{align*}
\end{proof}

\begin{rem}
	The first twelve polynomials are
\begin{gather*}
\begin{aligned}
	e_0 &=1,& e_1&=-\frac{3}{\kappa},& e_2&=\frac{6 \kappa+x}{\kappa^3},&e_3&=-\frac{2 (5 \kappa+3 x)}{\kappa^4},&
e_4&=\frac{15 \kappa^2+21\kappa x +x^2}{\kappa^6},&e_5&=-\frac{21 \kappa^2+56\kappa x +9 x^2}{\kappa^7},
\end{aligned}\\
\begin{aligned}
	e_6&=\frac{28 \kappa^3+126 \kappa^2 x+45 \kappa x^2 +x^3}{\kappa^9},&e_7&=-\frac{3 \left(12 \kappa^3+84 \kappa^2 x+55 \kappa x^2+4 x^3\right)}{\kappa^{10}},&e_8&=\frac{45 \kappa^4+462 \kappa^3x +495 \kappa^2 x^2+78 \kappa x^3 +x^4}{\kappa^{12}},
\end{aligned}\\
\begin{aligned}
	e_9&=-\frac{55 \kappa^4+792 \kappa^3 x^3 +1287 \kappa^2 x^2 +364 \kappa x^3 +15 x^4}{\kappa^{13}},&
e_{10}&=\frac{66 \kappa^5+1287 \kappa^4 x+3003 \kappa^3 x^2+1365 \kappa^2 x^3+120 \kappa x^4 +x^5}{\kappa^{15}},
\end{aligned}\\
e_{11}=-\frac{78 \kappa^5+2002 \kappa^4 x+6435\kappa^3 x^2 +4368 \kappa^2 x^3+680\kappa x^4 +18 x^5}{\kappa^{16}}.
\end{gather*}

\end{rem}

\begin{lemma}\label{lem:A}
For $n\in\N$, the multiple orthogonal polynomials are
	\begin{align*}
		A_1^{(n+2)}(x)&=\Big(A_1^{(2)}(x)+\frac{3}{\kappa} A_1^{(1)}+\frac{3}{\kappa^2} -\frac{x}{\kappa^3}\Big)e_n(x)+\Big(A_1^{(1)}+\frac{3}{\kappa}\Big)e_{n+1}(x)+e_{n+2}(x),\\
	A_2^{(n+2)}(x)&=\Big(A_2^{(2)}+\frac{3}{\kappa} A_2^{(1)}\Big)e_n(x)+A_2^{(1)}e_{n+1}(x),
	\end{align*}
with initial conditions given by
\begin{center}
	\begin{NiceTabular}{|c|c|c|}[rules/color=[Gray]{0.75},
	code-before = \cellcolor{Gray!15}{7-2,8-2}\cellcolor{MidnightBlue!15}{5-2,9-2,11-2}\cellcolor{Maroon!15}{3-3,9-3,2-3,8-3}
	\cellcolor{Emerald!15}{6-3,12-3,10-3}	\cellcolor{Plum!15}{7-3,11-3,13-3}	\cellcolor{RedOrange!15}{4-3,5-3}
		]\toprule
		$(a,b,c,d)$ \rule[-2ex]{0pt}{0pt}
		&
		$\left\{A_{1}^{(0)},A_1^{(1)},A_1^{(2)}\right\}$ 
		& 
		$\left\{A_{2}^{(0)},A_2^{(1)},A_2^{(2)}\right\}$ \\\toprule
		$\left(\frac{1}{3},\frac{2}{3},\frac{1}{2},1\right\} $&
		$\left\{1,\frac{27}{2},\frac{6561 x}{64}-\frac{3645}{16}\right\} $&
		$\left\{0,-\frac{27}{2},\frac{729}{4} \right\} $\\\midrule
		$\left(\frac{2}{3},\frac{1}{3},\frac{1}{2},1\right) $&$\left\{1,-\frac{27}{4},\frac{6561 x}{64}+\frac{729}{16}\right\}$ &$\left\{0,\frac{27}{4},-\frac{729}{8}\right\} $\\\midrule
		$ \left(\frac{2}{3},\frac{4}{3},1,\frac{3}{2}\right) $ 
		& $\left\{1,\frac{27}{2},\frac{19683 x}{64}-\frac{5103}{8}\right\}$ & $ \left\{0,-\frac{27}{2},\frac{3645}{8}\right\} $\\\midrule
		$\left(\frac{4}{3},\frac{2}{3},1,\frac{3}{2}\right)$ & $\left\{1,-\frac{27}{4},\frac{19683 x}{64}+\frac{729}{16}\right\} $ &
		$\left\{0,\frac{27}{4},-\frac{3645}{16}\right\} $\\\midrule
		$	\left(\frac{4}{3},\frac{5}{3},\frac{3}{2},2\right) $ & 
		$	\left\{1,\frac{135}{4},\frac{19683 x}{64}-\frac{3645}{4}\right\}$ & $	\left\{0,-\frac{135}{4},\frac{10935}{16}\right\} $\\\midrule
		$	\left(\frac{5}{3},\frac{4}{3},\frac{3}{2},2\right) $ &$	\left\{1,-27,\frac{19683 x}{64}+\frac{5103}{16}\right\}$ &
		$ \left\{0,27,-\frac{2187}{4}\right\} $\\
	\arrayrulecolor{black!90}	\midrule
		$\left(\frac{1}{3},\frac{2}{3},1,\frac{3}{2}\right)$ 
		& $	\left\{1,-27,\frac{19683 x}{64}+\frac{5103}{16}\right\}$ & $ \left\{0,27,-\frac{729}{2}\right\} $\\
			\arrayrulecolor{gray!75}\midrule
		$\left(\frac{2}{3},\frac{1}{3},1,\frac{3}{2}\right) $ & $\left\{1,-\frac{27}{4},\frac{19683 x}{64}+\frac{729}{16}\right\} $ &
		$\left\{0,\frac{27}{4},-\frac{729}{8}\right\} $\\
		\midrule
		$\left(\frac{2}{3},\frac{4}{3},\frac{3}{2},2\right)$
		&$\left\{1,-\frac{135}{2},\frac{19683 x}{64}+\frac{5103}{4}\right\} $
		&$\left\{0,\frac{135}{2},-\frac{10935}{8}\right\}$\\\midrule
		$\left(\frac{4}{3},\frac{2}{3},\frac{3}{2},2\right) $ &$\left\{1,-\frac{27}{4},\frac{19683 x}{64}+\frac{729}{16}\right\}$ &
		$\left\{0,\frac{27}{4},-\frac{2187}{16}\right\} $\\\midrule
		$\left(\frac{4}{3},\frac{5}{3},2,\frac{5}{2}\right)$ & $\left\{1,-\frac{135}{4},\frac{19683 x}{64}+\frac{2187}{4}\right\}$ &
		$\left\{0,\frac{135}{4},-\frac{10935}{16}\right\}$\\ \midrule
		$ \left(\frac{5}{3},\frac{4}{3},2,\frac{5}{2}\right)$ & $\left\{1,-\frac{27}{2},\frac{19683 x}{64}+\frac{2187}{16}\right\} $ &
		$\left\{0,\frac{27}{2},-\frac{2187}{8}\right\}$
		\\\bottomrule
	\end{NiceTabular}
\end{center}
\end{lemma}
\begin{proof}
	A direct consequence of Proposition \ref{pro:V} and Lemma \ref{lem:trinomial} is
		\begin{align*}
v_{n+2}(x)&=\Big(v_2(x)+\frac{3}{\kappa} v_1(x)+\Big(\frac{3}{\kappa^2} -\frac{x}{\kappa^3}\Big)v_0(x)\Big)e_n(x)+\Big(v_1(x)+\frac{3}{\kappa} v_0(x)\Big)e_{n+1}(x)+v_0(x)e_{n+2}(x).
	\end{align*}
Hence, applying it to the multiple orthogonal polynomials we get the result.The initial conditions are gotten directly from the Gauss--Borel factorization problem.
\end{proof}
\begin{rem}
These relations for $n=0$ are identities.
\end{rem}

\begin{rem}
A remarkable fact is that, having different sets of weights there are sequences of type~I multiple orthogonal polynomials that coincide, or are multiples.
In the above table, we have identified such classes with the same color.
Namely, $\left(\frac{4}{3},\frac{2}{3},1,\frac{3}{2}\right)$, 
$\left(\frac{2}{3},\frac{1}{3},1,\frac{3}{2}\right) $
and
$\left(\frac{4}{3},\frac{2}{3},\frac{3}{2},2\right) $
have the same sequence
$\big\{A_1^{(n)}\big\}_{n=0}^\infty$, the same happens with
$\left(\frac{5}{3},\frac{4}{3},\frac{3}{2},2\right) $
and
$\left(\frac{1}{3},\frac{2}{3},1,\frac{3}{2}\right)$.
For
$\left(\frac{2}{3},\frac{1}{3},\frac{1}{2},1\right) $
and
$\left(\frac{2}{3},\frac{1}{3},1,\frac{3}{2}\right) $
we have the same sequence
$\big\{A_2^{(n)}\big\}_{n=0}^\infty$
and
$\left(\frac{1}{3},\frac{2}{3},1,\frac{3}{2}\right)$
has the same sequence multiplied by $4$, and
$\left(\frac{1}{3},\frac{2}{3},\frac{1}{2},1\right) $
the same sequence multiplied by $-2$.
For
$\left(\frac{4}{3},\frac{5}{3},\frac{3}{2},2\right) $
and
$\left(\frac{4}{3},\frac{5}{3},2,\frac{5}{2}\right)$
we have opposite sequences
$\big\{A_2^{(n)}\big\}_{n=0}^\infty$,
moreover the sequence for
$\left(\frac{2}{3},\frac{4}{3},\frac{3}{2},2\right)$
is a half of the previous one.
Similarly, the sequences
$\big\{A_2^{(n)}\big\}_{n=0}^\infty$
for 
$\left(\frac{4}{3},\frac{2}{3},\frac{3}{2},2\right) $,
$ \left(\frac{5}{3},\frac{4}{3},2,\frac{5}{2}\right)$
and
$\left(\frac{5}{3},\frac{4}{3},\frac{3}{2},2\right) $,
are obtained by multiplying by two the previous one.
Finally, the~$A_2$ sequence for
$\left(\frac{4}{3},\frac{2}{3},1,\frac{3}{2}\right)$
is $-2$ times the $A_2$ sequence for 	
$ \left(\frac{2}{3},\frac{4}{3},1,\frac{3}{2}\right) $.
All these symmetries reduce the number essentially different sequences
$\big\{A_2^{(n)}\big\}_{n=0}^\infty$ to $4$
types.
\end{rem}

\begin{rem}
A similar (but slightly more complicated) construction holds for the type II multiple orthogonal polynomials. However, such a construction provides more involved expressions than the $\tensor[_2]{F}{_3}$ hypergeometric expressions found in \cite{lima_loureiro}. 
\end{rem}

\begin{teo}\label{teo:uniform_type_I}
For $n\in\N_0$,	 the type I multiple orthogonal polynomials corresponding to uniform tuples are given in terms of the polynomials $\{e_n\}_{n=0}^\infty$ in Proposition \ref{pro:the polynomials e_n} in the following table
\begin{center}
	\begin{NiceTabular}{|c|c|c|}[rules/color=[Gray]{0.75},
		code-before = \cellcolor{Gray!15}{7-2,8-2}\cellcolor{MidnightBlue!15}{5-2,9-2,11-2}\cellcolor{Maroon!15}{3-3,9-3,2-3,8-3}
		\cellcolor{Emerald!65!Plum!15}{6-3,12-3,10-3,7-3,11-3,13-3}	\cellcolor{RedOrange!15}{4-3,5-3}
		]\toprule
		$(a,b,c,d)$\rule[-1.2ex]{0pt}{0pt}
		&
		$A_1^{(n+2)}$ 
		& 
				$A_2^{(n+2)}$  \\   \toprule
		$\left(\frac{1}{3},\frac{2}{3},\frac{1}{2},1\right) $&
		$\big(\frac{729}{4}-\frac{6561 x}{32}\big)e_n+\frac{135}{4}e_{n+1}+e_{n+2}$&
		$ -\frac{729}{8}e_ n -\frac{27}{2}e_{n+1} $\\\midrule
		$\left(\frac{2}{3},\frac{1}{3},\frac{1}{2},1\right) $&
		$\big(\frac{729}{16}-\frac{6561 x}{32}\big)e_n+\frac{27}{2}e_{n+1}+e_{n+2}$ &
		$\frac{729}{16}e_n+\frac{27}{4}e_{n+1}$\\\midrule
		$ \left(\frac{2}{3},\frac{4}{3},1,\frac{3}{2}\right) $ &
		$-\frac{3645}{16} e_n+\frac{135}{4} e_{n+1}+e_{n+2}$ & 
		$ \frac{729}{4}e_n -\frac{27}{2}e_{n+1}$\\\midrule
		$\left(\frac{4}{3},\frac{2}{3},1,\frac{3}{2}\right)$ & 
		$\frac{729}{16}e_n+ \frac{27}{2} e_{n+1}+e_{n+2}$ &
		$-\frac{729}{8} e_n+\frac{27}{4} e_{n+1} $\\\midrule
		$	\left(\frac{4}{3},\frac{5}{3},\frac{3}{2},2\right) $ & 
		$	-\frac{729}{8}e_n+ 54e_{n+1}+e_{n+2}$ &
		$-\frac{135}{4}e_{n+1} $\\\midrule
		$	\left(\frac{5}{3},\frac{4}{3},\frac{3}{2},2\right) $ &
		$	 -\frac{729}{8}e_n -\frac{27}{4}e_{n+1}+e_{n+2}$ &
		$ 27e_{n+1}$\\
		\arrayrulecolor{black!90}	\midrule
		$\left(\frac{1}{3},\frac{2}{3},1,\frac{3}{2}\right)$&
		$	-\frac{729}{8}e_n -\frac{27}{4}e_{n+1}+e_{n+2}$ & 
		$ \frac{729}{4}e_n+ 27e_{n+1} $\\
		\arrayrulecolor{gray!75}\midrule
		$\left(\frac{2}{3},\frac{1}{3},1,\frac{3}{2}\right) $ &
		$\frac{729}{16}e_n+\frac{27}{2} e_{n+1}+e_{n+2} $ &
		$\frac{729}{16} e_n+\frac{27}{4}e_{n+1}$\\
		\midrule
		$\left(\frac{2}{3},\frac{4}{3},\frac{3}{2},2\right)$&
		$ \frac{729}{16}e_n-\frac{189}{4}e_{n+1} +e_{n+2}$&
		$\frac{135}{2} e_{n+1}$\\\midrule
		$\left(\frac{4}{3},\frac{2}{3},\frac{3}{2},2\right) $ &
		$\frac{729}{16} e_n+\frac{27}{2}e_{n+1}+e_{n+2}$ &
		$\frac{27}{4} e_{n+1}$\\\midrule
		$\left(\frac{4}{3},\frac{5}{3},2,\frac{5}{2}\right)$ & 
		$-\frac{27}{2}e_{n+1}+e_{n+2}$ &
		$\frac{135}{4} e_{n+1}$\\ \midrule
		$ \left(\frac{5}{3},\frac{4}{3},2,\frac{5}{2}\right)$ & 
		$\frac{27}{4} e_{n+1}+e_{n+2}$ &
		$\frac{27}{2} e_{n+1}$
		\\\bottomrule
	\end{NiceTabular}
\end{center}
\end{teo}

\section{Summation formulas}\label{S:Summations}
In this section we assume \eqref{eq:region_parameters_pochhammer_perfect}, so that the system of weights is perfect, and the moments are given by~\eqref{eq:moments}. For $m\in\N_0$, we study the following moments of the type II polynomials, that we call type II generalized moments
\begin{align*}
 \eta^{(2n)}_{m,1}&:= \int_0^1B^{(2n)}(x)x^{n+m}w_1(x)\d\mu(x),&
 \eta^{(2n+1)}_{m,1}&:= \int_0^1B^{(2n+1)}(x)x^{n+m}w_1(x)\d\mu(x),\\
 \eta^{(2n)}_{m,2}&:= \int_0^1B^{(2n)}(x)x^{n+m}w_2(x)\d\mu(x),&
 \eta^{(2n+1)}_{m,2}&:= \int_0^1B^{(2n+1)}(x)x^{n+m}w_2(x)\d\mu(x).
\end{align*}
These coefficients are key objects in the type II, Hermite--Padé approximation problem. In fact,
departing from $\{ B^{(n)} \}$, multiple orthogonal polynomials of type II in the step-line, with respect to a system of weights $w_1$, $w_2$, on $[0,1]$, we define the system of functions of second kind by
\begin{align*}
f^{(2n)}_1 (z) & = \int_{0}^1 \frac{B^{(2n)} (x)}{z-x} w_1 (x) \, \d x
= \sum_{k=0}^\infty \frac{\eta_{k,1}^{(2n)}}{z^{n+k}} , &
f^{(2n+1)}_1 (z) & = \int_{0}^1 \frac{B^{(2n+1)} (x)}{z-x} w_1 (x) \, \d x
= \sum_{k=0}^\infty \frac{\eta_{k+1,1}^{(2n+1)}}{z^{n+k+1}} , \\
f^{(2n)}_2 (z) & = \int_{0}^1 \frac{B^{(2n)} (x)}{z-x} w_2 (x) \, \operatorname{d} x= \sum_{k=0}^\infty \frac{\eta_{k,2}^{(2n)}}{z^{n +k}} , &
f^{(2n+1)}_2 (z) & = \int_{0}^1 \frac{B^{(2n+1)} (x)}{z-x} w_2 (x) \, \operatorname{d} x= \sum_{k=0}^\infty \frac{\eta_{k,2}^{(2n)}}{z^{n +k}} ,
\end{align*}
with the series converging uniformly on compact sets of $\mathbb C \setminus [0,1]$. In the previous set of equations $n\in\N_0$.
For $n=0$ these functions are 
\begin{align*} 
\mathscr S_{w_1}(z) = f^{(0)}_1 (z) & = \int_{0}^1 \frac{ w_1 (x)}{z-x} \, \d x ,&
\mathscr S_{w_2}(z) = f^{(0)}_2 (z) & = \int_{0}^1 \frac{ w_2 (x)}{z-x} \, \d x,
\end{align*}
the Stieltjes--Markov transforms of the weights. We can also introduce
 $\{ P^{(n) }_j \} $, $j =1,2$ 
\begin{align*}
 P^{(n)}_1 (z) & = \int_{0}^1 \frac{B^{(n)} (z)-B^{(n)} (x)}{z-x} w_1 (x) \, \d x , & n &\in \mathbb N _0,&
 P^{(n)}_2 (z) & = \int_{0}^1 \frac{B^{(n)} (z)-B^{(n)} (x)}{z-x} w_2 (x) \, \d x , & n &\in \mathbb N ,
\end{align*}
which are polynomials with $\deg P^{(n)}<n$, called the associated polynomials with respect to $\{ B^{(n)} \}$ and the system of weights $w_1$, $w_2$ (cf. \cite{nikishin_sorokin}).
Then, for the Stieltjes--Markov transforms of the weights we have the following representation, or Hermite--Padé approximation problem of type~II.
\begin{align*}
B^{(n)} (z) \mathscr S_{w_1}(z) - P^{(n) }_1 (z) & = f^{(n)}_1 (z) , &
B^{(n)} (z) \mathscr S_{w_2}(z)- P^{(n) }_2 (z) & = f^{(n)}_2 (z).
\end{align*}
Thus we have simultaneous rational approximants $\frac{P^{(n) }_j}{B^{(n)}}$ to $\mathscr S_{w_j}(z)$ for $j=1,2$.
Hence, $f^{(n)}_j$ are the remainders of the interpolation conditions defining the type II Hermite--Padé approximants to the Markov functions. 

We have already studied the first type II generalized moments, in fact
\begin{align}\label{eq:H_eta}
 \eta_{0,1}^{(2n)}&=H_{2n}, & \eta^{(2n+1)}_{0,2}&=H_{2n+1}.
\end{align}
Notice that by orthogonality relations $\eta^{(2n+1)}_{0,1}=0$.
The nontrivial first two moments can be evaluated by means of a Karlsson--Minton formula, see \cite{lima_loureiro}; i.e.,
\begin{align}\label{eq:eta_h_pochhammer}
 \eta_{0,1}^{(2n)}&=\frac{(2n)!(a)_{2n}(b)_{2n}(d-a)_{n}(d-b)_n}{(c)_{3n}(d)_{3n}(d+n-1)_{2n}}, &
 \eta_{0,2}^{(2n+1)}&=\frac{(2n+1)!(a)_{2n+1}(b+1)_{2n}(c-a+1)_n(c-b)_{n+1}}{(c+1)_{3n+1}(c+n)_{2n+1}(d)_{3n+1}},
\end{align}

Now, using a recent extension of the Karlsson--Minton summation formula due to Karp and Prilepkina we can perform a similar evaluation for all the remaining nontrivial type II generalized moments.
\begin{pro}[Type II moment summations]\label{pro:summations_Padé}
Let us assume that $c-d\not\in\Z$. Then, the following relations hold~true
\begin{enumerate}
 \item 
 For $m\in\N$ and $(\theta_1,\ldots,\theta_{2m+1})=(c,c+1,\ldots,c+m-1,d-1,d,\ldots,d+m-1)$:
 \begin{align}\label{eq:summation_product_2n_k_1}
\eta^{(2n)}_{m,1}= 
 \frac{(2n)!(a)_{2n}(b)_{2n}}{(c)_{3n}(d)_{3n-1} }
 \,
 \sum_{l=1}^{2m+1}\frac{(\theta_l-a-m+1)_{n+m}(\theta_l-b-m+1)_{n+m}}{\prod_{i\ne l}(\theta_i-\theta_l) (\theta_l+n)(\theta_l+n+1)_{2n}}.
 \end{align}
\item For $m\in\N$ and $(\theta_1,\ldots,\theta_{2m})=(c,c+1,\ldots,c+m-1,d,d+1,\ldots,d+m-1)$:
\begin{align}\label{eq:summation_product_2n+1_k_1}
\eta^{(2n+1)}_{m,1}= - \frac{(2n+1)!(a)_{2n+1}(b)_{2n+1}}{(c)_{3n+1}(d)_{3n+1}}
 \sum_{l=1}^{2m}\frac{(\theta_l-a-m+1)_{n+m}(\theta_l-b-m+1)_{n+m}}{\prod_{i\ne l}(\theta_i-\theta_l) (\theta_l+n)(\theta_l+n+1)_{2n+1}}.
\end{align}
\item For $m\in\N_0$ and $(\theta_1,\ldots,\theta_{2m+2})=(c,c+1,\ldots,c+m,d-1,d,\ldots,d+m-1)$
\begin{align}\label{eq:summation_product_2n_k_2}
\eta^{(2n)}_{m,2}= 
- \frac{c}{b}\frac{(2n)!(a)_{2n}(b)_{2n}}{(c)_{3n}(d)_{3n-1}}
 \sum_{l=1}^{2m+2}\frac{(\theta_l-a-m+1)_{n+m}(\theta_l-b-m)_{n+m+1}}{\prod_{i\ne l}(\theta_i-\theta_l)(\theta_l+n)(\theta_l+n+1)_{2n}}
\end{align}
\item For $m\in\N$ and $(\theta_1,\ldots,\theta_{2m+1})=(c,c+1,\ldots,c+m,d,d+1,\ldots,d+m-1)$
\begin{align}\label{eq:summation_product_2n+1_k_2}
\eta^{(2n+1)}_{m,2}= 
\frac{c}{b} \frac{(2n+1)!(a)_{2n+1}(b)_{2n+1}}{(c)_{3n}(d)_{3n+1}}
\sum_{l=1}^{2m+1}\frac{(\theta_l-a-m+1)_{n+m}(\theta_l-b-m)_{n+m+1}}{\prod_{i\ne l}(\theta_i-\theta_l)(\theta_l+n)(\theta_l+n+1)_{2n+1}}.
\end{align}
\end{enumerate}
\end{pro}
\begin{proof}
According to (3.6) in \cite{lima_loureiro} we have
\begin{align}\label{eq:product_2n_k_1}
 \int_0^1B^{(2n)}(x)x^kw_1(x)\d\mu(x)&= 
 \frac{(a)_{2n}(b)_{2n}(a)_k(b)_k}{(c+n)_{2n}(d+n-1)_{2n} (c)_k(d)_k}\,
\tensor[_5]{F}{_4}\hspace*{-3pt}\left[{\begin{NiceArray}{c}[small]-2n,\;a+k,\;b+k,\; c+n,\;d+n-1 \\a,\;b,\;c+k,\;d+k\end{NiceArray}};1\right],
 \\\notag
 \int_0^1B^{(2n+1)}(x)x^kw_1(x)\d\mu(x)&= -
 \frac{(a)_{2n+1}(b)_{2n+1}(a)_k(b)_k}{(c+n)_{2n+1}(d+n)_{2n+1} (c)_k(d)_k}\,
\tensor[_5]{F}{_4}\hspace*{-3pt}\left[{\begin{NiceArray}{c}[small]-2n-1,\;a+k,\;b+k,\; c+n,\;d+n \\a,\;b,\;c+k,\;d+k\end{NiceArray}};1\right],
 \\
 \notag
 \int_0^1B^{(2n)}(x)x^kw_2(x)\d\mu(x)&= 
 \frac{(a)_{2n}(b)_{2n}(a)_k(b+1)_k}{(c+n)_{2n}(d+n-1)_{2n} (c+1)_k(d)_k}\,
\tensor[_5]{F}{_4}\hspace*{-3pt}\left[{\begin{NiceArray}{c}[small]-2n,\;a+k,\;b+k+1,\; c+n,\;d+n-1 \\a,\;b,\;c+k+1,\;d+k\end{NiceArray}};1\right],
 \\
 \notag
 \int_0^1B^{(2n+1)}(x)x^kw_2(x)\d\mu(x)&= -
 \frac{(a)_{2n+1}(b)_{2n+1}(a)_k(b+1)_k}{(c+n)_{2n+1}(d+n)_{2n+1} (c+1)_k(d)_k}\,
\tensor[_5]{F}{_4}\hspace*{-3pt}\left[{\begin{NiceArray}{c}[small]-2n-1,\;a+k,\;b+k+1,\; c+n,\;d+n \\a,\;b,\;c+k+1,\;d+k\end{NiceArray}};1\right],
\end{align}

To evaluate these generalized hypergeometrical functions at unity we use the Karp--Prilepkina extension on the Karlsson--Minton summation formulas. For $b_1,b_2\in\C$, $N,p_1,p_2,m_1,m_2\in\N$, let 
\begin{align*}
 \boldsymbol{\beta}=(b_1,b_1+1,\cdots, b_1+p_1-1,b_2+b_2+1,\ldots, b_2+p_2-1)=(\beta_1,\ldots,\beta_{p_1+p_2})
\end{align*}
with $\beta_i\neq\beta_j$ for $i\neq j$, and assume that $p_1+p_2+N-m_1-m_2>0$. Then, 
a particular instance of \cite[Theorem~2.2]{Karp_Prilepkina} gives the following summation formula
\begin{align*}
{}_{5}F_{4}\hspace*{-3pt}\left[{\begin{NiceArray}{c}[small]-N,\;b_1,\;b_2,\; f_1+m_1,\;f_2+m_2 \\b_1+p_1,\;b_2+p_2,\;f_1\;f_2\end{NiceArray}};1\right]=
N!\frac{(b_1)_{p_1}(b_2)_{p_2}}{(f_1)_{m_1}(f_2)_{m_2}}\sum_{l=1}^{p_1+p_2}\frac{(f_1-\beta_l)_{m_1}(f_2-\beta_l)_{m_2}}{\prod_{i\ne l}(\beta_i-\beta_l) \beta_l(\beta_l+1)_{N}}.
\end{align*}
For \eqref{eq:product_2n_k_1} we take $m_1=m_2=k$, $f_1=a$, $f_2=b$, $N=2n$, $b_1=c+n$, $b_2=d+n-1$, $p_1=k-n$ and $p_2=k-n+1$; with
$
 \boldsymbol{\beta}=(c+n,c+n+1,\ldots,c+k-1,d+n-1,d+n,\ldots,d+k-1)$.
If we assume $c-d\ne \Z$ all the $\beta$'s are different, but this is not the only case with all these coefficients distinct. In this case
$
p_1+p_2+N-m_1-m_2=2(k-n)+1+2n-2k=1>0
$
and the Karp--Prilepkina summation formula applies:
\begin{align*}
\tensor[_5]{F}{_{4}}\hspace*{-3pt}\left[{\begin{NiceArray}{c}[small]-2n,\;a+k,\;b+k,\; c+n,\;d+n-1 \\a,\;b,\;c+k,\;d+k\end{NiceArray}};1\right]=
 (2n)!
 \frac{(c+n)_{k-n}(d+n-1)_{k-n+1}}{(a)_{k}(b)_k}
 \sum_{l=1}^{2(k-n)+1}\frac{(a-\beta_l)_{k}(b-\beta_l)_{k}}{\prod_{i\ne l}(\beta_i-\beta_l) \beta_l(\beta_l+1)_{2n}},
\end{align*}
and taking $m=k-n$, Equation \eqref{eq:summation_product_2n_k_1} follows.

The remaining equations follow similarly.
\end{proof}

\begin{rem}
As examples of the previous summations we have
 \begin{enumerate}
 \item For $m=0$ and $(\theta_1,\theta_2)=(c,d-1)$:
 \begin{align*}
 \eta^{(2n)}_{0,2}= -\frac{c}{b}\frac{(2n)!(a)_{2n}(b)_{2n}}{(c)_{3n}(d)_{3n-1}(d-c-1)}\Bigg(
 \frac{(c-a+1)_n(c-b)_{n+1}}{(c+n)(c+n+1)_{2n}}- \frac{(d-a)_n(d-1-b)_{n+1}}{(d-1-c)(d-n)_{2n}}
 \Bigg).
 \end{align*}
 \item For $m=0$ and $(\theta_1,\theta_2)=(c,d)$:
 \begin{align*}
 \eta^{(2n+1)}_{1,1}=-
 \frac{(2n+1)!(a)_{2n+1}(b)_{2n+1}}{(c)_{3n+1}(d)_{3n+1}(d-c)}\Bigg(
 \frac{(c-a)_{n+1}(c-b)_{n+1}}{(c+n)(c+n+1)_{2n+1}}- \frac{(d-a)_{n+1}(d-b)_{n+1}}{(d+n)(d+n+1)_{2n+1}}
 \Bigg).
 \end{align*}
 
 \item For $m=1$ and $(\theta_1,\theta_2,\theta_{3})=(c,d-1,d)$:
\begin{align*}
\begin{multlined}[t][0.85\textwidth]
 \eta^{(2n)}_{1,1}
= \frac{(2n)!(a)_{2n}(b)_{2n}}{(c)_{3n}(d)_{3n-1} }\Bigg(
 \frac{(c-a)_{n+1}(c-b)_{n+1}}{(d-c)(d-1-c) (c+n)(c+n+1)_{2n}}\\+
 \frac{(d-1-a)_{n+1}(d-1-b)_{n+1}}{(c-d+1)(d-1+n)(d+n)_{2n}}-
 \frac{(d-a)_{n+1}(d-b)_{n+1}}{(c-d)(d+n)(d+n+1)_{2n}}\Bigg).
 \end{multlined}
\end{align*}

\item For $m=1$ and $(\theta_1,\theta_2,\theta_3)=(c,c+1,d)$
\begin{align*}
\begin{multlined}[t][0.85\textwidth]
 \eta^{(2n+1)}_{1,2} = 
 \frac{c}{b} \frac{(2n+1)!(a)_{2n+1}(b)_{2n+1}}{(c)_{3n}(d)_{3n+1}}
\Bigg( 
 \frac{(c-a)_{n+1}(c-b-1)_{n+2}}{(d-c)(c+n)(c+n+1)_{2n+1}}-
 \frac{(c+1-a)_{n+1}(c-b)_{n+2}}{(d-c-1)(c+1+n)(c+n+2)_{2n+1}}\\+
 \frac{(d-a)_{n+1}(d-b-1)_{n+2}}{(c-d)(c+1-d)(d+n)(d+n+1)_{2n+1}}
 \Bigg).
\end{multlined}
\end{align*}
 \end{enumerate}
\end{rem}

\section*{Conclusions and outlook}

Since the 1930's it is known the relation between stochastic processes and orthogonal polynomials. Karlin--McGregor representation formula is a landmark in this respect. This formula gives the iterated probabilities and first passage probabilities in terms of an integral involving orthogonal polynomials determined by the measure fixed by the Markov matrix of the chain~\cite{KmcG}. Some authors have been able to go beyond the birth and death Markov chains and processes (with tridiagonal stochastic matrices)~\cite{Grunbaum1,Kovchegov}. In \cite{bfmaf} we shew how the theory of multiple orthogonal polynomials of types I and II supplies a framework to extend these connections to Markov chains beyond birth and death chains (with multidiagonal stochastic matrices) and illustrated the method for constructing random walks associated with Jacobi--Piñeiro multiple orthogonal polynomials. In this paper we have given a new interesting example of this fruitful relation, proving that the recently found hypergeometric multiple orthogonal polynomials \cite{lima_loureiro} are random walk polynomials, indeed.

This circle of ideas may be applied to other examples. If we look to mixed multiple orthogonal polynomials we will get banded stochastic matrices, with $N$ subdiagonals and~$M$ superdiagonals, describing more general Markov chains that the ones connected with non mixed multiple orthogonality. In this respect, the results on the zeros of mixed systems in~\cite{fidalgo} will allow to apply our normalization technique whenever the Jacobi matrix is nonnegative and weights are of~AT type. We are studying such mixed multiple Jacobi type orthogonal polynomials on the step line. We are also working in the discrete multiple case \cite{Arvesu_Coussment_Coussment_VanAssche}.
Also multivariate orthogonal polynomials, see~\cite{multivariate,ariznabarreta_manas,ariznabarreta_manas2} and Chapter~2 in~\cite{Ismail2} by Yuan Xu, can be associated with random walks whenever a nonnegative Jacobi matrix is provided. See also
 \cite{KmcG_multivariate}.

\section*{Acknowledgments}

We would like to thank FA~Grünbaum for inspiring conversations regarding the Karlin--McGregor representation formula and its possible extensions. We thanks JMR~Parrondo for illuminating comments and encouragement.

\end{document}